\documentclass[10pt]{article}

\usepackage[T1]{fontenc}
\usepackage[utf8]{inputenc}
\usepackage{amsmath,mathtools,amsthm}
\usepackage{newtxtext,newtxmath}
\usepackage[letterpaper,left=102pt,right=103pt,top=80pt,bottom=121pt]{geometry}
\usepackage{bm}
\usepackage{graphicx}
\usepackage{enumitem}
\usepackage{booktabs}
\usepackage{xcolor}
\usepackage[numbers,sort&compress]{natbib}
\usepackage{microtype}
\IfFileExists{placeins.sty}{\usepackage{placeins}}{\newcommand{\FloatBarrier}{}}
\usepackage[hidelinks]{hyperref}

\theoremstyle{plain}
\newtheorem{theorem}{Theorem}[section]
\newtheorem{lemma}[theorem]{Lemma}
\newtheorem{proposition}[theorem]{Proposition}
\newtheorem{corollary}[theorem]{Corollary}

\theoremstyle{definition}
\newtheorem{assumption}[theorem]{Assumption}
\newtheorem{definition}[theorem]{Definition}
\newtheorem{remark}[theorem]{Remark}
\newtheorem{example}[theorem]{Example}

\allowdisplaybreaks

\title{Nonparametric Drift Estimation for Multidimensional\\
Stochastic Differential Equations under Censoring}
\author{Yichuan Huang\\[0.35em]
\small Laboratoire Modal'X, Universit\'e Paris Nanterre, Nanterre, France\\
\small Laboratoire MAP5, Universit\'e Paris Cit\'e, Paris, France\\
\small \texttt{yhuang@parisnanterre.fr}}
\date{}

\begin{document}
\maketitle

\begin{abstract}
In many applications, a diffusion trajectory is observed without distortion
only while it remains inside a time-varying region; outside this region, only
its nearest boundary point and an indicator of visibility are recorded. We
formalize this setting through a censoring scheme for multidimensional
stochastic differential equations, in which the latent state is observed via
its Euclidean projection onto a random, time-dependent closed convex set.
Classical arguments handle the one-dimensional case. In higher dimensions,
however, the projection generates additional finite-variation terms, and it is
not immediate that the local information required for drift estimation is
preserved under censoring. We show that it is. Under a suitable regularity
condition on the projection map, allowing the censoring region to be
nonsmooth, the observed process is a continuous semimartingale and admits a
generalized Itô decomposition, whose finite-variation correction term assigns
no mass to visible times. Consequently, stochastic integrals over visible
times coincide exactly for the observed and latent processes. Building on
this identity, we construct a Nadaraya--Watson-type drift estimator with a
data-driven bandwidth selection rule and establish oracle inequalities. We
also derive anisotropic minimax lower bounds for both the full-observation and
censored experiments.
\end{abstract}



\section{Introduction} \label{sec:introduction}

We consider a multidimensional diffusion observed through a random, possibly
time-varying closed convex censoring set. The latent state is recorded exactly
while it lies in the interior of the set and is otherwise replaced by its
Euclidean projection onto the set, together with an indicator of visibility. This framework is motivated by several applications: sensor
measurements clipped at a finite operating range
\cite{chan2023restoration, lin2025saturation}; asset returns censored at
regulatory price limits \cite{xu2024tail}; and multivariate measurements subject to detection limits
\cite{perlino2025bayesian, wang2023multivariate, lachos2017finite}. The observable region may be a box, a band, or
a ball (Example~\ref{ex:censoring_regime}). We model the latent signal as a
multidimensional diffusion and the observation as its Euclidean projection
onto a random, possibly time-varying closed convex set, and study
nonparametric estimation of the drift.

Let $X^{(i)}=(X^{(i)}(t))_{t\in[0,T]}$, $i=1,\ldots,N$, be independent
trajectories satisfying the stochastic differential equation
\begin{equation}
\label{eq:SDE_X}
X^{(i)}(t)
=
x_0+\int_0^t b(X^{(i)}(s))\,ds
+\int_0^t \Sigma(X^{(i)}(s))\,dW^{(i)}(s),
\qquad x_0\in\mathbb{R}^d,\quad t\in[0,T],
\end{equation}
where $b:\mathbb{R}^d\to\mathbb{R}^d$ is the drift field,
$\Sigma:\mathbb{R}^d\to\mathbb{R}^{d\times q}$ is the dispersion matrix, and
$W^{(i)}=(W^{(i)}_1,\ldots,W^{(i)}_q)^\top$ is a $q$-dimensional standard
Brownian motion. Rather than observing each path directly, we consider a new observation scheme in which the censoring region at time $t$ is indexed by a time-varying auxiliary process
$\theta^{(i)}(t)\in\Theta\subset\mathbb{R}^{p}$. Specifically, the observed
data are
\begin{equation}
\label{eq:projected-observation}
Y^{(i)}(t)=\Pi_{D(\theta^{(i)}(t))}\!\left(X^{(i)}(t)\right),\qquad
\delta^{(i)}(t)=\mathbf{1}_{\{X^{(i)}(t)\in\mathrm{Int}(D(\theta^{(i)}(t)))\}},
\end{equation}
where $D$ assigns to each $\theta\in\Theta$ a nonempty closed convex set
$D(\theta)\subset\mathbb{R}^d$, $\Pi_A:\mathbb{R}^d\to\mathbb{R}^d$ denotes
the Euclidean projection onto a nonempty closed convex set
$A\subset\mathbb{R}^d$, and $\operatorname{Int}(\cdot)$ is the topological
interior. The latent state is either observed directly, when $\delta^{(i)}(t)=1$, or recorded through its projection onto the boundary of the censoring set. The goal of this paper is to
construct and analyze a nonparametric estimator of the drift $b$ from the
censored observations $\{(Y^{(i)},\delta^{(i)})\}_{i=1}^N$.

The starting point of our analysis is probabilistic. Under suitable geometric conditions on the censoring domains, we show that $Y$ is a continuous
semimartingale and admits a generalized It\^{o} decomposition adapted to the
possibly nonsmooth and time-dependent geometry of $D(\theta(t))$. In
particular, for every suitable predictable integrand $H$,
\begin{align}\label{paths}
\int_0^T \delta(t)H(t)\,dY(t)
=
\int_0^T \delta(t)H(t)\,dX(t).
\end{align}
This identity \eqref{paths} is the key probabilistic link between the censored
observation scheme and the subsequent drift estimation procedure. At visible times, that is, when $\delta(t)=1$, the observed increment $dY(t)$ carries exactly the same local information as the latent increment $dX(t)$.

In the SDE setting, censoring of the form $Y(t)=X(t)\wedge\theta(t)$ was studied
in Huang~\cite{huang2026censoredsde}, where \eqref{paths} is established via
Tanaka's formula (see, e.g., Revuz and
Yor~\cite[Chapter~VI, Theorem~1.2]{revuz2013continuous}), an intrinsically
one-dimensional argument. In dimension $d>1$, the projected process
$Y=\Pi_{D(\theta)}X$ acquires nontrivial finite-variation terms arising from
the projection geometry, and Tanaka's formula no longer applies directly. Nor
can the regularity of $(\theta,x)\mapsto\Pi_{D(\theta)}x$ be taken for granted,
since the projection onto a closed convex set need not even be
directionally differentiable (see
Shapiro~\cite[Introduction]{shapiro2016differentiability} and the references therein). So far, semimartingale decompositions of projected diffusions have been established only for sufficiently smooth boundaries. Gobet~\cite[Proposition~3.1]{gobet2000weak} treats
$C^3$ domains, flattening the boundary locally so that Tanaka's formula applies in the normal direction. Here, the censoring sets may be nonsmooth and time-dependent. It is therefore not immediately clear that
the information needed for drift estimation can still be recovered
from the censored path in the multidimensional case.

Identity \eqref{paths} is the path-valued analog of a classical scalar
identity underlying complete-case regression with a censored covariate. In
that setting, one observes a response $V=f(X)+\varepsilon$ with
$\mathbb{E}[\varepsilon\mid X]=0$, where $f(x):=\mathbb{E}[V\mid X=x]$ is the
regression function to be estimated; the covariate $X$ is right-censored by a
censoring variable $\theta$, so that only the censored covariate
$Z=X\wedge\theta$ and the censoring indicator
$\Delta=\mathbf{1}_{\{X\leq\theta\}}$ are observed rather than $X$ itself.
Under the conditional independence assumption $V\perp\theta\mid X$, we have
\begin{align}\label{expectation}
    f(x)=\mathbb{E}[V\mid Z=x,\,\Delta=1].
\end{align}
This identity shows
that restricting to the complete cases $\{\Delta=1\}$, on which $Z=X$,
recovers the uncensored regression target without bias; see
Efromovich~\cite[Section~3.5.2]{efromovich2026survival}, Lotspeich et
al.~\cite{lotspeich2024making}, and Wang and
Feng~\cite[Proposition~1]{wang2012multiple} for treatments in the regression setting.

Nonparametric inference for diffusion models has most often been studied in
the single-trajectory regime ($N=1$), where one path is observed over $[0,T]$,
with $T\to\infty$ under stationarity or ergodicity. For scalar diffusions,
drift estimation was developed by Kutoyants~\cite{kutoyants2004statistical},
Hoffmann~\cite{hoffmann1999adaptive}, Spokoiny~\cite{spokoiny2000adaptive},
Gobet, Hoffmann and Rei\ss~\cite{gobet2004nonparametric},
Dalalyan~\cite{dalalyan2005sharp}, Comte, Genon-Catalot and
Rozenholc~\cite{comte2007penalized}, and Comte and
Genon-Catalot~\cite{comte2021drift}. In dimension $d>1$, contributions include
Schmisser~\cite{Schmisser2013}, Strauch~\cite{Strauch2015}, Nickl and
Ray~\cite{nickl2020nonparametric}, Aeckerle-Willems and
Strauch~\cite{AeckerleWillemsStrauch2022}, and Oga and
Koike~\cite{oga2024drift}. In a second regime, closer to functional data analysis, one observes $N$ independent
paths on a fixed interval $[0,T]$ and lets $N\to\infty$; nonparametric drift
estimation in this setting was developed, for scalar diffusions, by Comte and
Genon-Catalot~\cite{comte2020nonparametric}, Denis, Dion-Blanc and
Martinez~\cite{denis2021ridge}, Marie and Rosier~\cite{marie2023nadaraya}, and
Marie~\cite{marie2025nonparametric}. The present work belongs to the fixed-$T$, $N\to\infty$ regime. Multidimensional
drift estimation on a fixed horizon as $N\to\infty$ was developed by Della Maestra
and Hoffmann~\cite{dellamaestra2022nonparametric} for McKean--Vlasov particle
systems, a framework that, as the authors note
in~\cite[Section~1.2]{dellamaestra2022nonparametric}, also covers independent
paths in the absence of interaction. Our setting differs in that the trajectories are not fully observed, but are instead recorded through the censoring scheme~\eqref{eq:projected-observation}.

Building on identity \eqref{paths}, we construct an estimator of
Nadaraya--Watson type and first establish the standard bias--variance
decomposition of its risk. We then introduce a data-driven bandwidth
selection rule in the spirit of Goldenshluger and
Lepski~\cite{goldenshluger2011bandwidth}. Since our estimator is a ratio, the bandwidths of the numerator and
denominator are selected separately, following Comte and
Marie~\cite{comte2021nadaraya}. For drift estimation, the Goldenshluger--Lepski device has been employed in the
single-trajectory $T\to\infty$ regime by Aeckerle-Willems and
Strauch~\cite{AeckerleWillemsStrauch2022}, under sup-norm risk, and in the
fixed-$T$, $N\to\infty$ regime by Della Maestra and
Hoffmann~\cite{dellamaestra2022nonparametric}, under pointwise risk. Here, the
risk is the integrated $L^2$ risk, and the bandwidth
is a $d$-vector with one component per coordinate, allowing the procedure to adapt to potentially different smoothness levels across spatial directions.

The form of the risk also determines which empirical quantities need to be controlled through concentration inequalities. The Talagrand-type inequality of Klein and
Rio~\cite{klein2005concentration}, which already underlies the analysis of~\cite{comte2021nadaraya}, remains the appropriate tool, but our summands involve stochastic integrals and are not uniformly bounded over sample paths. We show that a single Hilbert-norm truncation at a suitably chosen level suffices, thereby yielding a deviation inequality for families of Hilbert-space-valued empirical means with possibly unbounded summands. This inequality may be of independent
interest. In addition, we establish anisotropic minimax lower bounds for
multidimensional drift classes. In the one-dimensional uncensored setting, a corresponding lower bound was previously obtained by Denis, Dion-Blanc and
Martinez~\cite[Theorem 4.7]{denis2021ridge}; we extend this result to the multidimensional censored setting and to anisotropic drift classes.

In multivariate survival analysis, the problem of nonparametric estimation when $T\in\mathbb{R}^d$ is censored through the coordinatewise minimum
$T\wedge\theta$ has been extensively studied; see Dabrowska~\cite{dabrowska1988kaplan},
Lin and Ying~\cite{lin1993simple}, and Prentice and
Cai~\cite{prentice1992covariance}, with a recent survey by Janssen and
Veraverbeke~\cite{janssen2024nonparametric}. This
corresponds to the special case $D(\theta)=\prod_j(-\infty,\theta_j]$
of our projection mechanism. The analogous projection onto
$D(\theta)=\prod_j[\theta_j,\infty)$ arises in models subject to lower detection limits; see Lachos et
al.~\cite{lachos2017finite}, Wang~\cite{wang2023multivariate}, and Perlino et
al.~\cite{perlino2025bayesian}. The distinguishing feature of
\eqref{eq:projected-observation} is that both the latent object and the
censoring set are path-valued and time-dependent. Interval-censored multivariate data, where the observation specifies a
compatible set such as $T\in(L,R]$ rather than a projected value
$Y=\Pi_D(X)$, are conceptually related but structurally distinct; see Zeng et
al.~\cite{zeng2017maximum} and Du and Yu~\cite{du2024regression}. Both forms of censoring can be viewed as
partial localization of a latent variable within the coarse-data framework of
Heitjan and Rubin~\cite{heitjan1991ignorability}.

The paper is organized as follows.
Section~\ref{sec:model-preliminaries} introduces the projection-censoring
framework and establishes the identity for uncensored increments.
Section~\ref{sec:kernel-estimator} constructs the estimator and provides
fixed-bandwidth and adaptive risk bounds. Section~\ref{sec:minimax-rate}
establishes the minimax lower bound in the full-observation experiment and
deduces its censored-observation counterpart.
Section~\ref{sec:simulation} reports the simulation study. Proofs,
implementation details, and auxiliary results are collected in the
appendices.

\vspace{5 mm}
\textit{Notation.} For \(\boldsymbol h=(h_1,\ldots,h_d)\in(0,\infty)^d\), define
\(x/\boldsymbol h:=(x_1/h_1,\ldots,x_d/h_d)^\top\). For a kernel \(K:\mathbb R^d\to\mathbb R\), set
\(K_{\boldsymbol h}(x):=(\prod_{j=1}^d h_j)^{-1}K(x/\boldsymbol h)\). We write \(\mathbb N=\{1,2,\ldots\}\) and
\(\mathbb N_0=\{0,1,2,\ldots\}\). Vectors are column vectors. For
\(x,y\in\mathbb R^m\), we use
\(\langle x,y\rangle_m:=x^\top y\) and
\(|x|_m:=\langle x,x\rangle_m^{1/2}\). When the ambient dimension is
clear, we simply write \(\langle x,y\rangle\) and \(|x|\). The canonical
basis of \(\mathbb R^m\) is denoted by \((e_1,\ldots,e_m)\). 

Let \(d_1,d_2\in\mathbb N\). For a measurable set
\(A\subset\mathbb R^{d_1}\) and a measurable function
\(u:A\to\mathbb R^{d_2}\), set
\(\|u\|_{p,A}:=(\int_A |u(x)|_{d_2}^p\,dx)^{1/p}\) for
\(1\le p<\infty\), and
\(\|u\|_{\infty,A}:=\operatorname*{ess\,sup}_{x\in A}|u(x)|_{d_2}\). If
\(A=\mathbb R^{d_1}\), we simply write \(\|u\|_p\) and
\(\|u\|_\infty\). For \(p=2\), we also write
\(\langle u,v\rangle_{2,A}:=\int_A\langle u(x),v(x)\rangle_{d_2}\,dx\),
and, if \(A=\mathbb R^{d_1}\), \(\langle u,v\rangle_2\). Given a
nonnegative weight function \(\mu\), define
\(\|u\|_{p,\mu,A}:=(\int_A |u(x)|_{d_2}^p\mu(x)\,dx)^{1/p}\). If
\(A=\mathbb R^{d_1}\), we write \(\|u\|_{p,\mu}\).

\section{Model and assumptions}
\label{sec:model-preliminaries}

This section fixes the probabilistic framework used throughout the paper.
Using a generalized Itô formula, we first derive a semimartingale identity showing that $dY$ coincides with
$dX$ at uncensored times. We then identify \(\overline f\) as
the time-averaged occupation density of the uncensored observations; this is
the target of the estimator \(\widehat f\) defined in
Section~\ref{subsec:nw-estimator}. Its positivity on the estimation set is
proved under a simple visibility condition, which allows us to analyze the
ratio estimator \(\widehat{bf}/\widehat f\) in
Sections~\ref{subsec:nw-estimator} and~\ref{subsec:gl-selection}.

\subsection{Semimartingale decomposition of the censored process}
\label{subsec:projected-observation}

Let \(X=(X(t))_{t\in[0,T]}\) be an \(\mathbb R^d\)-valued continuous
semimartingale, and let \(\theta=(\theta(t))_{t\in[0,T]}\) be a \(\Theta\)-valued
continuous semimartingale, where \(\Theta\subset\mathbb R^{p}\) is open.
For each \(\theta\in\Theta\), let \(D(\theta)\subset\mathbb R^d\) be nonempty, closed,
and convex, and let \(\Pi_{D(\theta)}\) denote the Euclidean projection onto \(D(\theta)\).
We use the notation
\begin{align*}
        Z(t)&=(\theta(t),X(t)), &
        F(z)=F(\theta,x)&=\Pi_{D(\theta)}x, &
        Y(t)&=F(Z(t)),\\
        \mathcal U_D
        &:=\{(\theta,x)\in\Theta\times\mathbb R^d:
        x\in\operatorname{Int}D(\theta)\}, &
        \delta(t)
        &:=\mathbf 1_{\{Z(t)\in\mathcal U_D\}} .
\end{align*}

Equivalently,
\(\delta(t)=\mathbf 1_{\{X(t)\in\operatorname{Int}D(\theta(t))\}}\).
In dimension \(d=1\), if \(D(\theta)=(-\infty,\theta]\), then
\(Y(t)=X(t)\wedge \theta(t)\) and
\(\delta(t)=\mathbf 1_{\{X(t)<\theta(t)\}}\), which is the right-censoring
scheme considered in Huang~\cite{huang2026censoredsde}. In the cases considered below, \(\mathcal U_D\) is Borel and, in fact, open.
Since \(Z\) is continuous and adapted, it is predictable; hence
\(\delta=\mathbf 1_{\{Z\in\mathcal U_D\}}\) is also predictable. Moreover, it is bounded by 1; thus, the stochastic
integrals below are well defined. Throughout this article, inequalities between vectors are understood componentwise. We begin with two elementary definitions from optimization theory.

\begin{definition}
A set \(P\subset\mathbb R^\ell\) is called a polyhedron if there exist \(r\in\mathbb N\), \(\Gamma\in\mathbb R^{r\times \ell}\), and \(\gamma\in\mathbb R^r\) such that \(P=\{u\in\mathbb R^\ell:\Gamma u\le\gamma\}\). For an optimization problem \(\min_u \phi(u)\) subject to \(h(u)\le0\), its feasible set is \(\{u:h(u)\le0\}\); the problem is feasible if this set is nonempty, and any point in it is called feasible.
\end{definition}

We now impose a polyhedral structure on the moving censoring sets so that the censoring problem can be treated using existing quadratic programming theory.

\begin{assumption}
\label{ass:proj_decomp}
There exist \(m\ge1\), \(M\in\mathbb R^{m\times d}\),
\(c_0\in\mathbb R^m\), and \(\mathsf B_D\in\mathbb R^{m\times p}\) such that, for every \(\theta\in\Theta\), \(D(\theta)\) is the polyhedron
\begin{align*}
    D(\theta)&=\{y\in\mathbb R^d:My\le c(\theta)\}
    = \bigcap_{j=1}^m \{y\in\mathbb R^d:m_j^\top y\le c_j(\theta)\},
    & c(\theta)&:=c_0+\mathsf B_D\theta. \tag{Q}\label{eq:Q}
\end{align*}
Here \(m_j^\top\) denotes the \(j\)-th row of \(M\). Moreover, we assume that
\begin{align}
&\forall \theta\in\Theta,\quad \exists y_\theta\in\mathbb R^d
\text{ such that } My_\theta<c(\theta).
\tag{S}\label{eq:Slater}
\end{align}
\end{assumption}

Under Assumption~\ref{ass:proj_decomp}, for every $\theta\in \Theta$, $\Pi_{D(\theta)} x$ is the unique solution to the following linearly constrained quadratic program:
\begin{align} \label{optimization_problem}
    \min _{y \in \mathbb{R}^d} \frac{1}{2}|y-x|^2 \quad \text { such that } \quad M y \leq c_0+\mathsf B_D\theta .
\end{align}
Moreover,
\(
\mathcal U_D
=\{(\theta,x)\in\Theta\times\mathbb R^d:Mx<c(\theta)\}
=\bigcap_{j=1}^m\{(\theta,x):m_j^\top x<c_j(\theta)\},
\)
which is open because it is a finite intersection of open sets. 
\begin{remark}
    \begin{enumerate}
    [label=\textup{(\roman*)}, leftmargin=*, itemsep=0.35em]
    \item The affine form \(c(\theta)=c_0+\mathsf B_D\theta\) in condition~\eqref{eq:Q} makes the
feasible set in~\eqref{optimization_problem} depend affinely on the joint parameter
\((\theta,x)\), which is used in Proposition~\ref{prop:projection-DC} to show
that the projection map \((\theta,x)\longmapsto \Pi_{D(\theta)}x\) is jointly piecewise affine in \((\theta,x)\). Smooth
nonlinear dependence on \(\theta\) may still lead to a locally regular projection map
under additional regularity assumptions, but
it would not follow from the same argument used here; see Still~\cite[Chapter 6]{still2018lectures} and the references therein.

        \item Condition~\eqref{eq:Slater} guarantees not only that the problem \eqref{optimization_problem} is feasible, but also that 
\[
        \operatorname{Int}D(\theta)
        =
        \bigcap_{j=1}^m
        \operatorname{Int}\{y:m_j^\top y\le c_j(\theta)\}
        =
        \bigcap_{j=1}^m
        \{y:m_j^\top y<c_j(\theta)\}
        =
        \{y:My<c(\theta)\}
        \ne \emptyset .
\]
In the above display, the only nontrivial step is the second identity, and it is
guaranteed by condition~\eqref{eq:Slater}. Indeed, if \(m_j\ne \mathbf 0\), where
\(\mathbf 0\) denotes the zero vector, then
\(\{y:m_j^\top y\le c_j(\theta)\}\) is a half-space and
\(        \operatorname{Int}\{y:m_j^\top y\le c_j(\theta)\}
        =
        \{y:m_j^\top y<c_j(\theta)\}.
\)
If \(m_j=\mathbf 0\), condition~\eqref{eq:Slater} forces \(c_j(\theta)>0\), so
\(
        \{y:\mathbf 0^\top y\le c_j(\theta)\}
        =
        \mathbb R^d
        =
        \{y:\mathbf 0^\top y<c_j(\theta)\}.
\)
Notice that, without condition~\eqref{eq:Slater}, the latter identity may fail. For instance,
if \(m_j=\mathbf 0\) and \(c_j(\theta)=0\), then
\(\{y:\mathbf 0^\top y\le c_j(\theta)\}=\mathbb R^d\), whereas
\(\{y:\mathbf 0^\top y<c_j(\theta)\}=\emptyset\). Moreover, replacing \(My_\theta<c(\theta)\) in condition~\eqref{eq:Slater} with the weaker
condition \(My_\theta\le c(\theta)\) would only guarantee that \(D(\theta)\) is nonempty; its
interior may still be empty, for example, when
\(D(\theta)=\{y\in\mathbb R:y\le0,\ -y\le0\}=\{0\}\).

\end{enumerate}
\end{remark}

We will use the following standard fact from multi-parametric quadratic programming.

\begin{lemma}
\label{lem:mpqp-pwa}
Let \(\Xi\subset\mathbb R^s\) be a nonempty polyhedron, let \(\mathsf H=\mathsf H^\top\succ0\), \(\mathsf A\in\mathbb R^{r\times \ell}\), \(w\in\mathbb R^r\), and \(S\in\mathbb R^{r\times s}\). For \(\xi\in\Xi\), consider
\[
v^\star(\xi)\in\arg\min_{v\in\mathbb R^\ell}\frac12 v^\top \mathsf H v
\quad\text{subject to}\quad \mathsf A v\le w+S\xi .
\]
If the feasible set \(\{v\in\mathbb R^\ell:\mathsf A v\le w+S\xi\}\) is nonempty for every \(\xi\in\Xi\), then \(v^\star(\xi)\) is unique for every \(\xi\in\Xi\), and the map \(\xi\mapsto v^\star(\xi)\) is continuous and piecewise affine on \(\Xi\).
\end{lemma}

This is a classical result in multi-parametric quadratic programming; see Bemporad et al.~\cite[Section~4]{bemporad2002explicit}, and also T{\o}ndel et al.~\cite[Theorem~1]{tondel2001algorithm}. Using Lemma~\ref{lem:mpqp-pwa}, we deduce a semimartingale decomposition of $Y$.

\begin{proposition}
\label{prop:projection-DC}
Under Assumption~\ref{ass:proj_decomp}, the following statements hold.
\begin{enumerate}
[label=\textup{(\roman*)}, leftmargin=*, itemsep=0.35em]
    \item For each \(j=1,\ldots,d\), the coordinate map
\(
F_j(\theta,x):=\big(\Pi_{D(\theta)}x\big)_j
\)
is locally difference-of-convex (DC) on \(\Theta\times\mathbb R^d\).
That is, every point of \(\Theta\times\mathbb R^d\) has an open convex
neighborhood on which \(F_j\) can be written as the difference of two convex
functions.
\item Let \(O:=\Theta\times\mathbb R^d\). For each \(j\), let
\((U_{j,n})_{n\ge1}\) be a countable open cover of \(O\) by open convex sets,
and let \(G_{j,n},H_{j,n}\) be convex functions on \(U_{j,n}\) such that
\(F_j=G_{j,n}-H_{j,n}\) on \(U_{j,n}\). Set \(C_{j,1}:=U_{j,1}\) and
\(C_{j,n}:=U_{j,n}\setminus\bigcup_{k<n}U_{j,k}\) for \(n\ge2\). Let
\(G^*_{j,n}:U_{j,n}\to\mathbb R^{p+d}\) and \(H^*_{j,n}:U_{j,n}\to\mathbb R^{p+d}\) be measurable functions such that \(G^*_{j,n}(z)\in\partial G_{j,n}(z)\) and \(H^*_{j,n}(z)\in\partial H_{j,n}(z)\), where \(\partial\) denotes the convex subdifferential. Define
\[
F_j^*(z):=\sum_{n\ge1}\mathbf 1_{C_{j,n}}(z)
\big(G^*_{j,n}(z)-H^*_{j,n}(z)\big),\qquad z\in O.
\]
Here \(F_j^*(z)=(F^*_{j,1}(z),\ldots,F^*_{j,p+d}(z))^\top\).
Then \(Y_j=F_j(Z)\) is a continuous semimartingale and admits the
decomposition
\begin{equation}
\label{eq:Y-semimartingale-decomp}
Y_j(t)=Y_j(0)+\sum_{r=1}^{p+d}\int_0^t F^*_{j,r}(Z(s))\,dZ_r(s)+A_j(t),
\qquad t\in[0,T],
\end{equation}
where \(A_j\) is a continuous finite-variation process with \(A_j(0)=0\).
\end{enumerate}

\end{proposition}

The proof of Proposition~\ref{prop:projection-DC} is given in
Appendix~\ref{Sec:pf-projection-DC}. The first assertion is not a
consequence of the closed convexity of \(D(\theta)\subset\mathbb R^d\) alone.
Indeed, locally DC functions possess one-sided directional derivatives,
whereas it is known that the projection onto a fixed closed convex set
need not be directionally differentiable in general; see
Shapiro~\cite[Introduction]{shapiro2016differentiability} and the references therein. Thus,
Assumption~\ref{ass:proj_decomp} should be understood as a nontrivial
geometric regularity assumption. The second assertion is a generalized Itô formula for locally DC
functions. Its validity in the globally DC case is guaranteed by
Bouleau~\cite[Theorem~3]{bouleau1984formules}, and the extension to
locally DC functions is explained in the discussion following that result.

\begin{theorem}
\label{thm:semimartingale-identity}
Under Assumption~\ref{ass:proj_decomp}, for each \(j=1,\ldots,d\),
in the decomposition \eqref{eq:Y-semimartingale-decomp} of
Proposition~\ref{prop:projection-DC}, the following assertions hold.
\begin{enumerate}[label=\textup{(\roman*)}, leftmargin=*, itemsep=0.35em]
\item The functions \(F_j^*\) and the finite-variation terms \(A_j\) may be
chosen so that
\begin{align}
&F^*_{j,r}(\theta,x)
=
\begin{cases}
0, & 1\le r\le p,\\
\mathbf 1_{\{r=p+j\}}, & p+1\le r\le p+d,
\end{cases}
\qquad (\theta,x)\in\mathcal U_{D},
\label{eq:Fstar-on-U}
\\&
\int_0^T \delta(t)\,d|A_j|(t) =0
\label{eq:A-vanishes-on-U}
\end{align}

\item Let \(\Phi=(\Phi(t))_{t\in[0,T]}\) be a real-valued predictable process such
that the stochastic integrals below are well defined. Then, for each
\(j=1,\ldots,d\),
\begin{equation}
\label{eq:pathwise-identity}
\int_0^T \delta(t)\Phi(t)\,dY_j(t)
=
\int_0^T \delta(t)\Phi(t)\,dX_j(t).
\end{equation}
\end{enumerate}
\end{theorem}

The proof of Theorem~\ref{thm:semimartingale-identity} is given in
Appendix~\ref{Sec:pf-lem:pathwise}. The identity \eqref{eq:A-vanishes-on-U} also has a useful interpretation.
For almost every $\omega$, the total-variation measure $ d|A_j|(\omega, \cdot)$ assigns zero mass to the set of times at which $Z(t, \omega) \in \mathcal{U}_D$:
\[
 d|A_j|(\omega,\{t\in[0,T]:Z(t,\omega)\in\mathcal U_D\})=0 .
\]
In particular, the finite-variation correction term in $dY$ generated by the projection does not contribute when weighted by \(\delta(t)\), which yields the identity \eqref{eq:pathwise-identity}. Unlike the classical complete-case identity \eqref{expectation} for censored covariates recalled in Section~\ref{sec:introduction}, \eqref{eq:pathwise-identity} is a pathwise identity and requires no conditional-independence assumption between the latent process and the censoring mechanism. This identity is the key ingredient in the subsequent
analysis of our estimator. We next give several projection maps with
locally DC coordinates:
\(
(\theta,x)\longmapsto \Pi_{D(\theta)}x
\).

\begin{example}
\label{ex:censoring_regime}
\begin{enumerate}[label=(\alph*), leftmargin=*, itemsep=0.35em]

\item \emph{Coordinatewise right censoring.}
Let \(\Theta=\mathbb R^d\) and
\(
D(\theta)=\prod_{j=1}^d(-\infty,\theta_j].
\)
Then
\[
F_j(\theta,x)=x_j\wedge \theta_j=x_j-(x_j-\theta_j)^+,
\qquad
\mathcal U_{D}=\{(\theta,x):x_j<\theta_j,\ j=1,\ldots,d\}.
\]

\item \emph{Random band.}
Fix a unit vector \(n\in\mathbb R^d\). Let
\(\Theta=\{(L,U)\in\mathbb R^2:L<U\}\) and
\(
D(L,U)=\{x\in\mathbb R^d:L\le \langle x,n\rangle\le U\}.
\)
Then
\[
F(L,U,x)
=
 x+(L-\langle x,n\rangle)^+n
 -(\langle x,n\rangle-U)^+n,
\qquad
\mathcal U_{D}=\{(L,U,x):L<\langle x,n\rangle<U\}.
\]

\item \emph{Random ball.}
Let \(\Theta=(0,\infty)\) and \(D(r)=\overline B(0,r)\). Then
\[
F(r,x)=x\mathbf 1_{\{|x|\le r\}}
+r\frac{x}{|x|}\mathbf 1_{\{|x|>r\}},
\qquad
\mathcal U_{D}=\{(r,x):|x|<r\}.
\]

\end{enumerate}
\end{example} 

\begin{remark}
\label{rem:scope-proj-decomp}
We now explain how the preceding examples relate to
Assumption~\ref{ass:proj_decomp}.

\begin{enumerate}
[label=\textup{(\roman*)}, leftmargin=*, itemsep=0.35em]
\item 
Example~\ref{ex:censoring_regime}(a) satisfies Assumption~\ref{ass:proj_decomp}. Indeed,
\(D(\theta)=\{y:y_j\le \theta_j,\ j=1,\dots,d\}\), so one may take \(M=I_d\) and
\(c(\theta)=\theta\). Condition~\eqref{eq:Slater} holds, for example, with
\(y_\theta=\theta-\mathbf 1_d\).

Example~\ref{ex:censoring_regime}(b) is also covered. In this case
\(M=(-n^\top,n^\top)^\top\) and
\(c(L,U)=(-L,U)^\top\). Condition~\eqref{eq:Slater} holds by taking
\(y_{L,U}=((L+U)/2)n\).

\item 
Example~\ref{ex:censoring_regime}(c) is not covered by
Assumption~\ref{ass:proj_decomp}. Nevertheless, it
satisfies the locally DC condition required in Proposition~\ref{prop:projection-DC}.
Indeed,
\[
\Pi_{\overline B(0,r)}x
=
x-(|x|-r)^+\mathbf 1_{\{x\ne0\}}\frac{x}{|x|}.
\]
Near any point $(r_0,0)$ with $r_0>0$, the projection coincides with
the identity on a neighborhood. Away from $x=0$, each coordinate of
the map $x\mapsto x/|x|$ is smooth and hence locally DC, whereas
$(r,x)\mapsto (|x|-r)^+$ is convex. Since the product of two DC
functions on a finite-dimensional open convex set is again DC
\cite[Corollary following Theorem~(II), p.~708]{Hartman1959},
it follows, by applying this result locally, that each coordinate of
$(r,x)\mapsto\Pi_{\overline B(0,r)}x$ is locally DC. Since
Assumption~\ref{ass:proj_decomp} is used in Proposition~\ref{prop:projection-DC}
only to obtain this locally DC property, Proposition~\ref{prop:projection-DC} and Theorem~\ref{thm:semimartingale-identity} therefore remain valid for random balls.

\item The literature on killed diffusions studies processes sent to a cemetery
state either upon first hitting the boundary of a fixed domain or at a random
time governed by a state-dependent killing rate; see, e.g., Casella and
Roberts~\cite{casella2008exact}, Del Moral and
Villemonais~\cite{delmoral2018exponential}, and Nickl and
Seizilles~\cite{nickl2025inferring}. In the context of weak Euler approximation for such processes, Gobet~\cite[Proposition~3.1]{gobet2000weak} studies the projected
process \(Y(t)=\Pi_{D}X(t)\) and derives its semimartingale
decomposition for \(C^3\) domains. The special case of
Example~\ref{ex:censoring_regime}(c) with constant \(r>0\), for which \(D\)
is a fixed ball, falls within this setting. Locally, a \(C^3\) change of
coordinates flattens the boundary into a half-space, allowing the
one-dimensional Tanaka formula to be applied in the normal direction. By
contrast, we use a multidimensional generalized It\^{o} formula and only
require the projection map to be locally DC, which also covers the nonsmooth
censoring sets in Examples~\ref{ex:censoring_regime}(a) and
\ref{ex:censoring_regime}(b). The resulting finite-variation term in
Proposition~\ref{prop:projection-DC} is, however, more abstract than Gobet's
boundary term and is not generally identified with boundary local time.
\end{enumerate}
\end{remark}

\subsection{SDE structure and positivity of the uncensored occupation density}
\label{subsec:model-assumptions}

We now return to the statistical diffusion model introduced in the
Introduction. We work on a filtered probability space
\((\Omega,\mathcal F,\{\mathcal F(t)\}_{t\ge0},\mathbb P)\) satisfying the
usual conditions. Let $W^{(1)}, \ldots, W^{(N)}$ be mutually independent $q$-dimensional standard Brownian motions. For each $i$, let $X^{(i)}$ denote the unique strong solution to \eqref{eq:SDE_X} driven by $W^{(i)}$.
The auxiliary processes \(\theta^{(i)}\) take values in \(\Theta\) and generate
the censored observations through the projection mechanism of
Section~\ref{subsec:projected-observation}. When no superscript is displayed,
\(X,\theta,Y,\delta\) denote a generic copy.

\begin{assumption}
\label{ass:standard_X}
The drift \(b\) is globally Lipschitz and of linear growth. The matrix
\(\Sigma\) is bounded, globally Lipschitz, and uniformly elliptic. There
exist constants \(0<\lambda_\Sigma\le \Lambda_\Sigma<\infty\) such that
\[
\lambda_\Sigma |\xi|^2 \le \xi^\top \Sigma(x)\Sigma(x)^\top \xi
\le \Lambda_\Sigma |\xi|^2,
\qquad x,\xi\in\mathbb R^d.
\]
Moreover, \(\Sigma\in C_b^{1+\alpha_\Sigma}(\mathbb R^d)\) for some
\(\alpha_\Sigma\in(0,1)\).
\end{assumption}

The Lipschitz and growth conditions ensure the existence of a unique strong solution to~\eqref{eq:SDE_X}; since the Brownian motions
\(W^{(i)}\) are mutually independent and \(x_0\) is deterministic, the
solutions \(X^{(i)}\) are i.i.d. The additional conditions on \(\Sigma\) allow us to apply the following density estimates, adapted from Menozzi, Pesce, and Zhang~\cite[Theorem~1.2]{menozzi2021density}.

\begin{proposition}
\label{ass:gaussian-density}
Under Assumption~\ref{ass:standard_X}, the SDE \eqref{eq:SDE_X} admits a
unique strong Markov solution, and for every \(t>0\), the transition
kernel \(P_t(x_0,dy)\) admits a density \(p_t(x_0,\cdot)\) with respect
to Lebesgue measure. Moreover, for every compact set
\(\mathcal C\subset\mathbb R^d\) and every interval \([a,b]\subset(0,T]\),
\[
0<
\inf_{t\in[a,b]}\inf_{x\in \mathcal C}p_t(x_0,x).
\]
In addition, there exist constants \(C_p,\mathfrak m_p, C_\nabla,\mathfrak m_\nabla>0\) such that, for
every \(t\in(0,T]\) and \(x\in\mathbb R^d\),
\begin{align*}
   &p_t(x_0,x)
\le
C_p t^{-d/2}
\exp\!\left(-\mathfrak m_p\frac{|x-x_0|^2}{t}\right), \qquad |\nabla_x p_t(x_0,x)|
\le
C_\nabla t^{-(d+1)/2}
\exp\!\left(-\mathfrak m_\nabla\frac{|x-x_0|^2}{t}\right).
\end{align*}
\end{proposition}

Throughout, we set
\(
\mathcal F(t)
:=\bigcap_{u>t}\left\{\sigma\bigl(W^{(i)}(s),\theta^{(i)}(s):
0\le s\le u,\ 1\le i\le N\bigr)\vee\mathcal N\right\},
\, t\ge0,
\)
where \(\mathcal N\) denotes the \(\mathbb P\)-null sets. We now impose conditions on the censoring mechanism and fix a compact estimation set $I \subset \mathbb R^d$. 

\begin{assumption}
\label{ass:censoring-nondegeneracy}
The process \(\theta^{(i)}=(\theta^{(i)}(t))_{t\in[0,T]}\) is a continuous semimartingale with respect to \(\{\mathcal F(t)\}_{t\ge0}\) and satisfies the following conditions.
\begin{enumerate}[label=\textup{(C\arabic*)}, leftmargin=*]
\item The pairs
\(\{(W^{(i)},\theta^{(i)})\}_{i=1}^N\) are independent and identically
distributed, and \(W^{(i)}\) and \(\theta^{(i)}\) are independent for every
\(i\).
\item There exists a deterministic time \(t_\star\in[0,T]\) such that
\(
\mathbb P\bigl(I\subset\operatorname{Int}(D(\theta(t_\star)))\bigr)>0.
\)
\end{enumerate}
\end{assumption}

By condition~(C1), each \(W^{(i)}\) is a Brownian motion with respect to \(\{\mathcal F(t)\}_{t\ge0}\). Since \(X^{(i)}\) is a strong solution driven by \(W^{(i)}\), \(X^{(i)}\) and \(\theta^{(i)}\) are also independent. For $t \in[0, T]$ and $x \in \mathbb{R}^d$, define the visibility probability
\begin{equation}
\label{eq:uncensoring-probability}
\overline G_t(x)
:=
\mathbb P\left[x\in\operatorname{Int}(D(\theta(t)))\right]
=
\mathbb E\!\left[
\mathbf 1_{\{x\in\operatorname{Int}(D(\theta(t)))\}}
\right].
\end{equation}
By the independence of \(X\) and \(\theta\), for every bounded measurable
\(\varphi\),
\begin{equation}
\label{eq:uncensored-moment-identity}
\mathbb E[\delta(t)\varphi(Y(t))]
=
\mathbb E[\overline G_t(X(t))\varphi(X(t))]
=
\int_{\mathbb R^d}\varphi(x)\overline G_t(x)p_t(x_0,x)\,dx .
\end{equation}
Equivalently, the sub-probability measure
\(B\mapsto\mathbb P(Y(t)\in B,\delta(t)=1)\) has Lebesgue density
\(
p_t^{\mathrm{unc}}(x):=\overline G_t(x)p_t(x_0,x).
\)
Thus, for \(t_0\in(0,T)\),
\begin{equation}
\label{def:uncensoring-prob}
\overline f(x)
:=
\frac{1}{T-t_0}
\int_{t_0}^T
p_t^{\mathrm{unc}}(x)\,dt
=
\frac{1}{T-t_0}
\int_{t_0}^T
\overline G_t(x)p_t(x_0,x)\,dt
\end{equation}
is the time-averaged occupation density of the uncensored observations. The following result shows that if \(t_0\) is chosen small enough, then the
uncensored occupation density is uniformly positive on \(I\).

\begin{proposition}
\label{prop:positive-denominator}
Under Assumptions~\ref{ass:proj_decomp}, \ref{ass:standard_X}, and
\ref{ass:censoring-nondegeneracy}, there exists \(t_I\in(0,T]\) such that, for
every \(t_0\in(0,t_I)\),
\begin{equation}
\label{eq:mI-positive}
m_I:=\inf_{x\in I}\overline f(x)>0 .
\end{equation}
\end{proposition}

The proof is given in Appendix~\ref{Sec:pf-positive-denominator}. In the sequel, all estimators are constructed from the censored observations restricted to the time interval \([t_0,T]\). The condition \(t_0<t_I\) ensures that the interval of integration contains a nondegenerate subinterval on which \(\inf_{x\in I}\overline G_t(x)>0\). The restriction \(t_0>0\) avoids the singular behavior of the transition
density at the initial time. See also
Marie and Rosier~\cite[Remark~2]{marie2023nadaraya}.

\section{Multivariate kernel estimator for the drift function}
\label{sec:kernel-estimator}

We now construct a Nadaraya--Watson-type estimator for the drift function. The section is
organized as follows. We first introduce anisotropic kernels and the
smoothness classes used to state the rates. We then define the kernel
estimator for the drift and provide fixed-bandwidth risk bounds for its
denominator, numerator, and ratio. Finally, we describe an
adaptive selection rule for the two anisotropic bandwidths.

\subsection{Anisotropic kernels and smoothness classes}
\label{subsec:kernels-smoothness}

The multivariate structure of the problem naturally calls for anisotropic smoothing, as the target functions may exhibit different smoothness levels across coordinates. We therefore work
with vector bandwidths
\(\boldsymbol h=(h_1,\ldots,h_d)\in(0,1]^d\).

\begin{assumption}
\label{ass:kernel}
The kernel \(K:\mathbb R^d\to\mathbb R\) is of product form
\(K(x)=\prod_{j=1}^d k_j(x_j)\), \(x\in\mathbb R^d\),
where each one-dimensional kernel \(k_j:\mathbb R\to\mathbb R\)
satisfies:
\begin{enumerate}[label=(\roman*), leftmargin=*]
\item \(k_j\) is Borel measurable, \(k_j\in L^1(\mathbb R)\cap L^2(\mathbb R)\cap L^\infty(\mathbb R)\), and \(k_j\) is even;
\item \(\int_{\mathbb R}k_j(u)\,du=1\), and there exists an integer
\(\nu_j\ge1\) such that
\(\int_{\mathbb R}u^\ell k_j(u)\,du=0\) for \(\ell=1,\ldots,\nu_j\),
and \(\int_{\mathbb R}(1+|u|^{\nu_j+1})|k_j(u)|\,du<\infty\).
\end{enumerate}
\end{assumption}

We shall state the smoothness assumptions in an anisotropic Nikol'skii form.
For a function \(u:\mathbb R^{d_1}\to\mathbb R^{d_2}\), \(j=1,\ldots,d_1\),
and an integer \(r\ge0\), \(\partial_j^r u\) denotes the \(r\)-th weak
partial derivative of \(u\) in the \(j\)-th coordinate, with
\(\partial_j^0u=u\).

\begin{definition}
\label{def:nikolskii-class}
Let \(d_1,d_2\in\mathbb N\),
\(\boldsymbol\alpha=(\alpha_1,\ldots,\alpha_{d_1})
\in(0,\infty)^{d_1}\), and \(L>0\). For each \(j\), set
\(r_j:=\lceil\alpha_j\rceil-1\) and \(\vartheta_j:=\alpha_j-r_j\in(0,1]\).
We say that \(u\in\mathcal N_{d_1}^{d_2}(\boldsymbol\alpha,L)\) if
\(u\in L^2(\mathbb R^{d_1};\mathbb R^{d_2})\), the weak derivatives
\(\partial_j^\ell u\in L^2(\mathbb R^{d_1};\mathbb R^{d_2})\) exist for
\(\ell=1,\ldots,r_j\), and for every \(z\in\mathbb R\),
\[
\big\|
\partial_j^{r_j}u(\cdot+ze_j)-\partial_j^{r_j}u
\big\|_{2}
\le
L|z|^{\vartheta_j},
\qquad j=1,\ldots,d_1 .
\]
\end{definition}

The following standard approximation result translates Nikol'skii smoothness into the kernel bias bounds used below.

\begin{lemma}
\label{lem:nikolskii-kernel-bias}
Assume that \(K\) satisfies Assumption~\ref{ass:kernel}(ii). Let
\(u\in\mathcal N_d^m(\boldsymbol\alpha,L)\), and assume that
\(\nu_j\ge \lceil\alpha_j\rceil-1\) for every \(j=1,\ldots,d\). Then
there exists a finite constant \(\mathrm C\), depending only on
\(d\), \(K\), and \(\boldsymbol\alpha\), such that for all
\(\boldsymbol h\in(0,1]^d\),
\[
\|K_{\boldsymbol h}*u-u\|_{2}^2
\le
\mathrm C L^2\sum_{j=1}^d h_j^{2\alpha_j}.
\]
\end{lemma}

The proof is standard and is omitted. It follows from the usual anisotropic
kernel approximation argument for Nikol'skii--Besov classes; see, for instance,
Kerkyacharian, Lepski and Picard~\cite[Proposition~3]{kerkyacharian2001nonlinear}. In the sequel, we assume that there exist \(\boldsymbol\beta,\boldsymbol\gamma\in(0,\infty)^d\) such that
\(
\overline f\in\mathcal N_d^1(\boldsymbol\beta,L_f),
\,
b\overline f\in\mathcal N_d^d(\boldsymbol\gamma,L_{bf}).
\)
\subsection{The Nadaraya--Watson estimator}
\label{subsec:nw-estimator}

We now define the two kernel estimators that enter the Nadaraya--Watson
ratio estimator. Let $I \subset \mathbb R^d$ be the compact estimation set introduced in Section~\ref{sec:model-preliminaries}. By Proposition~\ref{prop:positive-denominator},
\(m_I:=\inf_{x\in I}\overline f(x)>0\). For \(\boldsymbol h'\in(0,1]^d\), define
{\small
\begin{equation}
\label{eq:def_f_hat}
\widehat f_{N,\boldsymbol h'}(x)
:=
\frac{1}{N(T-t_0)}
\sum_{i=1}^N
\int_{t_0}^T
\delta^{(i)}(t)
K_{\boldsymbol h'}(Y^{(i)}(t)-x)\,dt.
\end{equation}
}
Using \eqref{eq:uncensored-moment-identity} and Fubini's theorem, we get
{\small
\begin{align}
\label{eq:expectation-fhat}
\mathbb E\!\left[\widehat f_{N,\boldsymbol h'}(x)\right]
&=
\frac{1}{T-t_0}
\int_{t_0}^T
\mathbb E\!\left[
\delta(t)K_{\boldsymbol h'}(Y(t)-x)
\right]\,dt \nonumber=
\frac{1}{T-t_0}
\int_{t_0}^T
\int_{\mathbb R^d}
K_{\boldsymbol h'}(y-x)\overline G_t(y)p_t(x_0,y)\,dy\,dt \nonumber\\
&=
\int_{\mathbb R^d}
K_{\boldsymbol h'}(y-x)\overline f(y)\,dy
=
(K_{\boldsymbol h'}*\overline f)(x).
\end{align}
}
For \(\boldsymbol h\in(0,1]^d\), define the vector-valued estimator
\(\widehat{bf}_{N,\boldsymbol h}
=(\widehat{bf}_{1,N,\boldsymbol h},\ldots,
\widehat{bf}_{d,N,\boldsymbol h})^\top\) whose $j$-th component is
{\small
\begin{align}
\widehat{bf}_{j,N,\boldsymbol h}(x)
&:=
\frac{1}{N(T-t_0)}
\sum_{i=1}^N
\int_{t_0}^T
\delta^{(i)}(t)
K_{\boldsymbol h}(Y^{(i)}(t)-x)\,dY_j^{(i)}(t) \nonumber\\
&=
\frac{1}{N(T-t_0)}
\sum_{i=1}^N
\int_{t_0}^T
\delta^{(i)}(t)
K_{\boldsymbol h}(X^{(i)}(t)-x)\,dX_j^{(i)}(t) \nonumber\\
&=
\frac{1}{N(T-t_0)}
\sum_{i=1}^N
\int_{t_0}^T
\delta^{(i)}(t)K_{\boldsymbol h}(X^{(i)}(t)-x)
\left[
 b_j(X^{(i)}(t))\,dt
+
\Sigma_{j\cdot}(X^{(i)}(t))\,dW^{(i)}(t)
\right].
\label{eq:def_bf_hat}
\end{align}
}
Here, the second equality in \eqref{eq:def_bf_hat} follows from applying Theorem~\ref{thm:semimartingale-identity} to each trajectory with
\(\Phi^{(i)}(t)=K_{\boldsymbol h}(Y^{(i)}(t)-x)\), together with the definition of 
\(\delta^{(i)}\). By condition~(C1), the Itô integrals in the last line have mean zero. Taking expectations and using \eqref{eq:uncensored-moment-identity} together with Fubini's theorem, we obtain, for each \(j=1,\ldots,d\),
\begin{align}
\label{eq:expectation-bfhat}
\mathbb E\!\left[\widehat{bf}_{j,N,\boldsymbol h}(x)\right]
&=
\frac{1}{T-t_0}
\int_{t_0}^T
\mathbb E\!\left[
\delta(t)K_{\boldsymbol h}(X(t)-x)b_j(X(t))
\right]\,dt \nonumber\\
&=
\int_{\mathbb R^d}
K_{\boldsymbol h}(y-x)b_j(y)\overline f(y)\,dy
=
\big(K_{\boldsymbol h}*(b_j\overline f)\big)(x).
\end{align}
These interchanges of expectation and integration are justified by Assumptions~\ref{ass:standard_X} and~\ref{ass:kernel} together with Proposition~\ref{ass:gaussian-density}. The final estimator for $b$ is
\begin{equation}
\label{eq:def_b_hat}
\widehat b_{N,\boldsymbol h,\boldsymbol h'}(x)
:=
\frac{\widehat{bf}_{N,\boldsymbol h}(x)}
{\widehat f_{N,\boldsymbol h'}(x)}
\mathbf 1_{\{\widehat f_{N,\boldsymbol h'}(x)>m_I/2\}},
\qquad x\in I .
\end{equation}
Since \(m_I\) is unknown, the theoretical constant \(m_I/2\) is
replaced in practice by a data-dependent threshold; see
Appendix~\ref{app:simulation-implementation-tables} for implementation
details.
Having defined the estimator, we turn to its statistical performance.
The next proposition collects the three fixed-bandwidth bounds needed
throughout the paper.
Set
\begin{align}
\label{eq:variance-constants}
\mathrm C_f&:=\|K\|_{2}^2,\qquad
\mathrm C_{bf}
:=
2\|K\|_{2}^2
\left(
\|b\|_{2,\overline f}^2
+
\frac{1}{T-t_0}
\big\|\sqrt{\operatorname{Tr}(\Sigma\Sigma^\top)}\big\|_{2,\overline f}^2
\right).
\end{align}

\begin{proposition}
\label{prop:fixed-bandwidth-risk}
Under Assumptions~\ref{ass:proj_decomp}, \ref{ass:standard_X}, \ref{ass:censoring-nondegeneracy}, and \ref{ass:kernel}(i), the following bounds hold for every
\(\boldsymbol h,\boldsymbol h'\in(0,1]^d\),
\begin{align}
&\mathbb E\!\left[
\big\|
\widehat f_{N,\boldsymbol h'}-\overline f
\big\|_{2}^2
\right]
\le
\big\|
K_{\boldsymbol h'}*\overline f-\overline f
\big\|_{2}^2
+
\frac{\mathrm C_f}{N \prod_{j=1}^d h'_j},
\label{eq:risk-den-aniso}
\\
&\mathbb E\!\left[
\big\|
\widehat{bf}_{N,\boldsymbol h}-b\overline f
\big\|_{2}^2
\right]
\le
\big\|
K_{\boldsymbol h}*(b\overline f)-b\overline f
\big\|_{2}^2
+
\frac{\mathrm C_{bf}}{N \prod_{j=1}^d h_j},
\label{eq:risk-num-aniso}
\\
&\mathbb E\!\left[
\big\|
\widehat b_{N,\boldsymbol h,\boldsymbol h'}-b
\big\|_{2,I}^2
\right]
\le
\frac{\mathrm C_{\mathrm{NW}}}{m_I^2}
\left[
\mathbb E\!\left[
\big\|
\widehat{bf}_{N,\boldsymbol h}-b\overline f
\big\|_{2}^2
\right]
+
\mathbb E\!\left[
\big\|
\widehat f_{N,\boldsymbol h'}-\overline f
\big\|_{2}^2
\right]
\right].
\label{eq:risk-ratio-aniso}
\end{align}
Here \(\mathrm C_f\) and \(\mathrm C_{bf}\) are defined in
\eqref{eq:variance-constants}, and
\(\mathrm{C}_{\mathrm{NW}}:=12\left(1 \vee\|b\|_{\infty,I}^2\right)\).
\end{proposition}

The proof relies on standard bias--variance decompositions and is given in Appendix~\ref{pf:prop:fixed-bandwidth-risk}. Combining Proposition~\ref{prop:fixed-bandwidth-risk} with
Lemma~\ref{lem:nikolskii-kernel-bias} yields explicit rates under the
Nikol'skii assumptions. For
\(\boldsymbol\alpha=(\alpha_1,\ldots,\alpha_d)\in(0,\infty)^d\), define
the harmonic mean
{
\[
\overline\alpha
:=
\left(\frac1d\sum_{j=1}^d\frac1{\alpha_j}\right)^{-1}.
\]
}
\begin{corollary}
\label{cor:fixed-bandwidth-rate}
Under the conditions of Proposition~\ref{prop:fixed-bandwidth-risk}, assume moreover that the kernel satisfies Assumption~\ref{ass:kernel}(ii).
Suppose that
\(\overline f\in\mathcal N_d^1(\boldsymbol\beta,L_f)\) and
\(b\overline f\in\mathcal N_d^d(\boldsymbol\gamma,L_{bf})\). Assume further that the kernel orders $\boldsymbol \nu =(\nu_1,..., \nu_d)$ satisfy \(\nu_j\ge (\lceil\gamma_j\rceil-1) \vee (\lceil\beta_j\rceil-1) \) for every \(j=1,\ldots,d\). 
Then, for a finite constant \(\mathrm C_{\mathcal N}\) depending only on \(d\), \(K\),
\(\boldsymbol\beta\), and \(\boldsymbol\gamma\),
\[
\mathbb E\!\left[
\big\|
\widehat b_{N,\boldsymbol h,\boldsymbol h'}-b
\big\|_{2,I}^2
\right]
\le
\frac{\mathrm C_{\mathrm{NW}}}{m_I^2}
\left[
\mathrm C_{\mathcal N}L_{bf}^2\sum_{j=1}^d h_j^{2\gamma_j}
+
\frac{\mathrm C_{bf}}{N \prod_{j=1}^d h_j}
+
\mathrm C_{\mathcal N}L_f^2\sum_{j=1}^d (h_j')^{2\beta_j}
+
\frac{\mathrm C_f}{N \prod_{j=1}^d h'_j}
\right].
\]
Here \(\mathrm C_f\) and
\(\mathrm C_{bf}\) are defined in \eqref{eq:variance-constants}. Thus, if we choose
\(
h_j^\star\asymp
N^{-\overline\gamma/(\gamma_j(2\overline\gamma+d))},
(h_j')^\star\asymp
N^{-\overline\beta/(\beta_j(2\overline\beta+d))},
\)
the risk of \(\widehat b_{N,\boldsymbol h^\star,(\boldsymbol h')^\star}\)
satisfies
\[
\mathbb E\!\left[
\big\|
\widehat b_{N,\boldsymbol h^\star,(\boldsymbol h')^\star}-b
\big\|_{2,I}^2
\right]
=
O\!\left(
N^{-2\overline\gamma/(2\overline\gamma+d)}
+
N^{-2\overline\beta/(2\overline\beta+d)}
\right).
\]
\end{corollary}

The proof of Corollary~\ref{cor:fixed-bandwidth-rate} is obtained by
applying Lemma~\ref{lem:nikolskii-kernel-bias} to the two bias terms in
Proposition~\ref{prop:fixed-bandwidth-risk}, and then balancing
\(\sum_j h_j^{2\gamma_j}\) with \((N \prod_{j=1}^d h_j)^{-1}\), and
\(\sum_j (h_j')^{2\beta_j}\) with \((N \prod_{j=1}^d h'_j)^{-1}\).

\subsection{Adaptive bandwidth selection}
\label{subsec:gl-selection}

We now choose the two bandwidths in a data-driven way. Let
\(\mathcal H_N\subset(0,1]^d\) be a finite bandwidth grid.
The following assumption controls the grid in the deviation bounds.

\begin{assumption}
\label{ass:GL-grid}
There exists a constant
\(\mathrm C_{\mathcal H}>0\) such that, for every \(N\ge2\),
\[
\frac{1}{N\prod_{j=1}^d h_j} \leq 1
\quad\text{for all }\boldsymbol h\in\mathcal H_N,
\qquad
|\mathcal H_N|\le
\mathrm C_{\mathcal H}N(\log N)^d.
\]
Moreover, for every \(\lambda>0\), there exists
\(\mathrm C_{\mathcal H}(\lambda)<\infty\), independent of \(N\), such that
{\small
\[
\sum_{\boldsymbol h\in\mathcal H_N}
\exp\!\left\{
-\lambda\left(\prod_{j=1}^d h_j\right)^{-1}
\right\}
\le
\mathrm C_{\mathcal H}(\lambda).
\]
}
\end{assumption}

This assumption is a complexity condition on the bandwidth family, preventing
the grid from being too rich at any fixed effective scale. 

\begin{example}
\label{ex:admissible-grids}
Assumption~\ref{ass:GL-grid} is satisfied by the dyadic and the
reciprocal grids,
\[
\begin{aligned}
\mathcal H_N^{\rm dyad}
&:=
\Bigl\{
(2^{-\ell_1},\ldots,2^{-\ell_d}):
\ell\in\mathbb N_0^d,\
\textstyle\sum_{j=1}^d\ell_j\le\lfloor\log_2 N\rfloor
\Bigr\},\\
\mathcal H_N^{\rm rec}
&:=
\Bigl\{
(1/k_1,\ldots,1/k_d):
k\in\mathbb N^d,\
\prod_{j=1}^{d} k_j\le N
\Bigr\}.
\end{aligned}
\]
In both cases, \(\prod_{j=1}^d h_j\ge N^{-1}\) by construction, and
Lemma~\ref{lem:product-index-counting} gives
\[
|\mathcal H_N^{\rm dyad}|
=
\sum_{s=0}^{\lfloor\log_2 N\rfloor}\binom{s+d-1}{d-1}
\le C_d(\log N)^d,
\qquad
|\mathcal H_N^{\rm rec}|
\le C_dN(\log N)^{d-1}.
\]
Finally, grouping the bandwidths according to the value of
\(\bigl(\prod_{j=1}^d h_j\bigr)^{-1}\), namely \(2^s\),
\(s\in\mathbb N_0\), for the dyadic grid and \(m\in\mathbb N\) for the
reciprocal grid, and using
Lemma~\ref{lem:product-index-counting} again, we obtain, for every
\(\lambda>0\),
\[
\sum_{\boldsymbol h\in\mathcal H_N^{\rm dyad}}
e^{-\lambda(\prod_{j=1}^d h_j)^{-1}}
\le
\sum_{s=0}^{\infty}
\binom{s+d-1}{d-1}e^{-\lambda 2^s}
<\infty,
\]
\[
\sum_{\boldsymbol h\in\mathcal H_N^{\rm rec}}
e^{-\lambda(\prod_{j=1}^d h_j)^{-1}}
\le
\sum_{m=1}^{\infty}
\#\{k\in\mathbb N^d:\prod_{j=1}^{d} k_j=m\}\,
e^{-\lambda m}
\le
\sum_{m=1}^{\infty}
m^{d-1}e^{-\lambda m}
<\infty.
\]
\end{example}

For \(\boldsymbol h,\boldsymbol\eta\in\mathcal H_N\), define the
smoothed numerator and denominator estimators by
\[
\widehat{bf}_{N,\boldsymbol h,\boldsymbol\eta}
:=
K_{\boldsymbol\eta}*\widehat{bf}_{N,\boldsymbol h},
\qquad
\widehat f_{N,\boldsymbol h,\boldsymbol\eta}
:=
K_{\boldsymbol\eta}*\widehat f_{N,\boldsymbol h}.
\]
The next lemma, proved in Appendix~\ref{subsec:pf:GL-commute}, provides a commutation property,  which is crucial for the arguments used in this section.

\begin{lemma}
\label{lem:GL-commute}
Assume that the kernel satisfies Assumption~\ref{ass:kernel}(i). Then, for all
\(\boldsymbol h,\boldsymbol\eta\in(0,1]^d\),
\(
\widehat{bf}_{N,\boldsymbol h,\boldsymbol\eta}
=
\widehat{bf}_{N,\boldsymbol\eta,\boldsymbol h}
\) and  \(
\widehat f_{N,\boldsymbol h,\boldsymbol\eta}
=
\widehat f_{N,\boldsymbol\eta,\boldsymbol h}
\).
Moreover, for every \(x\in\mathbb R^d\),
{\small
\begin{align*}
&\widehat{bf}_{N,\boldsymbol h,\boldsymbol\eta}(x)
=
\frac{1}{N(T-t_0)}
\sum_{i=1}^N
\int_{t_0}^T
\delta^{(i)}(t)
(K_{\boldsymbol h}*K_{\boldsymbol\eta})(X^{(i)}(t)-x)\,
dX^{(i)}(t), \\& \begingroup  \widehat f_{N,\boldsymbol h,\boldsymbol\eta}(x)
=
\frac{1}{N(T-t_0)}
\sum_{i=1}^N
\int_{t_0}^T
\delta^{(i)}(t)
(K_{\boldsymbol h}*K_{\boldsymbol\eta})(X^{(i)}(t)-x)\,dt. \endgroup
\end{align*}
}
\end{lemma}
 For \(\boldsymbol h\in\mathcal H_N\), define the selected bandwidth for the numerator estimator as
{\small
\begin{align} \label{eq:GL-hhat-bf}
    \widehat{\boldsymbol h}_N
\in
\arg\min_{\boldsymbol h\in\mathcal H_N}
\left\{
\sup_{\boldsymbol\eta\in\mathcal H_N}
\left(
\big\|
\widehat{bf}_{N,\boldsymbol h,\boldsymbol\eta}
-
\widehat{bf}_{N,\boldsymbol\eta}
\big\|_{2}^2
-
\frac{\kappa_1\mathrm C_{bf}}{N \prod_{j=1}^d \eta_j}
\right)_+ 
+
\frac{\kappa_2\mathrm C_{bf}}{N \prod_{j=1}^d h_j}
\right\},
\end{align}
}
where $\kappa_2\ge\kappa_1>0$. Similarly, the selected bandwidth for the denominator estimator is
{\small
\begin{align} \label{eq:GL-hhat-f}
    \widehat{\boldsymbol h}'_N
\in
\arg\min_{\boldsymbol h'\in\mathcal H_N}
\left\{
\sup_{\boldsymbol\eta\in\mathcal H_N}
\left(
\big\|
\widehat f_{N,\boldsymbol h',\boldsymbol\eta}
-
\widehat f_{N,\boldsymbol\eta}
\big\|_{2}^2
-
\frac{\kappa_3\mathrm C_f}{N \prod_{j=1}^d \eta_j}
\right)_+
+
\frac{\kappa_4\mathrm C_f}{N \prod_{j=1}^d h'_j}
\right\},
\end{align}
}
where $\kappa_4\ge\kappa_3>0,$ and \(u_+:=\max\{u,0\}.\)
The constants \(\mathrm C_f\) and \(\mathrm C_{bf}\) appearing in
\eqref{eq:GL-hhat-bf}--\eqref{eq:GL-hhat-f} are defined in
\eqref{eq:variance-constants}.
The constant \(\mathrm C_{bf}\) enters the criterion \eqref{eq:GL-hhat-bf}
only through the products \(\kappa_1\mathrm C_{bf}\) and
\(\kappa_2\mathrm C_{bf}\), so its exact value is immaterial. Any deterministic upper bound for $\mathrm C_{bf}$ may be used instead, at the cost of rescaling the calibration constants \(\kappa_1\) and \(\kappa_2\). In our
implementation, a rough preliminary estimate of \(\mathrm C_{bf}\) is used
only to fix the order of magnitude of these products, whose overall scale is
then calibrated numerically, as is customary for
adaptive procedures; see
Appendix~\ref{app:simulation-implementation-tables} for details.

The final selected ratio estimator is
\begin{equation}
\label{eq:GL-ratio-estimator}
\widehat b_N^{\mathrm{GL}}(x)
:=
\frac{
\widehat{bf}_{N,\widehat{\boldsymbol h}_N}(x)
}{
\widehat f_{N,\widehat{\boldsymbol h}'_N}(x)
}
\mathbf 1_{\{
\widehat f_{N,\widehat{\boldsymbol h}'_N}(x)>m_I/2
\}},
\qquad x\in I .
\end{equation}
Before stating the adaptive risk bounds, let us explain the role of the
contrast in \begingroup \eqref{eq:GL-hhat-bf}\endgroup. By \eqref{eq:expectation-bfhat} and
Lemma~\ref{lem:GL-commute},
\[
\mathbb E\!\left[
\widehat{bf}_{N,\boldsymbol h,\boldsymbol\eta}
\right]
-
\mathbb E\!\left[
\widehat{bf}_{N,\boldsymbol\eta}
\right]
=
K_{\boldsymbol\eta}*
\big(
K_{\boldsymbol h}*(b\overline f)-b\overline f
\big).
\]
Thus, the supremum term in
\eqref{eq:GL-hhat-bf} estimates a
\(K_{\boldsymbol\eta}\)-smoothed version of the bias of
\(\widehat{bf}_{N,\boldsymbol h}\), while the subtracted penalty compensates for the associated stochastic fluctuation, which is of order
\((N \prod_{j=1}^d \eta_j)^{-1}\). The term $\kappa_2\mathrm C_{bf}(N \prod_{j=1}^d h_j)^{-1}$ plays the usual role of the variance term. The contrast in \eqref{eq:GL-hhat-f} has the same
interpretation, with \(b\overline f\) replaced by \(\overline f\). Applying Proposition~\ref{prop:hilbert-uniform-deviation} to the families defined above yields the following deviation bounds.

\begin{lemma}
\label{lem:GL-deviation-verification}
Under Assumptions~\ref{ass:proj_decomp},
\ref{ass:standard_X}, \ref{ass:censoring-nondegeneracy},
\ref{ass:kernel}(i), and~\ref{ass:GL-grid}, let \(\mathrm C_f\) and
\(\mathrm C_{bf}\) be the constants defined in
\eqref{eq:variance-constants}. There exist finite constants \(\mathrm C\) and \(\rho_{\mathcal H}>0\), independent of \(N\) and uniform over the bandwidths in \(\mathcal H_N\), such that for every \(\kappa_1\ge72\left(1 \vee\|K\|_1^2\right)\),
\[
\mathbb E\!\left[
\sup_{\boldsymbol\eta\in\mathcal H_N}
\left(
\big\|\widehat{bf}_{N,\boldsymbol\eta}
-\mathbb E[\widehat{bf}_{N,\boldsymbol\eta}]\big\|_2^2
-\frac{\kappa_1\mathrm C_{bf}}{6N\prod_{j=1}^d\eta_j}
\right)_+\right]
\le \mathrm C\frac{(\log N)^{\rho_{\mathcal H}}}{N},
\]
and, uniformly in \(\boldsymbol h\in\mathcal H_N\),
\begin{align} \label{deviation_bf}
    \mathbb E\!\left[
\sup_{\boldsymbol\eta\in\mathcal H_N}
\left(
\big\|\widehat{bf}_{N,\boldsymbol h,\boldsymbol\eta}
-\mathbb E[\widehat{bf}_{N,\boldsymbol h,\boldsymbol\eta}]\big\|_2^2
-\frac{\kappa_1\mathrm C_{bf}}{6N\prod_{j=1}^d\eta_j}
\right)_+\right]
\le \mathrm C\frac{(\log N)^{\rho_{\mathcal H}}}{N}.
\end{align}
The same two bounds hold with
\(\widehat{bf}_{N,\boldsymbol\eta}\),
\(\widehat{bf}_{N,\boldsymbol h,\boldsymbol\eta}\), and
\(\mathrm C_{bf}\) replaced by
\(\widehat f_{N,\boldsymbol\eta}\),
\(\widehat f_{N,\boldsymbol h,\boldsymbol\eta}\), and
\(\mathrm C_f\), respectively, for every
\(\kappa_3\ge72\left(1 \vee\|K\|_1^2\right)\). One may take
\(\rho_{\mathcal H}=d+6\).
\end{lemma}

The proof is given in Appendix~\ref{app:GL-deviation-proof}. It relies on a
general deviation inequality for families of Hilbert-space-valued empirical
means, stated as Proposition~\ref{prop:hilbert-uniform-deviation} in
Appendix~\ref{app:hilbert-valued-concentration}. For a fixed bandwidth $\eta$, the key object in the proof is the $L^2$-valued random function
\[
Z_{\boldsymbol\eta}^{(i)}(x)
:=
\frac{1}{T-t_0}
\int_{t_0}^T
\delta^{(i)}(t)\,K_{\boldsymbol\eta}(Y^{(i)}(t)-x)\,dY^{(i)}(t),
\qquad\text{so that}\qquad
\widehat{bf}_{N,\boldsymbol\eta}(x)
=
\frac{1}{N}\sum_{i=1}^N Z_{\boldsymbol\eta}^{(i)}(x).
\]
Due to the stochastic-integral component, the norm $\|Z_{\boldsymbol\eta}\|_2$
admits no deterministic bound that is uniform over sample paths. We
therefore truncate on the event
\(
\big\{\|Z_{\boldsymbol\eta}\|_2\le M_{N,\boldsymbol\eta}\big\}
\)
for a suitable threshold $M_{N,\boldsymbol\eta}$; this single truncation
simultaneously controls every scalar projection
$\langle Z_{\boldsymbol\eta},\psi\rangle_2$ with $\|\psi\|_2=1$, and
constitutes the main new ingredient in the proof of
Proposition~\ref{prop:hilbert-uniform-deviation}. The proposition itself is
formulated for general families of unbounded Hilbert-space-valued statistics
and may be of independent interest. The proof of
Lemma~\ref{lem:GL-deviation-verification} consists of verifying its conditions for the stochastic integrals
$Z_{\boldsymbol\eta}^{(i)}$, which is where the diffusion structure enters the argument.

Using Lemmas~\ref{lem:GL-commute} and~\ref{lem:GL-deviation-verification}, we establish the following adaptive risk bound.

\begin{theorem}
\label{thm:GL-oracle}
Under Assumptions~\ref{ass:proj_decomp},
\ref{ass:standard_X}, \ref{ass:censoring-nondegeneracy},
\ref{ass:kernel}(i), and~\ref{ass:GL-grid}, assume also that
\(\kappa_1,\kappa_3\ge72\left(1 \vee\|K\|_1^2\right)\), and that
\(\kappa_2\ge\kappa_1\), \(\kappa_4\ge\kappa_3\). Let \(\mathrm C_f\) and
\(\mathrm C_{bf}\) be the constants defined in
\eqref{eq:variance-constants}. Then there exist finite constants
\(\mathrm C_1,\mathrm C_2,\mathrm C_3,\mathrm C_4,\mathrm C_5\) and
\(\rho_{\mathcal H}>0\), independent of \(N\) and uniform over the bandwidths in
\(\mathcal H_N\), such that for all \(N\ge2\),
{\small
\begin{align}
&\mathbb E\!\left[
\big\|
\widehat{bf}_{N,\widehat{\boldsymbol h}_N}
-
b\overline f
\big\|_{2}^2
\right]
\le
\mathrm C_1
\inf_{\boldsymbol h\in\mathcal H_N}
\left\{
\big\|
K_{\boldsymbol h}*(b\overline f)-b\overline f
\big\|_{2}^2
+
\frac{\kappa_2\mathrm C_{bf}}{N \prod_{j=1}^d h_j}
\right\}
+
\mathrm C_2\frac{(\log N)^{\rho_{\mathcal H}}}{N},
\label{eq:GL-oracle-bf}
\\
&\mathbb E\!\left[
\big\|
\widehat f_{N,\widehat{\boldsymbol h}'_N}
-
\overline f
\big\|_{2}^2
\right]
\le
\mathrm C_3
\inf_{\boldsymbol h'\in\mathcal H_N}
\left\{
\big\|
K_{\boldsymbol h'}*\overline f-\overline f
\big\|_{2}^2
+
\frac{\kappa_4\mathrm C_f}{N \prod_{j=1}^d h'_j}
\right\}
+
\mathrm C_4\frac{(\log N)^{\rho_{\mathcal H}}}{N}.
\label{eq:GL-oracle-f}
\\&\mathbb E\!\left[
\big\|
\widehat b_N^{\mathrm{GL}}-b
\big\|_{2,I}^2
\right]
\le
\begingroup \mathrm C_5\endgroup
\frac{\mathrm C_{\mathrm{NW}}}{m_I^2}
\Bigg[
\inf_{\boldsymbol h\in\mathcal H_N}
\left\{
\big\|
K_{\boldsymbol h}*(b\overline f)-b\overline f
\big\|_{2}^2
+
\frac{\mathrm C_{bf}}{N \prod_{j=1}^d h_j}
\right\}
\nonumber\\
& \qquad \qquad \qquad\qquad \qquad+
\inf_{\boldsymbol h'\in\mathcal H_N}
\left\{
\big\|
K_{\boldsymbol h'}*\overline f-\overline f
\big\|_{2}^2
+
\frac{\mathrm C_f}{N \prod_{j=1}^d h'_j}
\right\}
\Bigg]
+
\mathrm C_5\frac{(\log N)^{\rho_{\mathcal H}}}{N}.
\label{eq:GL-oracle-ratio}
\end{align}
}
Moreover, suppose that
\(\mathcal H_N\) is one of the two grids in
Example~\ref{ex:admissible-grids}, and that the kernel satisfies Assumption~\ref{ass:kernel}(ii) and 
\(b\overline f\in\mathcal N_d^d(\boldsymbol\gamma,L_{bf})\),
\(\overline f\in\mathcal N_d^1(\boldsymbol\beta,L_f)\), and that $\boldsymbol \nu =(\nu_1,..., \nu_d)$ in Assumption \ref{ass:kernel}(ii) is such that \(\nu_j\ge (\lceil\gamma_j\rceil-1) \vee (\lceil\beta_j\rceil-1) \) for every \(j=1,\ldots,d\). Then
\begin{equation}
\label{eq:GL-rate-final}
\mathbb E\!\left[
\big\|
\widehat b_N^{\mathrm{GL}}-b
\big\|_{2,I}^2
\right]
=
O\!\left(
N^{-2\overline\gamma/(2\overline\gamma+d)}
+
N^{-2\overline\beta/(2\overline\beta+d)}
+
\frac{(\log N)^{\rho_{\mathcal H}}}{N}
\right).
\end{equation}
\end{theorem}

The proof is provided in Appendix~\ref{subsec:pf:GL-oracle}.

\begin{remark}\label{rem:GL-rate}
\begin{enumerate}[label=(\roman*)]
\item
The rate in \eqref{eq:GL-rate-final} follows from the usual anisotropic
bias-variance balance; by Lemma~\ref{lem:nikolskii-kernel-bias}, the
numerator oracle bound reduces to balancing
\(\sum_j h_j^{2\gamma_j}\) with \((N\prod_{j=1}^d h_j)^{-1}\), whose
minimizer has coordinates
\(h_j^\star\asymp N^{-a_j}\), where
\(a_j:=\overline\gamma/(\gamma_j(2\overline\gamma+d))\),
\(j=1,\ldots,d\). Both grids in Example~\ref{ex:admissible-grids}
contain bandwidths of this order, because
\(\sum_{j=1}^d a_j
=\frac{\overline\gamma}{2\overline\gamma+d}\sum_{j=1}^d\gamma_j^{-1}
=\frac{d}{2\overline\gamma+d}<1\).
Indeed, for the dyadic grid, \(\ell_j=\lfloor a_j\log_2 N\rfloor\)
gives \(2^{-\ell_j}\asymp N^{-a_j}\) and
\(\sum_{j=1}^d\ell_j\le(\sum_{j=1}^d a_j)\log_2 N\le\log_2 N\),
so \((2^{-\ell_1},\ldots,2^{-\ell_d})\in\mathcal H_N^{\rm dyad}\);
for the reciprocal grid, \(k_j=\lfloor N^{a_j}\rfloor\vee1\) gives
\(1/k_j\asymp N^{-a_j}\) and
\(\prod_{j=1}^d k_j\le N^{\sum_j a_j}\le N\),
so \((1/k_1,\ldots,1/k_d)\in\mathcal H_N^{\rm rec}\).
The same argument applies to the denominator oracle bandwidths, with
\(\boldsymbol\gamma\) replaced by \(\boldsymbol\beta\).

\item
\begingroup  
It is worth emphasizing what censoring does and does not change in
\eqref{eq:GL-rate-final}. The rate is the same as in the uncensored
fixed-$T$, $N\to\infty$ regime; censoring enters the risk bound only
through the constants. Specifically, the target of the denominator
estimator is the uncensored occupation density
\(\overline f\), which is the
fully observed occupation density damped by the uncensoring
probability \(\overline G_t\le1\). Consequently, the lower bound
\(m_I=\inf_{x\in I}\overline f(x)\) can be substantially smaller than its
uncensored counterpart, for instance, when the estimation set $I$ is
visible only with small probability, so that \(\overline G_t\) is small on parts
of \(I\), and the bound \eqref{eq:GL-oracle-ratio} scales as
\(m_I^{-2}\). Heavy censoring, therefore, manifests itself as a
deterioration of the constants in the oracle inequality, not of the
rate.

\item
The lower bound \(\kappa_1,\kappa_3\ge72\,(1\vee\|K\|_1^2)\) required by the
theory is a sufficient condition and not a sharp one. It is common in adaptive estimation for theory to determine a penalty only up to a multiplicative constant, which must then be calibrated from the data. The associated
methodology goes back to Birgé and Massart~\cite{birge2007minimal} and Arlot and
Massart~\cite{arlot2009data}. For Goldenshluger--Lepski procedures specifically, we refer to
Lacour and Massart~\cite{lacour2016minimal}. Our own choice is
reported in Appendix~\ref{app:simulation-implementation-tables}.
\endgroup

\item It is worth comparing the selection rule \eqref{eq:GL-hhat-bf} with the bandwidth selection procedure for drift estimation developed by Della Maestra and Hoffmann~\cite{dellamaestra2022nonparametric}. First, they work in the fixed-$T$, $N\to\infty$ regime and select two bandwidths: one for time and a common one for all $d$ spatial coordinates. Unlike criterion \eqref{eq:GL-hhat-bf}, their criterion takes the supremum over bandwidths that precede $\boldsymbol{h}$ under a total ordering of the grid. To construct this ordering, they impose a relation between temporal and spatial smoothness. As they note, this imposes a restriction on the resulting anisotropic adaptation result. This restriction can be removed, as in Goldenshluger and Lepski~\cite{goldenshluger2011bandwidth}, but only at a significant technical cost. By contrast, criterion \eqref{eq:GL-hhat-bf} retains the unordered comparison scheme of \cite{goldenshluger2011bandwidth}: the supremum is taken over the entire grid. The commutation property established in Lemma~\ref{lem:GL-commute} makes this full-grid comparison tractable and allows the smoothness indices $\gamma_1,\ldots,\gamma_d$ to vary independently. Second, they consider pointwise risk, and the rate obtained in \cite[Theorem~15]{dellamaestra2022nonparametric} contains the logarithmic factor inherent in pointwise adaptation, in accordance with the classical Lepski--Low phenomenon~\cite{lepskii1991problem,low1992non}. By contrast, the integrated risk considered here yields the rate \eqref{eq:GL-rate-final}, in which the logarithmic factor appears only in the remainder term.

\end{enumerate}
\end{remark}

\section{Minimax lower bound}
\label{sec:minimax-rate}

We first establish an anisotropic minimax lower bound in the
full-observation experiment and then transfer it to the censored-observation setting. Throughout this section, the dispersion coefficient \(\Sigma\) and, when
censoring is considered, the censoring mechanism are fixed. The only unknown
parameter is the drift \(b\).

We first define the drift class over which the minimax risk is evaluated.
For \(\boldsymbol\alpha=(\alpha_1,\ldots,\alpha_d)\in[1,\infty)^d\) and
\(L_\alpha,L_0,B>0\), define
\begin{align*}
    \mathcal B_{\boldsymbol\alpha}(L_\alpha,L_0,B)&:=
\big\{
b:\mathbb R^d\to\mathbb R^d:
b\in\mathcal N_d^d(\boldsymbol\alpha,L_\alpha),\
|b(0)|\le B, 
|b(x)-b(y)|\le L_0|x-y|,\ \forall x,y\in\mathbb R^d
\big\}.
\end{align*}
The last two conditions imply
\(
|b(x)|\le B+L_0|x|,
\, x\in\mathbb R^d.
\)
Hence, every $b$ in this class satisfies the Lipschitz and linear-growth conditions in Assumption~\ref{ass:standard_X}.

\begin{theorem}
\label{thm:minimax-rate}
Under Assumption~\ref{ass:standard_X}, suppose moreover that the dispersion coefficient \(\Sigma\) is fixed and does
not depend on the drift \(b\). Fix \(\boldsymbol\alpha\in[1,\infty)^d\) and \(L_\alpha,L_0,B>0\). Suppose that there exist \(x^\star\in\mathbb R^d\) and
\(L_Q>0\) such that
\(
\prod_{j=1}^d[x_j^\star,x_j^\star+L_Q]\subset I .
\)
\begingroup 
Let
\(\mathcal X_N:=\{(X^{(i)}(t))_{t\in[0,T]}\}_{i=1}^N\)
denote the fully observed trajectories.
\endgroup
Then there exists a constant \(c>0\) such that, for all \(N\) large enough,
\[
\begingroup 
\inf_{\widetilde b_N=\widetilde b_N(\mathcal X_N)}
\endgroup
\sup_{b\in\mathcal B_{\boldsymbol\alpha}(L_\alpha,L_0,B)}
\mathbb E_b\!\left[
\|\widetilde b_N-b\|_{2,I}^2
\right]
\ge
cN^{-2\bar\alpha/(2\bar\alpha+d)}.
\]
\end{theorem}

\begin{corollary}
\label{cor:minimax-rate-censored}
\begingroup 
Under the assumptions of Theorem~\ref{thm:minimax-rate}, assume moreover
Assumption~\ref{ass:censoring-nondegeneracy}(C1), and that the censoring mechanism is fixed
and does not depend on \(b\). Let
\(O_N:=\{(Y^{(i)}(t),\delta^{(i)}(t))_{t\in[t_0,T]}\}_{i=1}^N\).
Then there exists a constant \(c>0\) such that, for all \(N\) large enough,
\[
\inf_{\widetilde b_N=\widetilde b_N(O_N)}
\sup_{b\in\mathcal B_{\boldsymbol\alpha}(L_\alpha,L_0,B)}
\mathbb E_b\!\left[
\|\widetilde b_N-b\|_{2,I}^2
\right]
\ge
cN^{-2\bar\alpha/(2\bar\alpha+d)}.
\]
\endgroup
\end{corollary}

\begingroup 
The proofs follow the many-hypotheses method of
Tsybakov~\cite[Chapter~2, Section~2.6]{tsybakov2009introduction} and are given
in Appendices~\ref{Sec:pf-minimax-rate} and~\ref{Sec:pf-minimax-rate-censored}.
\endgroup

\begin{remark}
\label{rem:minimax-optimality}
\begingroup 
\begin{enumerate}
    \item When \(d=1\), Theorem~\ref{thm:minimax-rate} yields the lower risk bound
\(
N^{-2\alpha_1/(2\alpha_1+1)}.
\)
The corresponding one-dimensional rate was previously established by
Denis, Dion-Blanc and Martinez~\cite{denis2021ridge}.
Theorem~\ref{thm:minimax-rate} provides its multidimensional anisotropic
counterpart, while Corollary~\ref{cor:minimax-rate-censored} shows that the same minimax lower-bound rate holds in the censored setting.

\item Fix \(\boldsymbol\gamma\in[1,\infty)^d\). For each
\(b\in\mathcal B_{\boldsymbol\gamma}(L_\gamma,L_0,B)\), let \(O_N^{(b)}\)
denote the censored observation \(O_N\) generated by the diffusion with drift
\(b\), and let \(\overline f_b\) denote the corresponding time-averaged
uncensored occupation density. We consider estimating \(b\) from
\(O_N^{(b)}\), uniformly over this drift class. To turn the bound of
Corollary~\ref{cor:fixed-bandwidth-rate} into a minimax upper bound, assume
additionally that its hypotheses hold uniformly over the class. There exist
\(\boldsymbol\beta\in(0,\infty)^d\) and constants
\(m_0,L_f,L_{bf}>0\), all independent of \(b\), such that, for every
\(b\in\mathcal B_{\boldsymbol\gamma}(L_\gamma,L_0,B)\),
\(\inf_{x\in I}\overline f_b(x)\ge m_0\),
\(\overline f_b\in\mathcal N_d^1(\boldsymbol\beta,L_f)\), and
\(b\overline f_b\in\mathcal N_d^d(\boldsymbol\gamma,L_{bf})\), and suppose
that the remaining constants in that corollary are uniformly bounded. If
\(\beta_j\ge\gamma_j\) for every \(j\), then
\(\overline\beta\ge\overline\gamma\), so the denominator term converges no
more slowly than the numerator term. Corollary~\ref{cor:fixed-bandwidth-rate}
therefore gives the uniform upper rate
\(N^{-2\overline\gamma/(2\overline\gamma+d)}\), whereas
Corollary~\ref{cor:minimax-rate-censored}, applied with
\(\boldsymbol\alpha=\boldsymbol\gamma\), gives a lower bound of the same
order. Under these additional uniform assumptions, the fixed-bandwidth upper
bound is thus minimax optimal. The same minimax-optimality conclusion holds in the full-observation setting of Theorem~\ref{thm:minimax-rate}.
\item As a special case, suppose that the preceding uniform conditions hold with \(\boldsymbol\beta=\boldsymbol\gamma=\mathbf 1\), where \(\mathbf 1=(1,\ldots,1)\). Taking \(\boldsymbol\alpha=\mathbf 1\) in Theorem~\ref{thm:minimax-rate} and Corollary~\ref{cor:minimax-rate-censored} yields a minimax lower bound of order \(N^{-2/(d+2)}\), while Corollary~\ref{cor:fixed-bandwidth-rate} gives an upper bound of the same order. Hence, under these uniform regularity conditions, the minimax rate is \(N^{-2/(d+2)}\). Such conditions can, in particular, be verified under additional regularity assumptions on the uncensoring probabilities \(\overline G_t\) associated with the process \(\theta(t)\). These assumptions allow one to establish uniformly over the drift class that \(\overline f_b\in\mathcal N_d^1(\mathbf 1,L_f)\) and \(b\overline f_b\in\mathcal N_d^d(\mathbf 1,L_{bf})\). Since their verification requires further assumptions on the censoring mechanism and is not needed for the general lower-bound result, the details are omitted.

\item The lower bound in Corollary~\ref{cor:minimax-rate-censored} does not depend on the censoring mechanism. It shows that the full-observation minimax lower bound continues to hold under censoring, but it does not quantify the additional loss caused by censoring. In particular, it does not recover the factor \(m_I^{-2}\) appearing in the upper bounds.
\end{enumerate}

\endgroup
\end{remark}

\section{Simulation study}
\label{sec:simulation}

\renewcommand{\floatpagefraction}{0.75}
\renewcommand{\textfraction}{0.08}
\renewcommand{\topfraction}{0.93}
\renewcommand{\bottomfraction}{0.85}

We illustrate the finite-sample behavior of the estimator
\(\widehat b_N^{\mathrm{GL}}\) defined in
\eqref{eq:GL-ratio-estimator}. The numerical study is conducted in
\(d=2\) under the projected observation scheme
\eqref{eq:projected-observation}. The process \(X(t)\) follows
\eqref{eq:SDE_X}, with a constant dispersion matrix
\(\Sigma(x)=0.5I_2\), \(x\in\mathbb R^2\), all estimators are computed on the domain \(I=[-2,2]^2\).

Besides the proposed estimator, we report two additional
estimators for numerical comparison. The first, \(\widehat b_N^{\mathrm{LS}}\) is a projection least-squares estimator. Its one-dimensional version follows Huang~\cite{huang2026censoredsde}, while in dimension two, we use a tensor-product extension based on Dussap~\cite{dussap2023nonparametric}. The second,
\(\widehat b_N^{\mathrm{PCO}}\), is based on the PCO method of Lacour, Massart and Rivoirard~\cite{lacour2017estimator}, as adapted to i.i.d. diffusion paths by Marie and Rosier~\cite{marie2023nadaraya}. They are included as an empirical benchmark. The oracle inequality established in Section~\ref{sec:kernel-estimator} applies only to the GL procedure. The exact discrete formulae, grids, LS estimator, PCO criterion, effective comparison region, and MISE definition are given in Appendix~\ref{app:simulation-implementation-tables}.

As is standard in the fixed-\(T\), \(N\to\infty\) literature on drift
estimation from i.i.d.\ diffusion paths (see, e.g.,
Comte
and Genon-Catalot \cite{comte2020nonparametric}), the
theoretical analysis is conducted in the continuous-record framework,
and the discrete formulae introduced in Appendix~\ref{app:simulation-implementation-tables} are used as numerical approximations to the continuous-record estimators. The oracle inequalities established in Section~\ref{sec:kernel-estimator} apply to the continuous-record estimators and do not account for discretization error.
We consider the following three drift fields.
\begin{itemize}[leftmargin=*,topsep=2pt,itemsep=2pt]
\item \textbf{Ornstein-Uhlenbeck (OU):}
\(
 b(x_1,x_2)=(-x_1,-x_2).
\)
\\
This is the standard linear benchmark.
\item \textbf{Rotational OU:}
\(
 b(x_1,x_2)=(-x_1-x_2,\;x_1-x_2).
\)
\\
The antisymmetric component tests whether the procedure recovers both the
geometry and the orientation of a two-dimensional vector field, in the
spirit of multivariate and nonreversible diffusion settings such as
Strauch~\cite{Strauch2015} and Aeckerle-Willems and Strauch~\cite{AeckerleWillemsStrauch2022}.
\item \textbf{Fixman-like nonlinear drift:}
\(
 b(x_1,x_2)=\bigl(x_1(1-x_1^2),\;-0.5x_2\bigr).
\)
\\
This benchmark is inspired by the Fixman-potential example in
Schmisser~\cite{Schmisser2013} and is used to assess anisotropic
adaptation. 
\end{itemize}
The first two drifts satisfy Assumption~\ref{ass:standard_X} directly. The Fixman-like drift is smooth and locally Lipschitz, but it is not globally Lipschitz because of the cubic term in its first coordinate. We therefore use it as a nonlinear benchmark, while noting that it falls outside the globally Lipschitz framework of Assumption~\ref{ass:standard_X}.

All paths are simulated by the Euler--Maruyama scheme on the regular grid
\(t_k=k\Delta\), \(k=0,\ldots,n\), with \(\Delta=0.02\), \(T=50\), and
\(t_0=5\). Thus \(n=T/\Delta\), and only time indices \(k\) such that
\(t_k\ge t_0\) are used for estimation and evaluation.
We use the following two censoring mechanisms. Coordinatewise right censoring satisfies Assumption~\ref{ass:proj_decomp}, whereas random-ball censoring is covered by the locally DC argument in Remark~\ref{rem:scope-proj-decomp}(ii). The
auxiliary censoring processes are simulated independently of the Brownian increments used to generate the diffusion paths. Hence Assumption~\ref{ass:censoring-nondegeneracy}(C1) holds by construction.
\begin{itemize}[leftmargin=*,topsep=2pt,itemsep=3pt]
\item \textbf{Random ball censoring.} For trajectory \(i\),
\[
D^{(i)}_{t_k}=B(0,R^{(i)}_{t_k}),\qquad
Y^{(i)}_{t_k}=\Pi_{D^{(i)}_{t_k}}(X^{(i)}_{t_k}),\qquad
\delta^{(i)}_{t_k}=\mathbf 1_{\{|X^{(i)}_{t_k}|<R^{(i)}_{t_k}\}}.
\]
The radius process is generated by
\[
\widetilde R^{(i)}_{t_{k+1}}
=R^{(i)}_{t_k}+\kappa_R(r-R^{(i)}_{t_k})\Delta
+\sigma_R\sqrt{\Delta}\,\varepsilon^{(i)}_{t_{k+1}},\qquad
R^{(i)}_{t_{k+1}}=\max(\widetilde R^{(i)}_{t_{k+1}},R_{\min}),
\]
with \(R^{(i)}(0)=R_0\), \(R_0=3\), \(\kappa_R=0.8\),
\(\sigma_R=0.2\), and \(R_{\min}=0.3\).
\item \textbf{Coordinatewise right censoring.} For trajectory \(i\),
\[
D^{(i)}_{t_k}=(-\infty,\theta^{(i)}_{1,t_k}]\times(-\infty,\theta^{(i)}_{2,t_k}],
\quad
Y^{(i)}_{t_k}=X^{(i)}_{t_k}\wedge \theta^{(i)}_{t_k},
\quad
\delta^{(i)}_{t_k}
=\mathbf 1_{\{X^{(i)}_{1,t_k}<\theta^{(i)}_{1,t_k},\;
X^{(i)}_{2,t_k}<\theta^{(i)}_{2,t_k}\}}.
\]
The threshold process is generated coordinatewise by
\[
\theta^{(i)}_{j,t_{k+1}}
=\theta^{(i)}_{j,t_k}+\kappa_\theta(\theta_\infty-\theta^{(i)}_{j,t_k})\Delta
+\sigma_\theta\sqrt{\Delta}\,\xi^{(i)}_{j,t_{k+1}},
\qquad j=1,2,
\]
with \(\theta^{(i)}_j(0)=\theta_0\), \(\theta_0=3\), \(\kappa_\theta=0.8\), and
\(\sigma_\theta=0.15\).
\end{itemize}
In both mechanisms, \(\{\varepsilon^{(i)}_{t_{k+1}}\}\) and \(\{\xi^{(i)}_{j,t_{k+1}}\}\) are i.i.d.
standard normal variables over all displayed indices, and are independent
of the Brownian increments used to simulate \(X^{(i)}\). Since \(\sigma_R>0\), \(\sigma_\theta>0\), and the Gaussian variables are independent,
\(\mathbb P(R^{(i)}_{5}>2\sqrt 2)>0\) and
\(\mathbb P(\theta^{(i)}_{1,5}>2,\theta^{(i)}_{2,5}>2)>0\).
Thus, Condition (C2) holds with $t_{\star}=5$ for both mechanisms. Since $T>5$, the proof of Proposition \ref{prop:positive-denominator} yields $t_I>5$, so that the choice $t_0=5$ is admissible and $m_I>0$.

For each model and censoring
mechanism, the mean-reversion levels \(r,\theta_\infty\) are calibrated by a preliminary grid search so that the empirical
censoring rate is close to \(20\%\). We define the theoretical and empirical censoring rates by
\[
\mathrm{CR}=1-\frac{1}{T-t_0}\int_{t_0}^T\mathbb P(\delta(t)=1)\,dt,
\qquad
\widehat{\mathrm{CR}}
=1-\frac{1}{N(n-n_0+1)}\sum_{i=1}^N\sum_{k=n_0}^{n}\delta^{(i)}_{t_k},
\]
where \(n_0=t_0/\Delta\). The selected values are reported in
Table~\ref{tab:censoring-calibration}.

\begin{table}[!htbp]
\centering
\footnotesize
\setlength{\tabcolsep}{4pt}
\begin{tabular}{llccc}
\hline
Model & Censoring mechanism & Parameter & Value & \(\widehat{\mathrm{CR}}\) \\
\hline
OU & random ball & \(r\) & \(0.694737\) & \(0.1912\) \\
OU & coordinatewise right & \(\theta_\infty\) & \(0.494737\) & \(0.1809\) \\
Fixman-like & random ball & \(r\) & \(1.289167\) & \(0.2021\) \\
Fixman-like & coordinatewise right & \(\theta_\infty\) & \(1.084211\) & \(0.1850\) \\
rotational OU & random ball & \(r\) & \(0.694737\) & \(0.1878\) \\
rotational OU & coordinatewise right & \(\theta_\infty\) & \(0.494737\) & \(0.1810\) \\
\hline
\end{tabular}
\caption{Calibration of the censoring mechanisms. The parameters \(r\)
and \(\theta_\infty\) denote the mean-reversion levels of the random radius and coordinatewise thresholds, respectively.}
\label{tab:censoring-calibration}
\end{table}

Since the data are informative only where the process is sufficiently
often observed without censoring, graphical comparisons and error
evaluation are restricted to a common effective region \(\mathcal R\).
At the empirical level, \(\mathcal R\) plays the role of the compact estimation set $I$ used in the theoretical analysis. The risk bounds apply on any compact set over which the uncensored occupation density is bounded away from zero; see Proposition~\ref{prop:positive-denominator}. We therefore construct \(\mathcal R\) as a data-driven proxy for such a set, using a kernel estimate of the local uncensoring proportion. Importantly, \(\mathcal R\) is used only for evaluation and does not enter the construction of any estimator. All methods are evaluated on the same region. Its exact definition is given in Appendix~\ref{app:simulation-implementation-tables}.
Figure~\ref{fig:rotational-mask-context} shows the two projection
mechanisms and the resulting effective region for the rotational OU
model.

\begin{figure}[!t]
\centering
\includegraphics[width=0.66\textwidth]{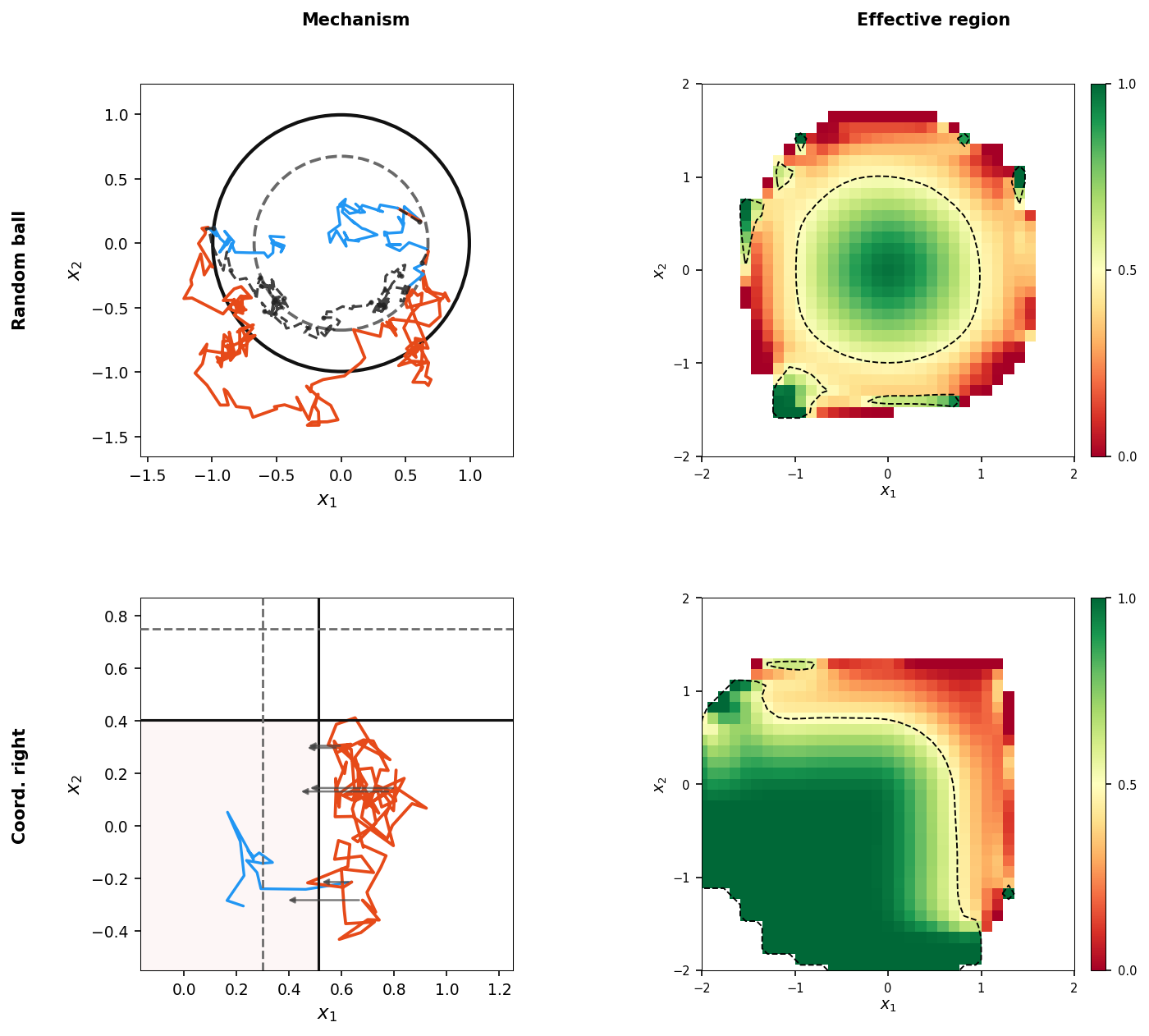}
\caption{Projection mechanisms and effective regions for the rotational
OU model. The top row corresponds to random ball censoring and the bottom
row to coordinatewise right censoring. In the left panels, blue trajectory
segments correspond to uncensored observations, while orange segments
correspond to censored observations. The right panels display
\(\widehat p_{\mathrm{unc}}\) on the evaluation grid, and the dashed curve
represents the boundary of \(\mathcal R\).}
\label{fig:rotational-mask-context}
\end{figure}

The effect of anisotropic smoothing is isolated in
Figure~\ref{fig:iso-aniso-fixman-focus}. For the Fixman-like model under
random ball censoring, the anisotropic grid can better adapt to the coordinate-dependent smoothness of the drift than the isotropic grid, particularly in regions dominated by the nonlinear first component.

\begin{figure}[!t]
\centering
\includegraphics[width=0.50\textwidth]{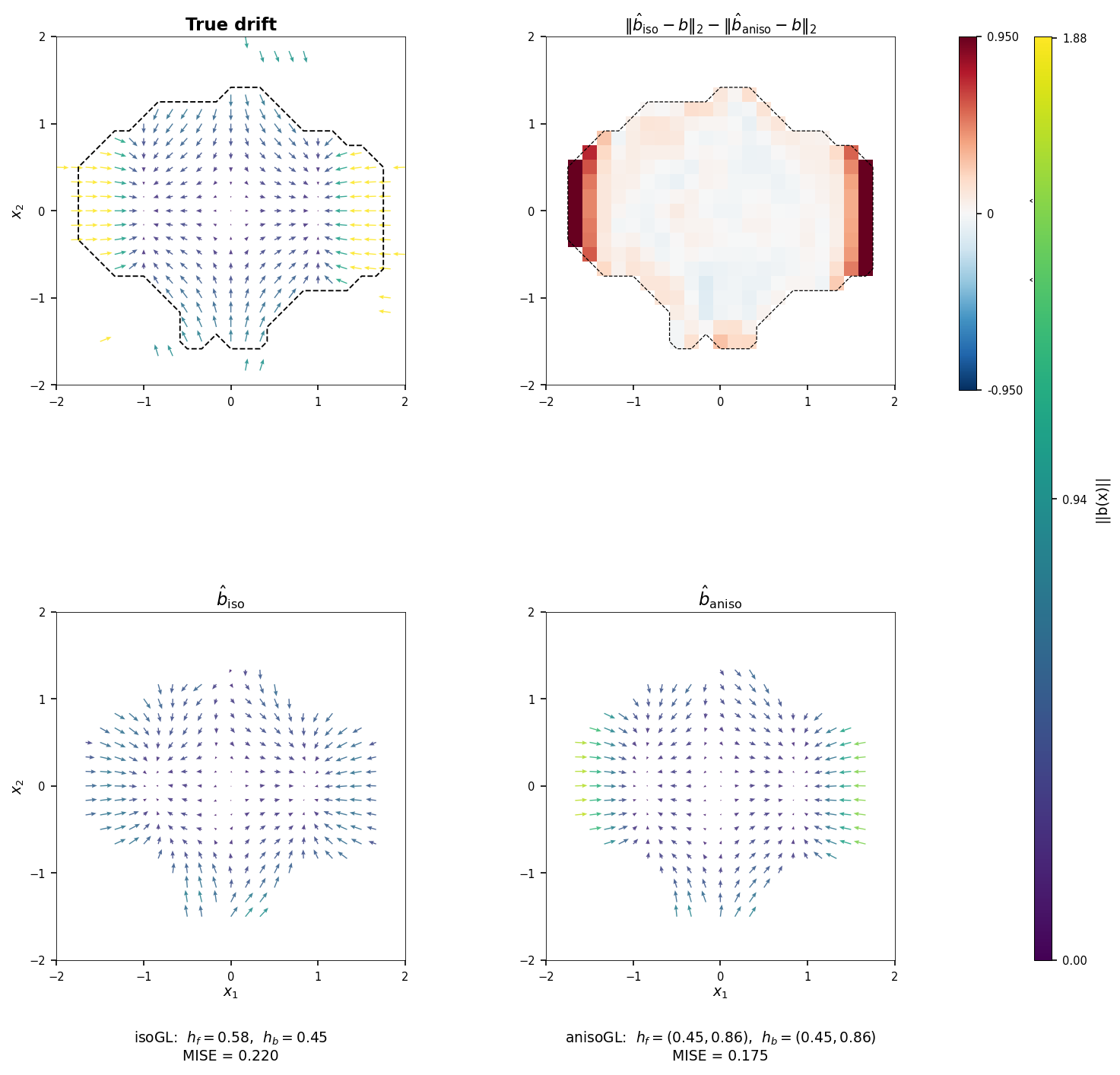}
\caption{Isotropic versus anisotropic GL reconstruction for the
Fixman-like model under random ball censoring, with \(N=50\) trajectories.
The upper-left panel shows the true drift and the boundary of the
effective region \(\mathcal R\). The lower panels show the isotropic and
anisotropic reconstructions. The upper-right panel displays
\(\|\widehat b_{\mathrm{iso}}^{\mathrm{GL}}-b\|_2
-\|\widehat b_{\mathrm{aniso}}^{\mathrm{GL}}-b\|_2\) on \(\mathcal R\);
positive values indicate a smaller pointwise error for the anisotropic
estimate.}
\label{fig:iso-aniso-fixman-focus}
\end{figure}

We now compare the three estimators visually at \(N=200\).
Figure~\ref{fig:fixman-like-comparison} shows the comparison for the Fixman-like nonlinear drift. The corresponding comparisons for the OU and
rotational OU drifts are reported in Figures~\ref{fig:linear-ou-comparison} and~\ref{fig:rotational-ou-comparison}.

\begin{figure}[!htbp]
\centering
\includegraphics[width=0.82\textwidth]{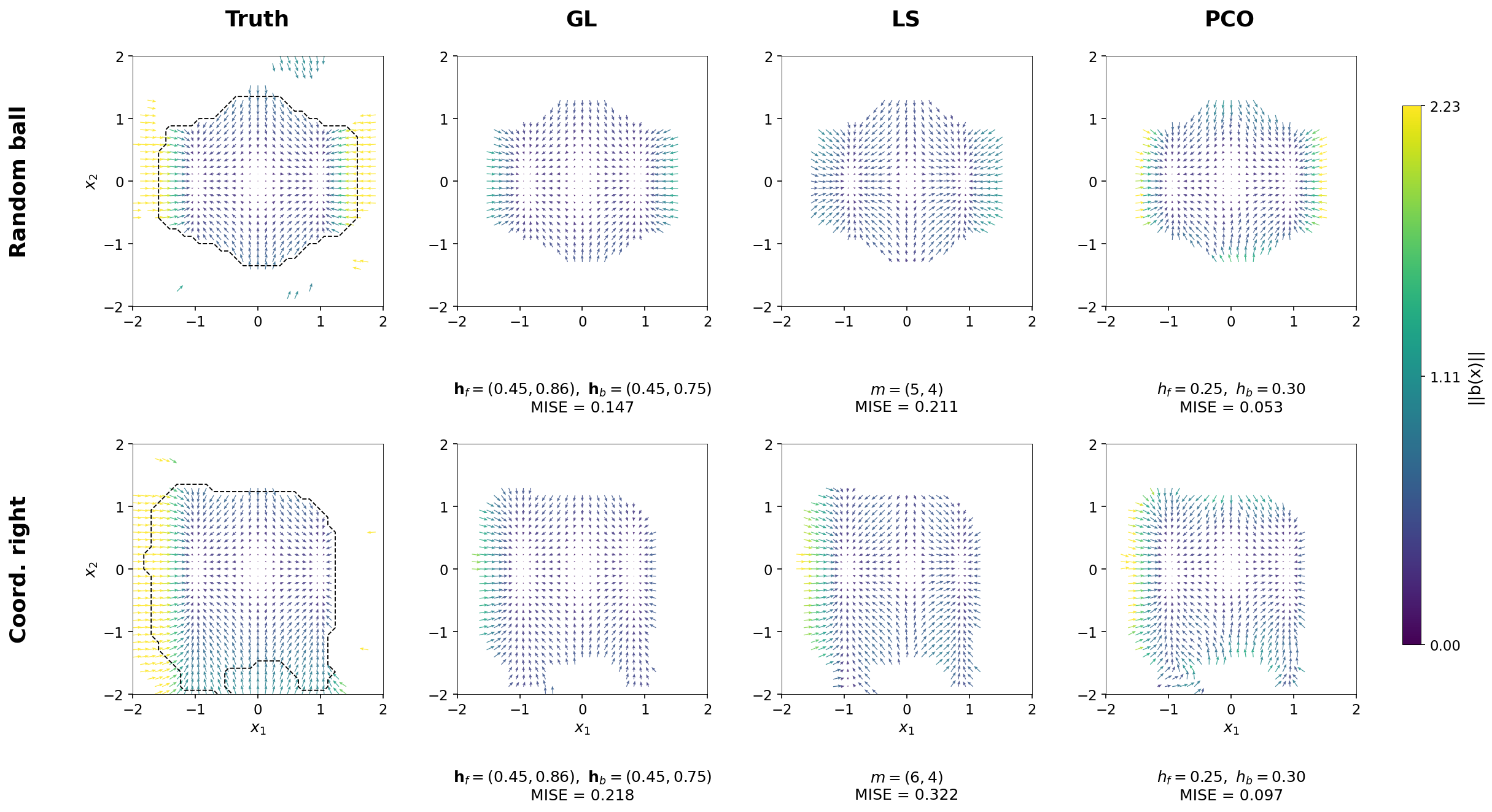}
\caption{Drift reconstruction for the Fixman-like model. The first
column shows the true drift field, and the next three columns show the
estimates obtained with \(\widehat b_N^{\mathrm{GL}}\),
\(\widehat b_N^{\mathrm{LS}}\), and
\(\widehat b_N^{\mathrm{PCO}}\). The top row corresponds to random ball
censoring and the bottom row to coordinatewise right censoring.}
\label{fig:fixman-like-comparison}
\end{figure}

Finally, we report empirical MISE curves as a function of the number
\(N\in\{50,100,200,300\}\) of independent trajectories, keeping \(T=50\)
fixed. The error is computed on \(\mathcal R\) over a \(35\times35\)
uniform grid on \([-2,2]^2\), with the normalization specified in
Appendix~\ref{app:simulation-implementation-tables}.
Figure~\ref{fig:mise-comparison} displays the MISE curves, while Tables~\ref{tab:mise-linear-ou}--\ref{tab:mise-fixman-like} report the corresponding numerical values.

The PCO estimator is empirically stable across the models and censoring mechanisms considered, and it achieves competitive performance in most settings. These numerical results do not constitute a theoretical guarantee. Among the three procedures, only the GL estimator satisfies an oracle inequality in the present censored setting. The LS estimator is less stable in some configurations, especially when the effective
region is irregular or the drift is strongly nonlinear, but it remains
computationally fast. The GL
estimator performs well for the two linear drifts, whereas its errors are larger for the
Fixman-like model. One possible explanation is the instability of the ratio estimator in regions where the estimated denominator is small or varies rapidly. This is consistent with the factor \(m_I^{-2}\) in the theoretical risk bound for the ratio estimator.

\FloatBarrier

\section{Concluding remarks}
\label{sec:concluding-remarks}

We conclude with three possible extensions of the present work.

First, the censoring law is fixed throughout the paper, and the analysis
requires the uncensored occupation density to be uniformly positive on the
estimation region. It would be interesting to study a regime of disappearing
visibility, where this density may decrease with $N$. The corresponding rates
should depend on an effective sample size reflecting both the number of
trajectories and the visibility level, and identifying the associated minimax
transition remains an open problem.

Second, our adaptive procedure is based on the Goldenshluger--Lepski method.
Developing a PCO alternative under censoring would require new arguments,
because the censoring indicator generates additional terms that are absent
under full observation.

Finally, we consider $N$ independent trajectories over a fixed time horizon.
In the complementary regime $N=1$ and $T\to\infty$, independence-based
concentration arguments would have to be replaced by mixing techniques.
Obtaining sharp risk bounds and quantifying the loss of information caused by
censoring in this long-time setting are natural directions for future work.

\section*{Acknowledgments}
The author expresses his sincere gratitude to both of his supervisors, Nicolas Marie (Laboratoire Modal'X, Universit\'e Paris Nanterre) and
Fabienne Comte (Laboratoire MAP5, Universit\'e Paris Cit\'e), for their invaluable guidance, insightful discussions, and continued support throughout the
development of this work.

\clearpage
\appendix
\setcounter{equation}{0}
\numberwithin{equation}{section}
\setcounter{figure}{0}
\numberwithin{figure}{section}
\setcounter{table}{0}
\numberwithin{table}{section}

\textit{Notation.}
For matrices,
\(\langle A,B\rangle_F:=\operatorname{Tr}(A^\top B)\) and
\(\|A\|_F:=\langle A,A\rangle_F^{1/2}\).

\section{Proofs for Section~\ref{sec:model-preliminaries}}
\label{app:proofs-model}

\begingroup
 
\subsection{Proof of Proposition~\ref{prop:projection-DC}}
\label{Sec:pf-projection-DC}

We first prove (i). Fix \((\bar \theta, \bar x)\in\Theta\times\mathbb R^d\). Since \(\Theta\times\mathbb R^d\) is open, there exists \(\rho_\theta,\rho_x>0\) such that 
\[
    \mathsf K:=\mathsf K_\theta\times \mathsf K_x := \prod_{i=1}^p[\bar \theta_{i}-\rho_\theta,\bar \theta_{i}+\rho_\theta] \times \prod_{j=1}^d[\bar x_{j}-\rho_x,\bar x_{j}+\rho_x] \subset \Theta\times\mathbb R^d.
\]
$\mathsf K$ is a polyhedron since it is described by finitely many linear inequalities
\[
    \mathsf K = \{ (\theta,x) \in  \Theta\times \mathbb R^d :\theta_i\le \bar \theta_{i}+\rho_\theta,\quad
    -\theta_i\le -\bar \theta_{i}+\rho_\theta,\quad
    x_j\le \bar x_{j}+\rho_x,\quad
    -x_j\le -\bar x_{j}+\rho_x \}.
\]
For fixed \((\theta,x)\), write \(y=x+u\). Then \eqref{optimization_problem} is equivalent to finding
\[
    u^\star(\theta,x)\in
    \arg\min_{u\in\mathbb R^d}\frac12|u|^2
    \quad\text{subject to}\quad
    Mu\le c_0+\mathsf B_D\theta-Mx,
\]
and \(\Pi_{D(\theta)}x=x+u^\star(\theta,x)\). With \(\xi=(\theta,x)\in\mathsf K\), this is exactly the form of Lemma~\ref{lem:mpqp-pwa}, with
\[
    \Xi=\mathsf K,\qquad
    \mathsf H=I_d,\qquad
    \mathsf A=M,\qquad
    w=c_0,\qquad
    S=(\mathsf B_D,-M).
\]
The feasibility condition holds by (S). Indeed, for each
\(\theta\), condition~(S) gives \(y_\theta\) such that
\(My_\theta<c(\theta)\). Hence, for any \(x\), the point \(u=y_\theta-x\)
satisfies
\[
Mu=My_\theta-Mx<c(\theta)-Mx=c_0+\mathsf B_D\theta-Mx,
\]
and is therefore feasible. Hence Lemma~\ref{lem:mpqp-pwa} gives that
\((\theta,x)\mapsto u^\star(\theta,x)\) is continuous and piecewise affine on
\(\mathsf K\). Since \((\theta,x)\mapsto x\) is affine,
\(
    F(\theta,x):=\Pi_{D(\theta)}x=x+u^\star(\theta,x)
\)
is continuous and piecewise affine on \(\mathsf K\), and therefore on the open neighborhood $U:=\operatorname{int} \mathsf K$ of \((\bar \theta,\bar x)\).

It remains to pass from piecewise affinity to local DC regularity. To this end, we use the
standard result that every continuous piecewise affine function can be
represented as the difference of two convex piecewise affine functions; see
Kripfganz and Schulze~\cite{kripfganz1987piecewise}. Hence each
coordinate \(F_j\), being continuous and piecewise affine, is DC
on \(U\). Since \((\bar \theta,\bar x)\) was arbitrary, every
\(F_j\) is locally DC on \(\Theta\times\mathbb R^d\).

Part (ii) is a direct application of Bouleau~\cite{bouleau1984formules}, since for every
\(j\in\{1,\ldots,d\}\), \((\theta,x) \mapsto F_j(\theta,x)\) is locally DC and $Z(t)=(\theta(t),X(t))$ is a continuous semi-martingale by definition.
\qed

\endgroup

\subsection{Proof of Theorem~\ref{thm:semimartingale-identity}}
\label{Sec:pf-lem:pathwise}

Fix \(j\in\{1,\ldots,d\}\) and put \(d_Z=p+d\). By
Proposition~\ref{prop:projection-DC}, the coordinate \(Y_j=F_j(Z)\)
admits the semimartingale decomposition
\eqref{eq:Y-semimartingale-decomp}; it remains to identify this
decomposition on the uncensored region \(\mathcal U_D\), where
\(F(\theta,x)=x\) and hence \(F_j=\ell_j\) with
\(\ell_j(\theta,x):=x_j\). Write
\(a_j:=\nabla\ell_j\in\mathbb R^{d_Z}\), the vector whose only nonzero
coordinate, equal to \(1\), is in position \(p+j\). We first verify
that the selections in Proposition~\ref{prop:projection-DC} can be
chosen so that \(F_j^*=a_j\) on \(\mathcal U_D\).

Take the open cover \((U_{j,n})_{n\ge1}\), the sets
\((C_{j,n})_{n\ge1}\), the local DC decompositions
\(F_j=G_{j,n}-H_{j,n}\) on \(U_{j,n}\), and the measurable subgradient
selections \((\overline G_{j,n}^*,\overline H_{j,n}^*)\) from
Proposition~\ref{prop:projection-DC}, and set
\(E_{j,n}:=\mathcal U_D\cap C_{j,n}\). Fix \(n\) and
\(z\in\mathcal U_D\cap U_{j,n}\). Since \(G_{j,n}-H_{j,n}=\ell_j\) in a
neighborhood of \(z\), the convex functions \(G_{j,n}\) and
\(H_{j,n}+\ell_j\) coincide near \(z\), and we claim that
\(\partial G_{j,n}(z)=\partial(H_{j,n}+\ell_j)(z)\).

Indeed, let
\(v\in\partial(H_{j,n}+\ell_j)(z)\) and \(y\in U_{j,n}\). For all
sufficiently small \(\lambda\in(0,1)\),
\(z_\lambda:=z+\lambda(y-z)\) lies in the neighborhood where
\(G_{j,n}=H_{j,n}+\ell_j\), so the subgradient inequality at \(z\) and
the convexity of \(G_{j,n}\) give
\[
G_{j,n}(z)+\lambda\langle v,y-z\rangle
\le
G_{j,n}(z_\lambda)
\le
(1-\lambda)G_{j,n}(z)+\lambda G_{j,n}(y),
\]
whence \(G_{j,n}(y)\ge G_{j,n}(z)+\langle v,y-z\rangle\), that is,
\(v\in\partial G_{j,n}(z)\); the reverse inclusion follows by
exchanging the roles of \(G_{j,n}\) and \(H_{j,n}+\ell_j\).

By the
subdifferential sum rule
\cite[Theorem~23.8]{Rockafellar1970}, therefore,
\[
\partial G_{j,n}(z)
=\partial H_{j,n}(z)+\{a_j\},
\qquad z\in\mathcal U_D\cap U_{j,n}.
\]
Now redefine, on \(U_{j,n}\),
\[
H_{j,n}^*:=\overline H_{j,n}^*,
\qquad
G_{j,n}^*
:=
\mathbf 1_{E_{j,n}}\bigl(\overline H_{j,n}^*+a_j\bigr)
+\mathbf 1_{U_{j,n}\setminus E_{j,n}}\,\overline G_{j,n}^*.
\]
Since \(\mathcal U_D\) is Borel, these are measurable subgradient
selections on \(U_{j,n}\), and
\(G_{j,n}^*-H_{j,n}^*=a_j\) on \(E_{j,n}\). Consequently, for
\(z\in\mathcal U_D\),
\(\sum_{n\ge1}\mathbf1_{C_{j,n}}(z)
\bigl(G_{j,n}^*(z)-H_{j,n}^*(z)\bigr)=a_j\),
so the choice \(F_j^*=a_j\) on \(\mathcal U_D\) is compatible with the
representation in Proposition~\ref{prop:projection-DC}. This proves
\eqref{eq:Fstar-on-U}.

\begingroup 
For the finite-variation term, fix \(\omega\) outside a null set on
which the decomposition of Proposition~\ref{prop:projection-DC} holds
on \([0,T]\), and set
\(O_\omega:=\{t\in[0,T]:Z(t,\omega)\in\mathcal U_D\}\). Since
\(\mathcal U_D\) is open and \(t\mapsto Z(t,\omega)\) is continuous,
\(O_\omega\) is relatively open in \([0,T]\), hence a countable
disjoint union of relatively open intervals, say
\(O_\omega=\bigcup_{n\ge1}I_n\). On any \([a,b]\subset I_n\), the map
\(F_j\) coincides along \(Z(\cdot,\omega)\) with the affine map
\(\ell_j\), whose second derivatives vanish; hence, by
Bouleau~\cite[Proposition~4]{bouleau1984formules}, for every
\(t\in[a,b]\),
\[
A_j(t,\omega)-A_j(a,\omega)
=
\frac12\sum_{r,s=1}^{d_Z}
\int_a^t
\partial^2_{rs}F_j(Z(u,\omega))\,
d\langle Z_r,Z_s\rangle(u,\omega)
=0.
\]
Thus \(A_j(\cdot,\omega)\) is constant on each \(I_n\), so
\(|dA_j|_\omega(I_n)=0\) for every \(n\), and by countable
subadditivity
\(|dA_j|_\omega(O_\omega)=0\), which is
\eqref{eq:A-vanishes-on-U} since
\(\delta(t,\omega)=\mathbf1_{O_\omega}(t)\).
\endgroup

Finally, using successively \eqref{eq:Fstar-on-U} and
\eqref{eq:A-vanishes-on-U},
\[
\int_0^T\delta(t)\Phi(t)\,dY_j(t)
=
\int_0^T\delta(t)\Phi(t)\,dZ_{p+j}(t)
=
\int_0^T\delta(t)\Phi(t)\,dX_j(t),
\]
which is \eqref{eq:pathwise-identity}.
\qed

\subsection{Proof of Proposition~\ref{prop:positive-denominator}}
\label{Sec:pf-positive-denominator}
Set \(\Theta_I:=\{a\in\Theta:I\subset\operatorname{Int}(D(a))\}\).

We first note that \(\Theta_I\) is open. Fix \(a\in\Theta_I\). For each
\(x\in I\), the point \((a,x)\) lies in the open set \(\mathcal U_D\),
so there are open neighborhoods \(O_x\subset\Theta\) of \(a\) and
\(V_x\subset\mathbb R^d\) of \(x\) with
\(O_x\times V_x\subset\mathcal U_D\). By compactness, pick
\(x_1,\ldots,x_k\in I\) with \(I\subset\bigcup_{\ell=1}^k V_{x_\ell}\);
then \(O:=\bigcap_{\ell=1}^k O_{x_\ell}\) is a neighborhood of \(a\)
satisfying \(O\times I\subset\mathcal U_D\), that is,
\(O\subset\Theta_I\).

By Assumption~\ref{ass:censoring-nondegeneracy},
\(\mathbb P(\theta(t_\star)\in\Theta_I)>0\). Since \(\Theta_I\) is open,
there exist a compact set \(C_I\subset\Theta_I\) and \(\eta>0\) such that
\(\mathbb P(\theta(t_\star)\in C_I)>0\) and
\(C_I^\eta:=\{a\in\Theta:\operatorname{dist}(a,C_I)\le\eta\}
\subset\Theta_I\). By continuity of \(\theta\), there exists
\(\varepsilon_I>0\) such that
\[
\pi_I:=
\mathbb P\Bigl(
\theta(t_\star)\in C_I,\
\sup_{\substack{t\in[0,T]\\ |t-t_\star|\le\varepsilon_I}}
|\theta(t)-\theta(t_\star)|\le\eta
\Bigr)>0.
\]
Set \(t_I:=\min\{T,t_\star+\varepsilon_I\}\). On this event,
\(\theta(t)\in C_I^\eta\subset\Theta_I\) whenever \(t\in[0,T]\) and
\(|t-t_\star|\le\varepsilon_I\), whence, for every \(x\in I\),
\(\overline G_t(x)
=\mathbb P\bigl(x\in\operatorname{Int}(D(\theta(t)))\bigr)
\ge\pi_I\) for
\(t\in[\max\{0,t_\star-\varepsilon_I\},t_I]\).
Finally, for \(t_0\in(0,t_I)\) and \(x\in I\),
Proposition~\ref{ass:gaussian-density} gives
\[
\begin{aligned}
\overline f(x)
&\ge
\frac{1}{T-t_0}\int_{\max\{t_0,t_\star-\varepsilon_I\}}^{t_I}
\overline G_t(x)\,p_t(x_0,x)\,dt
\\
&\ge
\frac{t_I-\max\{t_0,t_\star-\varepsilon_I\}}{T-t_0}\,\pi_I
\inf_{t\in[\max\{t_0,t_\star-\varepsilon_I\},t_I]}
\inf_{z\in I}p_t(x_0,z)
>0,
\end{aligned}
\]
which proves \eqref{eq:mI-positive}.

\qed
\section{Proofs for Section~\ref{sec:kernel-estimator}}
\label{app:proofs-kernel}

Throughout the proofs, we repeatedly use the fact that, by translation
invariance and a change of variables, for every \(y\in\mathbb R^d\) and
\(\boldsymbol h\in(0,1]^d\),
{\small
\begin{equation}\label{eq:kernel-L2-norm}
\int_{\mathbb R^d}K_{\boldsymbol h}(y-x)^2\,dx
=\|K_{\boldsymbol h}\|_2^2
=\Bigl(\prod_{j=1}^d h_j\Bigr)^{-1}\|K\|_2^2.
\end{equation}
}

\subsection{Proof of Proposition~\ref{prop:fixed-bandwidth-risk}}
\label{pf:prop:fixed-bandwidth-risk}

We first prove \eqref{eq:risk-den-aniso}. For a single trajectory, set
\[
F_{\boldsymbol h'}(x)
:=
\frac{1}{T-t_0}\int_{t_0}^T\delta(t)K_{\boldsymbol h'}(Y(t)-x)\,dt,
\]
so that \(\widehat f_{N,\boldsymbol h'}=N^{-1}\sum_{i=1}^N
F^{(i)}_{\boldsymbol h'}\) with i.i.d.\
\(L^2(\mathbb R^d;\mathbb R)\)-valued summands, and, by
\eqref{eq:expectation-fhat},
\(\mathbb E[\widehat f_{N,\boldsymbol h'}]=K_{\boldsymbol h'}*\overline f\).
The bias-variance decomposition and independence give
\begin{align*}
    \mathbb E\bigl[\|\widehat f_{N,\boldsymbol h'}-\overline f\|_2^2\bigr]
=
\|K_{\boldsymbol h'}*\overline f-\overline f\|_2^2
+
\frac1N\,
\mathbb E\bigl[\|F_{\boldsymbol h'}
-\mathbb E[F_{\boldsymbol h'}]\|_2^2\bigr]
\le
\|K_{\boldsymbol h'}*\overline f-\overline f\|_2^2
+
\frac1N\,\mathbb E\bigl[\|F_{\boldsymbol h'}\|_2^2\bigr].
\end{align*}
By the Cauchy--Schwarz inequality,
Fubini's theorem, \(\delta\le1\), and \eqref{eq:kernel-L2-norm},
\[
\mathbb E\bigl[\|F_{\boldsymbol h'}\|_2^2\bigr]
\le
\frac{1}{T-t_0}\int_{t_0}^T
\mathbb E\Bigl[\int_{\mathbb R^d}
K_{\boldsymbol h'}(X(t)-x)^2\,dx\Bigr]dt
=
\Bigl(\prod_{j=1}^d h'_j\Bigr)^{-1}\|K\|_2^2,
\]
which proves \eqref{eq:risk-den-aniso}.

We now prove \eqref{eq:risk-num-aniso}. For a single trajectory, define
the \(L^2(\mathbb R^d;\mathbb R^d)\)-valued random function
\[
Z_{\boldsymbol h}(x)
:=
\frac{1}{T-t_0}\int_{t_0}^T
\delta(t)K_{\boldsymbol h}(Y(t)-x)\,dY(t),
\]
so that \(\widehat{bf}_{N,\boldsymbol h}
=N^{-1}\sum_{i=1}^N Z^{(i)}_{\boldsymbol h}\) with i.i.d.\ summands
and, by \eqref{eq:expectation-bfhat},
\(\mathbb E[\widehat{bf}_{N,\boldsymbol h}]
=K_{\boldsymbol h}*(b\overline f)\). As above,
\[
\mathbb E\bigl[\|\widehat{bf}_{N,\boldsymbol h}-b\overline f\|_2^2\bigr]
\le
\|K_{\boldsymbol h}*(b\overline f)-b\overline f\|_2^2
+
\frac1N\,\mathbb E\bigl[\|Z_{\boldsymbol h}\|_2^2\bigr].
\]
Applying Theorem~\ref{thm:semimartingale-identity} with
\(\Phi(t)=K_{\boldsymbol h}(Y(t)-x)\), using the definition of
\(\delta(t)\), and then \eqref{eq:SDE_X},
\begin{align}\label{eq:DM-decomposition}
Z_{\boldsymbol h}(x)
&=
\frac{1}{T-t_0}\int_{t_0}^T
\delta(t)K_{\boldsymbol h}(X(t)-x)\,b(X(t))\,dt
+
\frac{1}{T-t_0}\int_{t_0}^T
\delta(t)K_{\boldsymbol h}(X(t)-x)\,\Sigma(X(t))\,dW(t) \nonumber \\& =:D_{\boldsymbol h}(x)+ M_{\boldsymbol h}(x),
\end{align}
and \(\mathbb E[\|Z_{\boldsymbol h}\|_2^2]
\le2\,\mathbb E[\|D_{\boldsymbol h}\|_2^2]
+2\,\mathbb E[\|M_{\boldsymbol h}\|_2^2]\).
For the drift part, \begingroup the same argument as for
\(F_{\boldsymbol h'}\), followed by
\eqref{def:uncensoring-prob}, gives\endgroup
\begin{equation}\label{eq:drift-bound}
\mathbb E\bigl[\|D_{\boldsymbol h}\|_2^2\bigr]
\le
\frac{\bigl(\prod_{j=1}^d h_j\bigr)^{-1}\|K\|_2^2}{T-t_0}
\int_{t_0}^T
\mathbb E\bigl[\delta(t)|b(X(t))|^2\bigr]dt
=
\Bigl(\prod_{j=1}^d h_j\Bigr)^{-1}\|K\|_2^2\,
\|b\|_{2,\overline f}^2.
\end{equation}
For the martingale part, \begingroup Fubini's theorem and
Itô's isometry, applied componentwise, together with
\(\sum_{j=1}^d|\Sigma_{j\cdot}(x)|^2
=\operatorname{Tr}(\Sigma\Sigma^\top)(x)\),
\eqref{eq:kernel-L2-norm}, and
\eqref{def:uncensoring-prob}, give\endgroup
\begin{align}\label{eq:martingale-bound}
\mathbb E\bigl[\|M_{\boldsymbol h}\|_2^2\bigr]
&=
\frac{\bigl(\prod_{j=1}^d h_j\bigr)^{-1}\|K\|_2^2}{(T-t_0)^2}
\int_{t_0}^T
\mathbb E\bigl[\delta(t)
\operatorname{Tr}(\Sigma\Sigma^\top)(X(t))\bigr]dt \nonumber
\\&=
\frac{\bigl(\prod_{j=1}^d h_j\bigr)^{-1}\|K\|_2^2}{T-t_0}\,
\bigl\|\sqrt{\operatorname{Tr}(\Sigma\Sigma^\top)}
\bigr\|_{2,\overline f}^2.
\end{align}
Combining \eqref{eq:drift-bound} and \eqref{eq:martingale-bound}
proves \eqref{eq:risk-num-aniso}.

It remains to prove \eqref{eq:risk-ratio-aniso}. Let
\(J(x):=\mathbf1_{\{\widehat f_{N,\boldsymbol h'}(x)>m_I/2\}}\).
For \(x\in I\),
\[
\widehat b_{N,\boldsymbol h,\boldsymbol h'}(x)-b(x)
=
J(x)\,
\frac{\widehat{bf}_{N,\boldsymbol h}(x)-b(x)\overline f(x)}
{\widehat f_{N,\boldsymbol h'}(x)}
+
J(x)\,b(x)\,
\frac{\overline f(x)-\widehat f_{N,\boldsymbol h'}(x)}
{\widehat f_{N,\boldsymbol h'}(x)}
-
(1-J(x))\,b(x).
\]
On \(\{J=1\}\),
\(\widehat f_{N,\boldsymbol h'}(x)^{-2}\le4m_I^{-2}\); on \(\{J=0\}\),
since \(\overline f(x)\ge m_I\) on \(I\),
\(1\le4m_I^{-2}
\bigl(\widehat f_{N,\boldsymbol h'}(x)-\overline f(x)\bigr)^2\).
\begingroup Since the terms multiplied by \(J(x)\) and
\(1-J(x)\) are supported on disjoint events, the inequality
\(|u+v|_d^2\le2|u|_d^2+2|v|_d^2\) yields, for every \(x\in I\),\endgroup
\begin{equation}\label{eq:risk-ratio-fixed}
\bigl|\widehat b_{N,\boldsymbol h,\boldsymbol h'}(x)-b(x)\bigr|^2
\le
\frac{8}{m_I^2}
\bigl|\widehat{bf}_{N,\boldsymbol h}(x)-b(x)\overline f(x)\bigr|_d^2
+
\frac{12}{m_I^2}\,|b(x)|^2
\bigl(\widehat f_{N,\boldsymbol h'}(x)-\overline f(x)\bigr)^2.
\end{equation}
\begingroup Integrating over \(I\), taking expectations, and
using \(8\le\mathrm C_{\mathrm{NW}}\) and
\(12\|b\|_{\infty,I}^2\le\mathrm C_{\mathrm{NW}}\) proves
\eqref{eq:risk-ratio-aniso} and completes the proof.\endgroup
\qed
\subsection{Auxiliary tools for the adaptive analysis}
\label{app:adaptive-auxiliary-tools}

\subsubsection{Elementary counting bounds}
\label{sec:app:product-counting}
The following elementary estimates, used to verify the size and
summability conditions on the bandwidth grids of
Example~\ref{ex:admissible-grids}, follow from standard counting
arguments; see Stanley~\cite[Chapter~1,
Sections~1.1--1.2]{stanley2012enumerative}.

\begin{lemma}
\label{lem:product-index-counting}
For every fixed \(d\in\mathbb N\), the following bounds hold.
\begin{enumerate}[label=(\roman*), leftmargin=*]
\item For every \(s\in\mathbb N_0\),
\begin{equation}
\label{eq:count-fixed-s}
\#\bigl\{
\ell\in\mathbb N_0^d:
\ell_1+\cdots+\ell_d=s
\bigr\}
=
\binom{s+d-1}{d-1},
\end{equation}
and consequently, for every \(J\in\mathbb N_0\),
\(
\sum_{s=0}^{J}
\#\bigl\{
\ell\in\mathbb N_0^d:
\ell_1+\cdots+\ell_d=s
\bigr\}
=
\binom{J+d}{d}
\le C_d(J+1)^d.
\)
\item For every \(N\ge2\),
\(
\#\bigl\{
k\in\mathbb N^d:
\prod_{j=1}^{d} k_j\le N
\bigr\}
\le
C_dN(\log N)^{d-1}.
\)
\item For every \(m\in\mathbb N\),
\(
\#\bigl\{
k\in\mathbb N^d:
\prod_{j=1}^{d} k_j=m
\bigr\}
\le
m^{d-1}.
\)
\end{enumerate}
\end{lemma}

\begin{proof}
(i) The identity \eqref{eq:count-fixed-s} is the classical
stars-and-bars count. A vector \(\ell\in\mathbb N_0^d\) with
\(\ell_1+\cdots+\ell_d=s\) corresponds bijectively to a placement of
\(d-1\) separators among \(s\) unit symbols. For the second assertion, appending the extra coordinate
\(\ell_{d+1}:=J-(\ell_1+\cdots+\ell_d)\) puts
\(\{\ell\in\mathbb N_0^d:\ell_1+\cdots+\ell_d\le J\}\) in bijection
with \(\{\ell\in\mathbb N_0^{d+1}:\ell_1+\cdots+\ell_{d+1}=J\}\), so
that \eqref{eq:count-fixed-s}, applied with \(d+1\) in place of \(d\)
and \(s=J\), gives the count \(\binom{J+d}{d}\). Finally,
\(
\binom{J+d}{d}
=\frac{(J+1)\cdots(J+d)}{d!}
\le\frac{d^d}{d!}(J+1)^d
=:C_d(J+1)^d.
\)

(ii) Summing over the free coordinates \(k_1,\ldots,k_{d-1}\),
\[
\#\bigl\{k\in\mathbb N^d:\prod_{j=1}^{d} k_j\le N\bigr\}
=
\sum_{k_1=1}^{N}\cdots\sum_{k_{d-1}=1}^{N}
\Bigl\lfloor\frac{N}{\prod_{j=1}^{d-1} k_j}\Bigr\rfloor
\le
N\Bigl(\sum_{k=1}^{N}\frac1k\Bigr)^{d-1}
\le
N(1+\log N)^{d-1}.
\]

(iii) Each of \(k_1,\ldots,k_{d-1}\) lies in \(\{1,\ldots,m\}\), and
they determine \(k_d\) through \(\prod_{j=1}^{d} k_j=m\) whenever an integer
solution exists; hence there are at most \(m^{d-1}\) admissible
vectors.
\end{proof}

\subsubsection{Probabilistic results}
\label{app:probabilistic-results}

We collect three probabilistic tools used in the proofs.
The first result is a stochastic Fubini theorem for stochastic
integrals with respect to continuous semimartingales. The version
stated by Revuz and Yor~\cite[Chap.~IV, Sec.~5, Ex.~(5.17),
pp.~175--176]{revuz2013continuous} is formulated for a bounded measure;
applying it to the Jordan decomposition of a finite signed measure and
using the linearity of Lebesgue and stochastic integrals yields the
following extension, which is the form used to justify the convolution
commutation in Lemma~\ref{lem:GL-commute}.

\begin{theorem}
\label{thm:stoch-fubini}
Let \((\Omega,\mathcal F,(\mathcal F(t))_{t\ge0},\mathbb P)\) satisfy
the usual conditions, let \(X=(X(t))_{0\le t\le T}\) be a continuous
semimartingale, let \((S,\mathcal S)\) be a measurable space, and let
\(\nu\) be a finite signed measure on \((S,\mathcal S)\). Assume that
\(H:S\times[0,T]\times\Omega\to\mathbb R\) is bounded and
\(\mathcal S\otimes\mathcal P\)-measurable, where \(\mathcal P\)
denotes the predictable sigma-field associated with
\((\mathcal F(t))_{0\le t\le T}\). Then there exists an
\(\mathcal S\otimes\mathcal P\)-measurable process
\(I:S\times[0,T]\times\Omega\to\mathbb R\) such that, up to
indistinguishability, for every \(s\in S\),
{\small
\[
I(s,t)=\int_0^t H(s,u)\,dX(u),
\qquad 0\le t\le T,
\]
}
and
{\small
\[
\int_S I(s,t)\,\nu(ds)
=
\int_0^t\Bigl(\int_S H(s,u)\,\nu(ds)\Bigr)dX(u),
\qquad 0\le t\le T.
\]
}
\end{theorem}

The second result is an \(L^2\)-valued BDG inequality that controls stochastic
integrals as random elements of \(L^2(A;\mathbb R^{d_2})\); it is a
standard consequence of infinite-dimensional stochastic integration theory,
see Gawarecki and Mandrekar~\cite[Lemma~3.1,
Eq.~(3.21)]{gawarecki2011stochastic}.

\begin{theorem}
\label{thm:L2-valued-BDG}
Let \(A\subset\mathbb R^{d_1}\) be measurable, let
\(W=(W_1,\ldots,W_q)\) be a \(q\)-dimensional Brownian motion, and let
\(\Phi:[0,T]\times\Omega\to L^2(A;\mathbb R^{d_2\times q})\) be
predictable, identified with the linear operator from \(\mathbb R^q\)
to \(L^2(A;\mathbb R^{d_2})\) given by
\(z\mapsto\sum_{k=1}^q z_k\,\Phi_{\cdot k}(t,\omega,\cdot)\), where
\(\Phi_{\cdot k}\) denotes the \(k\)-th column of \(\Phi\). Assume
that, for some \(p\ge2\),
{\small
\[
\mathbb E\biggl[
\Bigl(\int_0^T\!\!\int_A\|\Phi(t,x)\|_F^2\,dx\,dt\Bigr)^{p/2}
\biggr]<\infty.
\]
}
For \(0\le t\le T\), define the \(L^2(A;\mathbb R^{d_2})\)-valued
stochastic integral \(M(t):=\int_0^t\Phi(s)\,dW(s)\), characterized by
{\small
\[
\langle M(t),g\rangle_{2,A}
=
\sum_{k=1}^q
\int_0^t\langle\Phi_{\cdot k}(s,\cdot),g\rangle_{2,A}\,dW_k(s),
\qquad g\in L^2(A;\mathbb R^{d_2}).
\]
}
Then \(M=(M(t))_{0\le t\le T}\) is an
\(L^2(A;\mathbb R^{d_2})\)-valued continuous martingale, and
{\small
\[
\mathbb E\Bigl[\sup_{0\le t\le T}\|M(t)\|_{2,A}^p\Bigr]
\le
C_p\,
\mathbb E\biggl[
\Bigl(\int_0^T\!\!\int_A\|\Phi(t,x)\|_F^2\,dx\,dt\Bigr)^{p/2}
\biggr],
\]
}
where \(\|\cdot\|_F\) is the Frobenius norm and \(C_p<\infty\) depends
only on \(p\).
\end{theorem}

The third result is a Talagrand-type inequality for empirical processes indexed by a countable
class of real-valued functionals on a Hilbert space.
\begingroup 
It follows from integrating the tail bound in the concentration
inequality of Klein and Rio~\cite{klein2005concentration}.
\endgroup

\begin{theorem}
\label{lem:talagrand-hilbert}
Let \((\mathcal H,\langle\cdot,\cdot\rangle_{\mathcal H},
\|\cdot\|_{\mathcal H})\) be a separable Hilbert space, let
\(X_1,\ldots,X_n\) be independent \(\mathcal H\)-valued random
variables, and let \(\mathcal F\) be an at most countable class of
measurable functions \(f:\mathcal H\to\mathbb R\). Define
\[
v_n(f)
:=
\frac1n\sum_{i=1}^n
\bigl(f(X_i)-\mathbb E[f(X_i)]\bigr),
\qquad f\in\mathcal F,
\]
and assume that there exist constants \(M,H,v>0\) such that
\[
\sup_{f\in\mathcal F}\|f\|_{\infty}\le M,\qquad
\mathbb E\Bigl[\sup_{f\in\mathcal F}|v_n(f)|\Bigr]\le H,\qquad
\sup_{f\in\mathcal F}\frac1n\sum_{i=1}^n
\operatorname{Var}(f(X_i))\le v.
\]
Then, for every \(\alpha>0\), setting
\(C_\alpha:=(\sqrt{1+\alpha}-1)\wedge1\), there exist universal
positive constants \(c_1,c_2,c_3\) such that
\[
\mathbb E\Bigl[
\Bigl(\sup_{f\in\mathcal F}|v_n(f)|^2-2(1+2\alpha)H^2\Bigr)_+
\Bigr]
\le
c_1\biggl[
\frac{v}{n}\exp\Bigl(-c_2\alpha\,\frac{nH^2}{v}\Bigr)
+
\frac{M^2}{n^2C_\alpha^2}
\exp\Bigl(-c_3C_\alpha\sqrt{\alpha}\,\frac{nH}{M}\Bigr)
\biggr].
\]
\end{theorem}

\subsubsection{Concentration for Hilbert-valued empirical means}
\label{app:hilbert-valued-concentration}

\begingroup 
Establishing the stochastic bound required for the adaptive analysis involves more than a pointwise variance calculation. One needs to control simultaneously, over a finite index set, a family of Hilbert-space-valued empirical means. We first present a general deviation bound used below.

Let \(\mathbb H\) be a separable real Hilbert space, let \(\mathcal S_0\) be a countable symmetric subset dense in its unit sphere, that is, \(\psi\in\mathcal S_0\) implies \(-\psi\in\mathcal S_0\). For every \(N\ge2\), let \(\Lambda_N\) be a finite index set. For each \(\lambda\in\Lambda_N\), let \(Z_\lambda\) be a measurable \(\mathbb H\)-valued random element, and let \(Z_\lambda^{(1)},\ldots,Z_\lambda^{(N)}\) be independent copies of \(Z_\lambda\). Set
{\small
\[
U_{N,\lambda}:=\frac1N\sum_{i=1}^N
\bigl(Z_\lambda^{(i)}-\mathbb E[Z_\lambda]\bigr).
\]
}
\begin{proposition}
\label{prop:hilbert-uniform-deviation}
Assume that (i) for some \(r_0>2\), \(\mathrm C_{r_0}<\infty\), and positive numbers \(s_\lambda,v_\lambda\),
{\small
\begin{equation}
\label{eq:abstract-H-moments}
\mathbb E[\|Z_\lambda\|_{\mathbb H}^2]\le s_\lambda^2,
\qquad
\sup_{\psi\in\mathcal S_0}\mathbb E[\langle Z_\lambda,\psi\rangle_{\mathbb H}^2]\le v_\lambda,
\qquad
\mathbb E[\|Z_\lambda\|_{\mathbb H}^{r_0}]\le \mathrm C_{r_0}s_\lambda^{r_0}.
\end{equation}
}
 (ii) for some positive sequence \(L_N\) and a constant \(\mathrm C_V(\mu)<\infty\)independent of \(N\),
 {\small
\begin{equation}
\label{eq:abstract-H-complexity}
\sum_{\lambda\in\Lambda_N}s_\lambda^2\le L_N,
\qquad
\sum_{\lambda\in\Lambda_N}v_\lambda
\exp\!\left(-\mu\frac{s_\lambda^2}{v_\lambda}\right)
\le \mathrm C_V(\mu),\quad \mu>0.
\end{equation}
}
Then, for every \(\kappa\ge12\), there exists a finite constant \(\mathrm C\), independent of \(N\), such that
{\small
\begin{equation}
\label{eq:abstract-H-conclusion}
\mathbb E\!\left[\sup_{\lambda\in\Lambda_N}
\left(\|U_{N,\lambda}\|_{\mathbb H}^2-\frac{\kappa s_\lambda^2}{N}\right)_+\right]
\le \mathrm C\!\left\{\frac1N
+\frac{L_N(\log N)^{r_0-2}}{N^{(r_0-2)/2}}\right\}.
\end{equation}
}
If, in addition, \(\sup_{\lambda\in\Lambda_N}v_\lambda\le v_0<\infty\), the second condition in \emph{(ii)} may be replaced by
\(
\sum_{\lambda\in\Lambda_N}e^{-\mu s_\lambda^2}
\le\widetilde{\mathrm C}_V(\mu)
\)
for every \(\mu>0\), with \(\widetilde{\mathrm C}_V(\mu)\) independent of \(N\).
\end{proposition}
\endgroup

\begin{proof}

Since \(\mathbb H\) is separable, its unit sphere is separable. Hence, for every \(u\in\mathbb H\),
{\small
\[
\|u\|_{\mathbb H}
=
\sup_{\psi\in\mathcal S_0}
|\langle u,\psi\rangle_{\mathbb H}|.
\]
}
The countability of \(\mathcal S_0\) and the finiteness of \(\Lambda_N\) ensure that all the suprema below are measurable. Fix \(\lambda\in\Lambda_N\). In order to control
\(\|U_{N,\lambda}\|_{\mathbb H}\), choose a cutoff
\(M_{N,\lambda}>0\), to be specified below, and set
{\small
\[
\widetilde Z_\lambda
:=
Z_\lambda\mathbf1_{\{\|Z_\lambda\|_{\mathbb H}\le M_{N,\lambda}\}},
\qquad
\widetilde U_{N,\lambda}
:=
\frac1N\sum_{i=1}^N
\bigl(\widetilde Z_\lambda^{(i)}-\mathbb E[\widetilde Z_\lambda]\bigr).
\]
}
We apply Theorem~\ref{lem:talagrand-hilbert} to the independent
random elements
\(
Z_\lambda^{(1)},\ldots,Z_\lambda^{(N)}
\)
and to the countable class of measurable functionals
{
\[
\mathcal F_{N,\lambda}
:=
\left\{
f_{\psi}(z)
=
\langle z,\psi\rangle_{\mathbb H}
\mathbf 1_{\{\|z\|_{\mathbb H}\le M_{N,\lambda}\}}
:
\psi\in\mathcal S_0
\right\}.
\]
}
For every \(\psi\in\mathcal S_0\),
\(
\frac1N\sum_{i=1}^N
\left(
f_{\psi}(Z_\lambda^{(i)})
-
\mathbb E\!\left[
f_{\psi}(Z_\lambda)
\right]
\right)
=
\left\langle
\widetilde U_{N,\lambda},\psi
\right\rangle_{\mathbb H}.
\)
Since \(\mathcal S_0\) is symmetric and dense in the unit sphere of \(\mathbb H\),
it follows that
{\small
\[
\sup_{f\in\mathcal F_{N,\lambda}}
\frac1N\sum_{i=1}^N
\left(
f(Z_\lambda^{(i)})-\mathbb E[f(Z_\lambda)]
\right)
=
\sup_{\psi\in\mathcal S_0}
\left\langle
\widetilde U_{N,\lambda},\psi
\right\rangle_{\mathbb H}
=
\sup_{\psi\in\mathcal S_0}
\left|
\left\langle
\widetilde U_{N,\lambda},\psi
\right\rangle_{\mathbb H}
\right|
=
\|\widetilde U_{N,\lambda}\|_{\mathbb H}.
\]
}
We now verify the three parameters appearing in
Theorem~\ref{lem:talagrand-hilbert}. First, since
\(\|\psi\|_{\mathbb H}=1\), the Cauchy--Schwarz inequality gives
\[
\left|f_{\psi}(z)\right|
\le
\|z\|_{\mathbb H}
\mathbf 1_{\{\|z\|_{\mathbb H}\le M_{N,\lambda}\}}
\le M_{N,\lambda}.
\]
Next, by 
the Cauchy--Schwarz inequality, and independence,
{\small
\begin{align*}
\mathbb E\!\left[
\sup_{f\in\mathcal F_{N,\lambda}}
\frac1N\sum_{i=1}^N
\left(
f(Z_\lambda^{(i)})-\mathbb E[f(Z_\lambda)]
\right)
\right]
&=
\mathbb E\!\left[
\|\widetilde U_{N,\lambda}\|_{\mathbb H}
\right] \le
\left(
\mathbb E\!\left[
\|\widetilde U_{N,\lambda}\|_{\mathbb H}^2
\right]
\right)^{1/2} = \left(
\frac1N
\mathbb E\!\left[
\left\|
\widetilde Z_\lambda
-
\mathbb E[\widetilde Z_\lambda]
\right\|_{\mathbb H}^2
\right]
\right)^{1/2}  \\& \le  \left( \frac1N
\mathbb E\!\left[
\|\widetilde Z_\lambda\|_{\mathbb H}^2
\right]
\right)^{1/2} \le  \left( \frac1N
\mathbb E\!\left[
\|Z_\lambda\|_{\mathbb H}^2
\right] 
\right)^{1/2} \le  \frac{s_\lambda}{\sqrt N} := H_{N,\lambda},
\end{align*}
}
where the last inequality follows from
\eqref{eq:abstract-H-moments}. Finally, for every
\(f_{\psi}\in\mathcal F_{N,\lambda}\),
\begin{align*}
\operatorname{Var}\!\left(
f_{\psi}(Z_\lambda)
\right)
&\le
\mathbb E\!\left[
\langle Z_\lambda,\psi\rangle_{\mathbb H}^2
\mathbf 1_{\{\|Z_\lambda\|_{\mathbb H}
\le M_{N,\lambda}\}}
\right] \le
\mathbb E\!\left[
\langle Z_\lambda,\psi\rangle_{\mathbb H}^2
\right]
\le v_\lambda,
\end{align*}
again by \eqref{eq:abstract-H-moments}.
Consequently, Theorem~\ref{lem:talagrand-hilbert} applies with
\(
M=M_{N,\lambda},
\,
H=H_{N,\lambda}=\frac{s_\lambda}{\sqrt N},
\,
v=v_\lambda.
\)
Taking \(\alpha=1\) in Theorem~\ref{lem:talagrand-hilbert} gives
{\small
\begin{align}
&\mathbb E\!\left[
\left(
\|\widetilde U_{N,\lambda}\|_{\mathbb H}^2
-6H_{N,\lambda}^2
\right)_+\right]
\le
\mathrm C\left[
\frac{v_\lambda}{N}
\exp\!\left(-\mathrm c
\frac{NH_{N,\lambda}^2}{v_\lambda}\right)
+
\frac{M_{N,\lambda}^2}{N^2}
\exp\!\left(-\mathrm c
\frac{NH_{N,\lambda}}{M_{N,\lambda}}\right)
\right].
\label{eq:abstract-talagrand-proof}
\end{align}
}
Choose
{\small
\[
M_{N,\lambda}
:=
\frac{\tau NH_{N,\lambda}}{\log N}
=
\frac{\tau\sqrt N\,s_\lambda}{\log N},
\]
}
where \(\tau=\tau(r_0)>0\) is sufficiently small. Then
\(NH_{N,\lambda}/M_{N,\lambda}=\log N/\tau\), and \(\tau\) can be chosen so that the second exponential in \eqref{eq:abstract-talagrand-proof} is bounded by \(N^{-(r_0-2)/2}\). Since
\(NH_{N,\lambda}^2=s_\lambda^2\), we obtain
{\small
\begin{equation}
\label{eq:abstract-truncated-bound}
\mathbb E\!\left[
\left(
\|\widetilde U_{N,\lambda}\|_{\mathbb H}^2
-\frac{6s_\lambda^2}{N}
\right)_+\right]
\le
\mathrm C\left[
\frac{v_\lambda}{N}
\exp\!\left(-\mathrm c\frac{s_\lambda^2}{v_\lambda}\right)
+
\frac{s_\lambda^2}{N^{r_0/2}(\log N)^2}
\right].
\end{equation}
}
It remains to control the truncation remainder. Write
\begin{align}
    &R_{N,\lambda}:=U_{N,\lambda}-\widetilde U_{N,\lambda},
\qquad
B_\lambda
:=
\|Z_\lambda\|_{\mathbb H}
\mathbf1_{\{\|Z_\lambda\|_{\mathbb H}>M_{N,\lambda}\}}.
\\& \|R_{N,\lambda}\|_{\mathbb H}
\le
\frac1N\sum_{i=1}^N B_\lambda^{(i)}+\mathbb E[B_\lambda],
\qquad
\mathbb E[\|R_{N,\lambda}\|_{\mathbb H}^2]
\le4\mathbb E[B_\lambda^2].
\end{align}
By \eqref{eq:abstract-H-moments},
\begin{align*}
\mathbb E[B_\lambda^2]
&=
\mathbb E\!\left[
\|Z_\lambda\|_{\mathbb H}^2
\mathbf1_{\{\|Z_\lambda\|_{\mathbb H}>M_{N,\lambda}\}}
\right]\le
M_{N,\lambda}^{-(r_0-2)}
\mathbb E[\|Z_\lambda\|_{\mathbb H}^{r_0}]
\le
\mathrm C(\log N)^{r_0-2}
N^{-(r_0-2)/2}s_\lambda^2.
\end{align*}
Moreover, for every \(\kappa\ge12\),
\begin{align*}
\left(
\|U_{N,\lambda}\|_{\mathbb H}^2
-\frac{\kappa s_\lambda^2}{N}
\right)_+
&\le
2\left(
\|\widetilde U_{N,\lambda}\|_{\mathbb H}^2
-\frac{\kappa s_\lambda^2}{2N}
\right)_+
+2\|R_{N,\lambda}\|_{\mathbb H}^2\le
2\left(
\|\widetilde U_{N,\lambda}\|_{\mathbb H}^2
-\frac{6s_\lambda^2}{N}
\right)_+
+2\|R_{N,\lambda}\|_{\mathbb H}^2.
\end{align*}
Taking the supremum over \(\lambda\in\Lambda_N\), using
\((\sup_\lambda a_\lambda)_+\le\sum_\lambda(a_\lambda)_+\), and applying \eqref{eq:abstract-truncated-bound} and \eqref{eq:abstract-H-complexity}, we get
{\small
\begin{align*}
&\mathbb E\!\left[
\sup_{\lambda\in\Lambda_N}
\left(
\|U_{N,\lambda}\|_{\mathbb H}^2
-\frac{\kappa s_\lambda^2}{N}
\right)_+\right]\quad\le
\mathrm C\left[
\frac1N
+
\frac{1}{N^{r_0/2}(\log N)^2}
\sum_{\lambda\in\Lambda_N}s_\lambda^2
+
(\log N)^{r_0-2}N^{-(r_0-2)/2}
\sum_{\lambda\in\Lambda_N}s_\lambda^2
\right].
\end{align*}
}
Since, for every \(N\ge2\),
\[
\frac{1}{N^{r_0/2}(\log N)^2}
\le
\mathrm C_{r_0}\frac{(\log N)^{r_0-2}}{N^{(r_0-2)/2}},
\]
the second term is absorbed by the third one, which gives
\eqref{eq:abstract-H-conclusion}.

Finally, if \(v_\lambda\le v_0\) for every \(\lambda\), then
\(
v_\lambda
\exp\!\left(-\mu\frac{s_\lambda^2}{v_\lambda}\right)
\le
v_0\exp\!\left(-\frac{\mu}{v_0}s_\lambda^2\right).
\)
Thus the simpler condition stated in the proposition implies the second condition in \emph{(ii)}. This completes the proof.
\end{proof}

\subsection{Proof of Lemma~\ref{lem:GL-commute}}
\label{subsec:pf:GL-commute}
\begingroup Fix \(x\in\mathbb R^d\) and
\(j\in\{1,\ldots,d\}\).\endgroup{} By
Theorem~\ref{thm:semimartingale-identity},
{\small
\[
\big(\widehat{bf}_{N,\boldsymbol h,\boldsymbol\eta}\big)_j(x)
=
\frac{1}{N(T-t_0)}
\sum_{i=1}^N
\int_{\mathbb R^d}
K_{\boldsymbol\eta}(x-y)
\left[
\int_{t_0}^T
\delta^{(i)}(t)
K_{\boldsymbol h}(X^{(i)}(t)-y)\,dX_j^{(i)}(t)
\right]dy.
\]
}
To exchange the \(y\)-integral and the stochastic integral, we apply
Theorem~\ref{thm:stoch-fubini} to the continuous semimartingale
\(X_j^{(i)}\), the finite signed measure
\(\nu_x(dy):=K_{\boldsymbol\eta}(x-y)\,dy\), which satisfies
\(|\nu_x|(\mathbb R^d)\le\|K\|_1<\infty\), and the integrand
\(H_x^{(i)}(y,t,\omega)
:=\delta^{(i)}(t,\omega)K_{\boldsymbol h}(X^{(i)}(t,\omega)-y)\),
which is bounded by \(\|K_{\boldsymbol h}\|_\infty<\infty\) by
Assumption~\ref{ass:kernel}(i). This gives
{\small
\[
\big(\widehat{bf}_{N,\boldsymbol h,\boldsymbol\eta}\big)_j(x)
=
\frac{1}{N(T-t_0)}
\sum_{i=1}^N
\int_{t_0}^T
\delta^{(i)}(t)
\left[
\int_{\mathbb R^d}
K_{\boldsymbol\eta}(x-y)K_{\boldsymbol h}(X^{(i)}(t)-y)\,dy
\right]dX_j^{(i)}(t),
\]
}
and since \(K_{\boldsymbol\eta}\) is even (\(K\) being even), the
change of variables \(u=y-x\) identifies the inner integral as
\((K_{\boldsymbol\eta}*K_{\boldsymbol h})(X^{(i)}(t)-x)\). This proves
the representation of \(\widehat{bf}_{N,\boldsymbol h,\boldsymbol\eta}\)
stated in the lemma; since the resulting expression is, by
commutativity of the convolution, symmetric in
\((\boldsymbol h,\boldsymbol\eta)\), it also gives
\(\widehat{bf}_{N,\boldsymbol h,\boldsymbol\eta}
=\widehat{bf}_{N,\boldsymbol\eta,\boldsymbol h}\).
\begingroup 
For the denominator, the same change of variables and evenness of
\(K\), now with the ordinary Fubini theorem in place of
Theorem~\ref{thm:stoch-fubini}, yield the representation of
\(
\widehat f_{N,\boldsymbol h,\boldsymbol\eta}
\)
stated in the lemma. Commutativity of the convolution again yields
\(\widehat f_{N,\boldsymbol h,\boldsymbol\eta}
=\widehat f_{N,\boldsymbol\eta,\boldsymbol h}\).
\endgroup
\qed
\subsection{Proof of Lemma~\ref{lem:GL-deviation-verification}}
\label{app:GL-deviation-proof}

Let \(\mathbb H=L^2(\mathbb R^d;\mathbb R^d)\), and let
\(\mathcal S_0\subset C_c^\infty(\mathbb R^d;\mathbb R^d)\) be a countable symmetric subset dense in its unit sphere. For \(\boldsymbol\eta
\in\mathcal H_N\), define
{\small
\[
Z_{\boldsymbol\eta}(x)
:=
\frac1{T-t_0}
\int_{t_0}^T
\delta(t)K_{\boldsymbol\eta}(Y(t)-x)\,dY(t),
\qquad
\widehat{bf}_{N,\boldsymbol\eta}
=
\frac1N\sum_{i=1}^N Z_{\boldsymbol\eta}^{(i)}.
\]
}
By Theorem~\ref{thm:semimartingale-identity}(ii),
\(Z_{\boldsymbol\eta}=D_{\boldsymbol\eta}+M_{\boldsymbol\eta}\), where
{\small
\begin{align}
D_{\boldsymbol\eta}(x)
:=
\frac1{T-t_0}\int_{t_0}^T
\delta(t)K_{\boldsymbol\eta}(X(t)-x)b(X(t))dt,\qquad
M_{\boldsymbol\eta}(x)
:=
\frac1{T-t_0}\int_{t_0}^T
\delta(t)K_{\boldsymbol\eta}(X(t)-x)\Sigma(X(t))dW(t).
\label{eq:verification-DM}
\end{align}
}
The first term is a measurable \(\mathbb H\)-valued random element, since
\(
\|K_{\boldsymbol\eta}(X(t)-\cdot)b(X(t))\|_2
=
\|K_{\boldsymbol\eta}\|_2|b(X(t))|
\)
and \(b(X(t))\) has finite moments on \([0,T]\). For the martingale term, define
{\small
\[
\Phi_{\boldsymbol\eta}(t,x)
:=
\frac{\mathbf1_{[t_0,T]}(t)}{T-t_0}
\delta(t)K_{\boldsymbol\eta}(X(t)-x)\Sigma(X(t)).
\]
}
The process \(t\mapsto\Phi_{\boldsymbol\eta}(t,\cdot)\) is predictable with values in
\(L^2(\mathbb R^d;\mathbb R^{d\times q})\), and
\begin{align*}
&\mathbb E\!\left[
\int_0^T\int_{\mathbb R^d}
\|\Phi_{\boldsymbol\eta}(t,x)\|_F^2\,dx\,dt
\right]=
\frac{\|K_{\boldsymbol\eta}\|_2^2}{(T-t_0)^2}
\int_{t_0}^T
\mathbb E\!\left[
\delta(t)\operatorname{Tr}(\Sigma\Sigma^\top)(X(t))
\right]dt
<\infty.
\end{align*}
Thus Theorem~\ref{thm:L2-valued-BDG} defines
\(M_{\boldsymbol\eta}\) as an \(\mathbb H\)-valued random element. We claim that there exist a finite constant \(\mathrm C_w\) and, for every \(r\ge2\), a finite constant \(\mathrm C_r\), independent of \(N\) and \(\boldsymbol\eta\), such that
\begin{align}
\sup_{\psi\in\mathcal S_0}
\mathbb E\!\left[
\langle Z_{\boldsymbol\eta},\psi\rangle_2^2
\right]
&\le \mathrm C_w,
\label{eq:verification-weak-bf}\\
\mathbb E\!\left[\|Z_{\boldsymbol\eta}\|_2^2\right]
&\le
\frac{\mathrm C_{bf}}{\prod_{j=1}^d\eta_j},
\qquad
\mathbb E\!\left[\|Z_{\boldsymbol\eta}\|_2^r\right]
\le
\mathrm C_r\left(
\frac{\mathrm C_{bf}}{\prod_{j=1}^d\eta_j}
\right)^{r/2}.
\label{eq:verification-strong-bf}
\end{align}

We first prove \eqref{eq:verification-weak-bf}. Let
\(\psi\in\mathcal S_0\), and set
\[
\phi_{\boldsymbol\eta,j}(y)
:=
\int_{\mathbb R^d}
K_{\boldsymbol\eta}(y-x)\psi_j(x)\,dx,
\qquad
\phi_{\boldsymbol\eta}(y)
:=
\bigl(\phi_{\boldsymbol\eta,1}(y),\ldots,
\phi_{\boldsymbol\eta,d}(y)\bigr)^\top.
\]
For each coordinate \(j\), the measure \(\psi_j(x)\,dx\) is finite, since \(\psi_j\) is bounded with compact support. Moreover, by Assumption~\ref{ass:kernel}(i), the process
\((x,t,\omega)\mapsto
\delta(t,\omega)K_{\boldsymbol\eta}(X(t,\omega)-x)\)
is bounded. Hence Theorem~\ref{thm:stoch-fubini}, applied componentwise to the semimartingales \(X_j\), gives
\begin{align*}
\langle Z_{\boldsymbol\eta},\psi\rangle_2
&=
\frac1{T-t_0}\sum_{j=1}^d
\int_{t_0}^T
\delta(t)\phi_{\boldsymbol\eta,j}(X(t))\,dX_j(t) \\&=
\frac1{T-t_0}\int_{t_0}^T
\delta(t)\langle b(X(t)),\phi_{\boldsymbol\eta}(X(t))\rangle\,dt +
\frac1{T-t_0}\int_{t_0}^T
\delta(t)\phi_{\boldsymbol\eta}(X(t))^\top
\Sigma(X(t))\,dW(t).
\end{align*}
Young's convolution inequality gives
\begin{equation}
\label{eq:verification-test-L2}
\|\phi_{\boldsymbol\eta}\|_2
\le
\|K_{\boldsymbol\eta}\|_1\|\psi\|_2
=
\|K\|_1.
\end{equation}
By Proposition~\ref{ass:gaussian-density}, the linear growth of \(b\), and \(t_0>0\),
\begin{equation}
\label{eq:verification-density-constants}
\mathrm P_0
:=
\sup_{t\in[t_0,T]}\|p_t(x_0,\cdot)\|_\infty
<\infty,
\qquad
\mathrm B_0
:=
\sup_{t\in[t_0,T]}\sup_{x\in\mathbb R^d}
|b(x)|^2p_t(x_0,x)
<\infty.
\end{equation}
Using \((a+b)^2\le2a^2+2b^2\), Jensen's inequality in time, the It\^o isometry, and
\(
|\Sigma(x)^\top z|^2\le\Lambda_\Sigma|z|^2
\) for some constant $\Lambda_\Sigma$, we obtain
\begin{align*}
\mathbb E\!\left[
\langle Z_{\boldsymbol\eta},\psi\rangle_2^2
\right]
&\le
\frac{2}{T-t_0}\int_{t_0}^T
\mathbb E\!\left[
|b(X(t))|^2|\phi_{\boldsymbol\eta}(X(t))|^2
\right]dt+
\frac{2\Lambda_\Sigma}{(T-t_0)^2}\int_{t_0}^T
\mathbb E\!\left[
|\phi_{\boldsymbol\eta}(X(t))|^2
\right]dt\\
&\le
2\mathrm B_0\|\phi_{\boldsymbol\eta}\|_2^2
+
\frac{2\Lambda_\Sigma\mathrm P_0}{T-t_0}
\|\phi_{\boldsymbol\eta}\|_2^2
\le
\mathrm C_w,
\end{align*}
which proves \eqref{eq:verification-weak-bf}.

We now prove \eqref{eq:verification-strong-bf}. For the drift part in \eqref{eq:verification-DM}, Minkowski's integral inequality, Jensen's inequality in time, and
\eqref{eq:kernel-L2-norm}
give, for every \(r\ge2\),
\begin{align}
\mathbb E\!\left[\|D_{\boldsymbol\eta}\|_2^r\right]
&\le
\mathbb E\!\left[
\left(
\frac1{T-t_0}\int_{t_0}^T
\delta(t)\|K_{\boldsymbol\eta}(X(t)-\cdot)b(X(t))\|_2\,dt
\right)^r\right]
\nonumber\\
&\le
\|K_{\boldsymbol\eta}\|_2^r
\frac1{T-t_0}\int_{t_0}^T
\mathbb E\!\left[|b(X(t))|^r\right]dt
\le
\mathrm C_r
\left(\prod_{j=1}^d\eta_j\right)^{-r/2}.
\label{eq:verification-drift-strong}
\end{align}
For the martingale part, Theorem~\ref{thm:L2-valued-BDG}, followed by Fubini's theorem and translation invariance of the \(L^2\)-norm, yields
\begin{align}
\mathbb E\!\left[\|M_{\boldsymbol\eta}\|_2^r\right]
&\le
\mathrm C_r\mathbb E\!\left[
\left(
\int_{t_0}^T\int_{\mathbb R^d}
\frac{\delta(t)}{(T-t_0)^2}
K_{\boldsymbol\eta}(X(t)-x)^2
\operatorname{Tr}(\Sigma\Sigma^\top)(X(t))\,dx\,dt
\right)^{r/2}\right]
\nonumber\\
&\le
\mathrm C_r
\left(\prod_{j=1}^d\eta_j\right)^{-r/2}
\frac1{T-t_0}\int_{t_0}^T
\mathbb E\!\left[
\bigl(\operatorname{Tr}(\Sigma\Sigma^\top)(X(t))\bigr)^{r/2}
\right]dt\le
\mathrm C_r
\left(\prod_{j=1}^d\eta_j\right)^{-r/2}.
\label{eq:verification-martingale-strong}
\end{align}
The moments in \eqref{eq:verification-drift-strong} and \eqref{eq:verification-martingale-strong} are finite because \(b\) has linear growth, \(X(t)\) has moments of every order on \([0,T]\), and \(\Sigma\) is bounded. For \(r=2\), the first inequality in \eqref{eq:verification-strong-bf} follows directly from \eqref{eq:drift-bound}--\eqref{eq:martingale-bound} and the definition of \(\mathrm C_{bf}\) in \eqref{eq:variance-constants}. For \(r>2\), combining \eqref{eq:verification-drift-strong} and \eqref{eq:verification-martingale-strong}, and enlarging \(\mathrm C_r\) using the fixed positive quantity \(\mathrm C_{bf}\), proves the second inequality in \eqref{eq:verification-strong-bf}. 
\\
\\
We now apply Proposition~\ref{prop:hilbert-uniform-deviation} to \(Z_{\boldsymbol\eta}\), with
\(
s_{\boldsymbol\eta}^2
:=
\mathrm C_{bf}/\prod_{j=1}^d\eta_j,
\,
v_{\boldsymbol\eta}:=\mathrm C_w,
\,
r_0=8,
\,
\Lambda_N=\mathcal H_N.
\)
By Assumption~\ref{ass:GL-grid},
\(
\sum_{\boldsymbol\eta\in\mathcal H_N}s_{\boldsymbol\eta}^2
\le
\mathrm C N^2(\log N)^d
=:L_N,
\)
and the last assertion of Proposition~\ref{prop:hilbert-uniform-deviation}, together with the second part of Assumption~\ref{ass:GL-grid}, gives the required exponential summability. Since \(\kappa_1/6\ge12\), the proposition, applied with \(\kappa=\kappa_1/6\), yields
\[
\mathrm C\left\{\frac1N+
\frac{L_N(\log N)^6}{N^3}\right\}
\le
\mathrm C\frac{(\log N)^{d+6}}{N},
\]
which proves \eqref{deviation_bf} with \(\rho_{\mathcal H}=d+6\).
\\
\\
For $\widehat{bf}_{N,\boldsymbol h,\boldsymbol\eta}$, we proceed similarly to above. Let \(\boldsymbol h\in\mathcal H_N\), replace
\(K_{\boldsymbol\eta}\) throughout by
\(K_{\boldsymbol h}*K_{\boldsymbol\eta}\). Young's inequality gives
\[
\|K_{\boldsymbol h}*K_{\boldsymbol\eta}\|_1
\le\|K\|_1^2,
\qquad
\|K_{\boldsymbol h}*K_{\boldsymbol\eta}\|_2
\le\|K\|_1\|K_{\boldsymbol\eta}\|_2,
\qquad
\|(K_{\boldsymbol h}*K_{\boldsymbol\eta})*\psi\|_2
\le\|K\|_1^2\|\psi\|_2.
\]
Thus \eqref{eq:verification-weak-bf} remains valid with a constant independent of \(\boldsymbol h\), while \eqref{eq:verification-strong-bf} holds with
\(
s_{\boldsymbol h,\boldsymbol\eta}^2
:=
\|K\|_1^2\mathrm C_{bf}/\prod_{j=1}^d\eta_j.
\)
Since \(\kappa_1/(6\|K\|_1^2)\ge12\), applying the proposition with \(\kappa=\kappa_1/(6\|K\|_1^2)\) for each fixed \(\boldsymbol h\) gives the second numerator bound, uniformly in \(\boldsymbol h\).
\\
\\
Finally, for \(\widehat f_{N,\boldsymbol\eta}\) and \(\widehat f_{N,\boldsymbol h,\boldsymbol\eta}\), define
\(
Z^f_{\boldsymbol\eta}(x)
:=
\frac1{T-t_0}\int_{t_0}^T
\delta(t)K_{\boldsymbol\eta}(X(t)-x)\,dt,
\)
and let
\(\mathcal S_0^f\subset C_c^\infty(\mathbb R^d;\mathbb R)\) be countable, symmetric, and dense in the unit sphere of \(L^2(\mathbb R^d;\mathbb R)\). The same argument, with Fubini's theorem and Jensen's inequality in place of the stochastic-integral calculations, gives
{\small
\[
\sup_{\psi\in\mathcal S_0^f}
\mathbb E\!\left[
\langle Z^f_{\boldsymbol\eta},\psi\rangle_2^2
\right]
\le\mathrm C_w,
\qquad
\mathbb E\!\left[\|Z^f_{\boldsymbol\eta}\|_2^2\right]
\le
\frac{\mathrm C_f}{\prod_{j=1}^d\eta_j},
\qquad
\mathbb E\!\left[\|Z^f_{\boldsymbol\eta}\|_2^r\right]
\le
\mathrm C_r\left(
\frac{\mathrm C_f}{\prod_{j=1}^d\eta_j}
\right)^{r/2}.
\]
}
Applying Proposition~\ref{prop:hilbert-uniform-deviation} with \(s_{\boldsymbol\eta}^2=\mathrm C_f/\prod_j\eta_j\) and \(\kappa=\kappa_3/6\) proves the bound for \(\widehat f_{N,\boldsymbol\eta}\). A similar Young’s inequality argument proves the bound for \(\widehat f_{N,\boldsymbol h,\boldsymbol\eta}\). This completes the proof.
\qed
\subsection{Proof of Theorem~\ref{thm:GL-oracle}}
\label{subsec:pf:GL-oracle}

Throughout the proof, \(\mathrm C\) denotes a finite positive constant
that may change from line to line and is independent of \(N\) and of the
bandwidths in \(\mathcal H_N\).
\begingroup 
For \(\boldsymbol h\in\mathcal H_N\), set
\begin{equation*}
A_N^b(\boldsymbol h)
:=
\sup_{\boldsymbol\eta\in\mathcal H_N}
\left(
\big\|
\widehat{bf}_{N,\boldsymbol h,\boldsymbol\eta}
-
\widehat{bf}_{N,\boldsymbol\eta}
\big\|_2^2
-
\frac{\kappa_1\mathrm C_{bf}}
{N\prod_{j=1}^d\eta_j}
\right)_+,
\end{equation*}
namely, the supremum term in \eqref{eq:GL-hhat-bf}.
\endgroup We prove inequality~\eqref{eq:GL-oracle-bf}. Inequality~\eqref{eq:GL-oracle-f} follows from the same
argument, with  \(\widehat f_{N,\boldsymbol h}\) in
place of \(\widehat{bf}_{N,\boldsymbol h}\).

Fix \(\boldsymbol h\in\mathcal H_N\). By the definition of
\(A_N^b(\boldsymbol h)\), for every
\(\boldsymbol\eta\in\mathcal H_N\),
{
\begin{equation}
\label{eq:GL-basic-bound-aniso}
\big\|
\widehat{bf}_{N,\boldsymbol h,\boldsymbol\eta}
-
\widehat{bf}_{N,\boldsymbol\eta}
\big\|_{2}^2
\le
A_N^b(\boldsymbol h)
+
\frac{\kappa_1\mathrm C_{bf}}{N \prod_{j=1}^d \eta_j} .
\end{equation}
}
Using Lemma~\ref{lem:GL-commute}, the triangle inequality gives
{
\begin{align}
\big\|
\widehat{bf}_{N,\widehat{\boldsymbol h}_N}
-
b\overline f
\big\|_{2}^2
&\le
3\big\|
\widehat{bf}_{N,\boldsymbol h}
-
b\overline f
\big\|_{2}^2
+3\big\|
\widehat{bf}_{N,\boldsymbol h,\widehat{\boldsymbol h}_N}
-
\widehat{bf}_{N,\widehat{\boldsymbol h}_N}
\big\|_{2}^2 
+3\big\|
\widehat{bf}_{N,\widehat{\boldsymbol h}_N,\boldsymbol h}
-
\widehat{bf}_{N,\boldsymbol h}
\big\|_{2}^2 .
\label{eq:GL-triangle-aniso}
\end{align}
}
Applying \eqref{eq:GL-basic-bound-aniso} twice, first with
\((\boldsymbol h,\boldsymbol\eta)
=(\boldsymbol h,\widehat{\boldsymbol h}_N)\) and then with
\((\boldsymbol h,\boldsymbol\eta)
=(\widehat{\boldsymbol h}_N,\boldsymbol h)\), yields
{
\begin{align}
\big\|
\widehat{bf}_{N,\widehat{\boldsymbol h}_N}
-
b\overline f
\big\|_{2}^2
&\le
3\big\|
\widehat{bf}_{N,\boldsymbol h}
-
b\overline f
\big\|_{2}^2
+
3A_N^b(\boldsymbol h) +
3\frac{\kappa_1\mathrm C_{bf}}
{N \prod_{j=1}^d \widehat h_{N,j}}  +
3A_N^b(\widehat{\boldsymbol h}_N)
+
3\frac{\kappa_1\mathrm C_{bf}}
{N \prod_{j=1}^d h_j} .
\label{eq:GL-before-min-aniso}
\end{align}
}
By the definition of \(\widehat{\boldsymbol h}_N\),
\(
A_N^b(\widehat{\boldsymbol h}_N)
+
\frac{\kappa_2\mathrm C_{bf}}{N \prod_{j=1}^d \widehat h_{N,j}}
\le
A_N^b(\boldsymbol h)
+
\frac{\kappa_2\mathrm C_{bf}}{N \prod_{j=1}^d h_j} .
\)
Since \(\kappa_2\ge\kappa_1\), substitution in
\eqref{eq:GL-before-min-aniso} gives
{
\begin{equation}
\label{eq:GL-start-ineq-aniso}
\big\|
\widehat{bf}_{N,\widehat{\boldsymbol h}_N}
-
b\overline f
\big\|_{2}^2
\le
3\big\|
\widehat{bf}_{N,\boldsymbol h}
-
b\overline f
\big\|_{2}^2
+
6A_N^b(\boldsymbol h)
+
6\frac{\kappa_2\mathrm C_{bf}}{N \prod_{j=1}^d h_j} .
\end{equation}
}
We now control \(A_N^b(\boldsymbol h)\). For fixed
\(\boldsymbol h,\boldsymbol\eta\in\mathcal H_N\),
\begin{align}
\big\|
\widehat{bf}_{N,\boldsymbol h,\boldsymbol\eta}
-
\widehat{bf}_{N,\boldsymbol\eta}
\big\|_{2}^2
&\le
3\big\|
\mathbb E\!\left[\widehat{bf}_{N,\boldsymbol h,\boldsymbol\eta}\right]
-
\mathbb E\!\left[\widehat{bf}_{N,\boldsymbol\eta}\right]
\big\|_{2}^2 
+
3\big\|
\widehat{bf}_{N,\boldsymbol h,\boldsymbol\eta}
-
\mathbb E\!\left[\widehat{bf}_{N,\boldsymbol h,\boldsymbol\eta}\right]
\big\|_{2}^2 \nonumber 
\\& \qquad +
3\big\|
\widehat{bf}_{N,\boldsymbol\eta}
-
\mathbb E\!\left[\widehat{bf}_{N,\boldsymbol\eta}\right]
\big\|_{2}^2 .
\label{eq:GL-add-subtract-aniso}
\end{align}
Moreover, by Lemma~\ref{lem:GL-commute} and \eqref{eq:expectation-bfhat},
and Young's inequality,
{
\begin{align}
\big\|
\mathbb E\!\left[\widehat{bf}_{N,\boldsymbol h,\boldsymbol\eta}\right]
-
\mathbb E\!\left[\widehat{bf}_{N,\boldsymbol\eta}\right]
\big\|_{2}^2
&=
\big\|
K_{\boldsymbol\eta}*
\big(K_{\boldsymbol h}*(b\overline f)-b\overline f\big)
\big\|_{2}^2 \le
\|K\|_1^2
\big\|
K_{\boldsymbol h}*(b\overline f)-b\overline f
\big\|_{2}^2 .
\label{eq:GL-bias-control-aniso}
\end{align}
}
Combining \eqref{eq:GL-add-subtract-aniso} and
\eqref{eq:GL-bias-control-aniso}, we obtain
{
\begin{align}
A_N^b(\boldsymbol h)
&\le
3\|K\|_1^2
\big\|
K_{\boldsymbol h}*(b\overline f)-b\overline f
\big\|_{2}^2 +
3\sup_{\boldsymbol\eta\in\mathcal H_N}
\left(
\big\|
\widehat{bf}_{N,\boldsymbol h,\boldsymbol\eta}
-
\mathbb E\!\left[\widehat{bf}_{N,\boldsymbol h,\boldsymbol\eta}\right]
\big\|_{2}^2
-
\frac{\kappa_1\mathrm C_{bf}}{6N \prod_{j=1}^d \eta_j}
\right)_+ \nonumber\\
&\quad
+
3\sup_{\boldsymbol\eta\in\mathcal H_N}
\left(
\big\|
\widehat{bf}_{N,\boldsymbol\eta}
-
\mathbb E\!\left[\widehat{bf}_{N,\boldsymbol\eta}\right]
\big\|_{2}^2
-
\frac{\kappa_1\mathrm C_{bf}}{6N \prod_{j=1}^d \eta_j}
\right)_+ .
\label{eq:GL-A-split-aniso}
\end{align}
}

Taking expectations in \eqref{eq:GL-A-split-aniso} and applying
\begingroup Lemma~\ref{lem:GL-deviation-verification}\endgroup{} gives
{
\begin{equation}
\label{eq:GL-EA-bound-aniso}
\mathbb E\!\left[A_N^b(\boldsymbol h)\right]
\le
\mathrm C
\big\|
K_{\boldsymbol h}*(b\overline f)-b\overline f
\big\|_{2}^2
+
\mathrm C\frac{(\log N)^{\rho_{\mathcal H}}}{N}.
\end{equation}
}
Taking expectations in \eqref{eq:GL-start-ineq-aniso}, using
\eqref{eq:GL-EA-bound-aniso}, and then applying
Proposition~\ref{prop:fixed-bandwidth-risk}, we obtain
{
\[
\mathbb E\!\left[
\big\|
\widehat{bf}_{N,\widehat{\boldsymbol h}_N}
-
b\overline f
\big\|_{2}^2
\right]
\le
\mathrm C
\left[
\big\|
K_{\boldsymbol h}*(b\overline f)-b\overline f
\big\|_{2}^2
+
\frac{\kappa_2\mathrm C_{bf}}{N \prod_{j=1}^d h_j}
\right]
+
\mathrm C\frac{(\log N)^{\rho_{\mathcal H}}}{N}.
\]
}
Since \(\boldsymbol h\in\mathcal H_N\) was arbitrary, taking the
infimum over \(\boldsymbol h\) proves \eqref{eq:GL-oracle-bf};
the proof of \eqref{eq:GL-oracle-f} is identical, with
\(\widehat f_{N,\boldsymbol h}\), \(\overline f\), \(\mathrm C_f\) in
place of \(\widehat{bf}_{N,\boldsymbol h}\), \(b\overline f\),
\(\mathrm C_{bf}\).

It remains to prove the ratio assertion. The pointwise bound
\eqref{eq:risk-ratio-fixed} holds for every pair of bandwidths and
every \(x\in I\); applying it with
\(\boldsymbol h=\widehat{\boldsymbol h}_N\),
\(\boldsymbol h'=\widehat{\boldsymbol h}'_N\), integrating over \(I\),
taking expectations, and using
\eqref{eq:GL-oracle-bf}--\eqref{eq:GL-oracle-f} gives
\eqref{eq:GL-oracle-ratio} \begingroup after choosing
\(\mathrm C_5\) large enough. Finally, under the additional grid and
smoothness assumptions of Theorem~\ref{thm:GL-oracle},
Lemma~\ref{lem:nikolskii-kernel-bias} and the bandwidth-approximation
argument in Remark~\ref{rem:GL-rate}(i) yield
\eqref{eq:GL-rate-final}.\endgroup
\qed
\section{Proofs for Section~\ref{sec:minimax-rate}}
\label{app:proofs-minimax}

\subsection{Proof of Theorem~\ref{thm:minimax-rate}}
\label{Sec:pf-minimax-rate}

\begingroup 
The proof follows Tsybakov's many-hypotheses method and rests on two
technical ingredients, proved in
Sections~\ref{Sec:pf-lipschitz-packing} and~\ref{Sec:pf-kl-girsanov}.
a separated finite family of drifts, and a control of the associated
Kullback--Leibler divergences.
\endgroup

\begin{lemma}
\label{lem:lipschitz-packing}
Under the assumptions of Theorem~\ref{thm:minimax-rate}, there exist
constants \(\ell_0,c_1,c_2,c_3>0\) such that for every
\(\ell\in(0,\ell_0]\) and every small enough \(\rho>0\), there are
\(M\ge1\) and
\(b_0,\ldots,b_M\in\mathcal B_{\boldsymbol\alpha}(L_\alpha,L_0,B)\)
with \(b_0=0\),
\[
\log M\ge c_1\rho^{-d/\bar\alpha},
\qquad
\min_{0\le q<q'\le M}\|b_q-b_{q'}\|_{2,I}^2\ge c_2\ell^2\rho^2,
\qquad
\max_{0\le q\le M}\|b_q\|_\infty\le c_3\ell\rho.
\]
\end{lemma}

\begin{lemma}
\label{lem:kl-girsanov}
Let \(b_0,\ldots,b_M\) be given by Lemma~\ref{lem:lipschitz-packing},
and let \(\mathbb P_q^{X,N}\) be the law of
\((X^{(1)},\ldots,X^{(N)})\) under the drift \(b_q\) and the same
diffusion coefficient \(\Sigma\). Under the assumptions of
Theorem~\ref{thm:minimax-rate},
\(\mathbb P_q^{X,N}\ll\mathbb P_0^{X,N}\) for every
\(q=1,\ldots,M\), and there exists \(C>0\),
\begingroup independent of \(q,N,\ell,\rho\),\endgroup
such that
\[
\max_{1\le q\le M}
K(\mathbb P_q^{X,N},\mathbb P_0^{X,N})
\le CN\ell^2\rho^2,
\]
\begingroup 
where \(K(P,Q):=\int\log(dP/dQ)\,dP\) denotes the Kullback--Leibler
divergence.
\endgroup
\end{lemma}

Choose \(\rho\asymp N^{-\bar\alpha/(2\bar\alpha+d)}\) and fix
\(\ell\in(0,\ell_0]\), to be taken small enough below. For \(N\) large
enough, \(\rho\) is small enough for Lemma~\ref{lem:lipschitz-packing}
to apply, yielding \(b_0,\ldots,b_M\) with \(b_0=0\). Set
\(s_\rho:=\sqrt{c_2}\,\ell\rho/2\). We apply
Tsybakov~\cite[Theorem~2.5, p.~99]{tsybakov2009introduction} to
\(b_0,\ldots,b_M\) with metric \(\|\cdot\|_{2,I}\); its assumptions are
verified as follows.
\begin{enumerate}[label=\textup{(\roman*)}, leftmargin=*]
\item By Lemma~\ref{lem:lipschitz-packing},
\(\|b_q-b_{q'}\|_{2,I}\ge\sqrt{c_2}\,\ell\rho=2s_\rho\) for
\(0\le q<q'\le M\).
\item \begingroup By Lemma~\ref{lem:kl-girsanov},
\(\mathbb P_q^{X,N}\ll\mathbb P_0^{X,N}\) for
\(q=1,\ldots,M\).\endgroup
\item \begingroup By Lemma~\ref{lem:kl-girsanov},
\(\frac1M\sum_{q=1}^M K(\mathbb P_q^{X,N},\mathbb P_0^{X,N})
\le C\ell^2N\rho^2\).\endgroup
Since \(N\rho^2\asymp\rho^{-d/\bar\alpha}\) while
\(\log M\ge c_1\rho^{-d/\bar\alpha}\), taking \(\ell\) small enough
ensures
\(\frac1M\sum_{q=1}^M K(\mathbb P_q^{X,N},\mathbb P_0^{X,N})
\le\begingroup \kappa\endgroup\log M\)
for some fixed
\(\begingroup \kappa\in(0,1/8)\endgroup\).
\end{enumerate}
Therefore, by Tsybakov~\cite[Theorem~2.5,
p.~99]{tsybakov2009introduction},
\[
\inf_{\widetilde b_N
=\widetilde b_N(\begingroup \mathcal X_N\endgroup)}
\sup_{0\le q\le M}
\begingroup \mathbb P_q^{X,N}\endgroup
\bigl(\|\widetilde b_N-b_q\|_{2,I}\ge s_\rho\bigr)
\ge
\frac{\sqrt M}{1+\sqrt M}
\left(1-2\begingroup \kappa\endgroup
-\sqrt{\frac{2\begingroup \kappa\endgroup}{\log M}}\right),
\]
\begingroup 
and since \(\log M\ge c_1\rho^{-d/\bar\alpha}\to\infty\) as
\(\rho\to0\),
\endgroup
the right-hand side is bounded below by a constant \(c_T>0\) for all
\(N\) large enough. \begingroup Using
\(Z^2\ge s_\rho^2\mathbf1_{\{Z\ge s_\rho\}}\) with
\(Z=\|\widetilde b_N-b_q\|_{2,I}\),\endgroup
\begin{align*}
\inf_{\widetilde b_N
=\widetilde b_N(\begingroup \mathcal X_N\endgroup)}
\sup_{b\in\mathcal B_{\boldsymbol\alpha}(L_\alpha,L_0,B)}
\mathbb E_b\bigl[\|\widetilde b_N-b\|_{2,I}^2\bigr]
&\ge
s_\rho^2
\inf_{\widetilde b_N
=\widetilde b_N(\begingroup \mathcal X_N\endgroup)}
\sup_{0\le q\le M}
\begingroup \mathbb P_q^{X,N}\endgroup
\bigl(\|\widetilde b_N-b_q\|_{2,I}\ge s_\rho\bigr)
\\
&\ge
c_Ts_\rho^2
=\frac{c_Tc_2}{4}\ell^2\rho^2
\ge cN^{-2\bar\alpha/(2\bar\alpha+d)},
\end{align*}
\begingroup 
which proves Theorem~\ref{thm:minimax-rate}.
\endgroup
\qed
\subsection{Proof of Corollary~\ref{cor:minimax-rate-censored}}
\label{Sec:pf-minimax-rate-censored}

\begingroup 
Set
\(A_N:=(X^{(1)},\ldots,X^{(N)},\theta^{(1)},\ldots,\theta^{(N)})\).
Since \(O_N\) is measurable with respect to \(A_N\), every estimator based on
\(O_N\) is also an estimator based on the augmented full observation \(A_N\).
\endgroup
\begingroup 
Thus,
\endgroup
{
\begin{align}
\label{eq:censored-full-comparison}
\inf_{\widetilde b_N=\widetilde b_N(O_N)}
\sup_{b\in\mathcal B_{\boldsymbol\alpha}(L_\alpha,L_0,B)}
\mathbb E_b
\left[
\|\widetilde b_N-b\|_{2,I}^2
\right]
\ge
\inf_{\widetilde b_N=\widetilde b_N(A_N)}
\sup_{b\in\mathcal B_{\boldsymbol\alpha}(L_\alpha,L_0,B)}
\mathbb E_b
\left[
\|\widetilde b_N-b\|_{2,I}^2
\right].
\end{align}
}
\begingroup 
For the finite family \(b_0,\ldots,b_M\) constructed in the proof of
Theorem~\ref{thm:minimax-rate}, let \(\mathbb P_q^{X,\theta,N}\) be the law of
\(A_N\) under \(b_q\). By Assumption~\ref{ass:censoring-nondegeneracy}(C1) and because the
censoring mechanism does not depend on \(b\),
\(\mathbb P_q^{X,\theta,N}=\mathbb P_q^{X,N}\otimes\mathbb P^{\theta,N}\).
Consequently, \(\mathbb P_q^{X,\theta,N}\ll
\mathbb P_0^{X,\theta,N}\) and
\[
K(\mathbb P_q^{X,\theta,N},\mathbb P_0^{X,\theta,N})
=K(\mathbb P_q^{X,N},\mathbb P_0^{X,N}).
\]
\endgroup
\begingroup 
Thus, the results from Lemma~\ref{lem:lipschitz-packing} and \ref{lem:kl-girsanov} used in the proof of
Theorem~\ref{thm:minimax-rate} remain unchanged.
The same application of Tsybakov~\cite[Theorem~2.5, p.~99]{tsybakov2009introduction} yields
\[
\inf_{\widetilde b_N=\widetilde b_N(A_N)}
\sup_{b\in\mathcal B_{\boldsymbol\alpha}(L_\alpha,L_0,B)}
\mathbb E_b\!\left[
\|\widetilde b_N-b\|_{2,I}^2
\right]
\ge cN^{-2\bar\alpha/(2\bar\alpha+d)}.
\]
Combining this inequality with~\eqref{eq:censored-full-comparison} proves the
corollary.
\endgroup
\qed
\subsection{Proof of Lemma~\ref{lem:lipschitz-packing}}
\label{Sec:pf-lipschitz-packing}

\begingroup 
Let \(x^\star=(x^\star_1,\ldots,x^\star_d)\) and
\(Q:=\prod_{j=1}^d[x^\star_j,x^\star_j+L_Q]\subset I\) be the hypercube
from Theorem~\ref{thm:minimax-rate}. For \(\rho>0\), set
\(h_j:=\rho^{1/\alpha_j}\) and \(n_j:=\lfloor L_Q/h_j\rfloor\). For
\(\rho\) small enough, \(h_j\le L_Q/2\) for all \(j\), and the rectangles
\[
Q_\nu:=\prod_{j=1}^d
[x^\star_j+\nu_jh_j,\,x^\star_j+(\nu_j+1)h_j],
\qquad
\nu\in\prod_{j=1}^d\{0,\ldots,n_j-1\},
\]
are contained in \(Q\) and have pairwise disjoint interiors. Enumerate
them as \(Q_1,\ldots,Q_{m_\rho}\), \(m_\rho:=\prod_{j=1}^dn_j\), and let
\(x_k\) denote the lower corner of \(Q_k\), so that
\(Q_k=x_k+\prod_{j=1}^d[0,h_j]\). Since
\(\lfloor L_Q/h_j\rfloor\in[L_Q/(2h_j),L_Q/h_j]\) and
\(\prod_{j=1}^dh_j=\rho^{d/\bar\alpha}\),
\[
m_\rho\ge c_Q\rho^{-d/\bar\alpha}
\quad\text{with } c_Q:=L_Q^d/2^d,
\qquad\text{and}\qquad
m_\rho\prod_{j=1}^dh_j\le L_Q^d.
\]
\endgroup
Let \(\psi\in C_c^\infty((0,1)^d)\) be nonzero with
\(0\le\psi\le1\), and set
\[
\psi_{k,\boldsymbol h}(x):=
\psi\Bigl(\frac{x_1-x_{k,1}}{h_1},\ldots,
\frac{x_d-x_{k,d}}{h_d}\Bigr),
\qquad k=1,\ldots,m_\rho.
\]
\begingroup 
Since \(\operatorname{supp}(\psi)\subset(0,1)^d\), we have
\(\operatorname{supp}(\psi_{k,\boldsymbol h})\subset Q_k\subset I\),
and the \(\psi_{k,\boldsymbol h}\) have pairwise disjoint supports.
\endgroup
For \(\theta\in\{0,1\}^{m_\rho}\), define
\(b_\theta:=\ell\rho\sum_{k=1}^{m_\rho}
\theta_k\psi_{k,\boldsymbol h}\,e_1\),
where \(e_1\) is the first canonical vector of \(\mathbb R^d\) and
\(\ell>0\) is chosen below.

We first check that \(b_\theta\) belongs to the drift class. By the
disjointness of the supports and \(0\le\psi\le1\),
\(\|b_\theta\|_\infty\le\ell\rho\), so
\(|b_\theta(0)|\le B\) for \(\rho\) small enough. Moreover, for
\(\rho\le1\) and every \(j\), since \(\alpha_j\ge1\),
\(\|\partial_jb_\theta\|_\infty
\le\ell\|\partial_j\psi\|_\infty\,\rho^{1-1/\alpha_j}
\le\ell\|\partial_j\psi\|_\infty\);
hence, if \(\ell_0\le L_0/\|\nabla \psi\|_{\infty}\), then
\(b_\theta\) is \(L_0\)-Lipschitz for every \(\ell\in(0,\ell_0]\).

\begingroup 
It remains to verify Nikol'skii regularity. For
\(j=1,\ldots,d\), set \(r_j:=\lceil\alpha_j\rceil-1\) and
\(\vartheta_j:=\alpha_j-r_j\in(0,1]\).
\endgroup
By disjointness of the supports, for every integer \(a\ge0\),
\begin{align*}
\|\partial_j^a b_\theta\|_2^2
=
\ell^2\rho^2h_j^{-2a}
\Bigl(\prod_{i=1}^dh_i\Bigr)
\|\partial_j^a\psi\|_2^2
\sum_{k=1}^{m_\rho}\theta_k
\begingroup 
\le
L_Q^d\|\partial_j^a\psi\|_2^2\,\ell^2\rho^2h_j^{-2a},
\endgroup
\end{align*}
\begingroup 
using \(\sum_k\theta_k\le m_\rho\) and
\(m_\rho\prod_ih_i\le L_Q^d\) in the inequality;
\endgroup
in particular, \(\partial_j^ab_\theta\in
L^2(\mathbb R^d;\mathbb R^d)\) for \(a=0,\ldots,r_j\).
\begingroup 
Since \(\rho=h_j^{\alpha_j}\), we obtain, for \(|z|\le h_j\),
\[
\bigl\|\partial_j^{r_j}b_\theta(\cdot+ze_j)
-\partial_j^{r_j}b_\theta\bigr\|_2
\le
|z|\,\|\partial_j^{r_j+1}b_\theta\|_2
\le
C\ell|z|\,h_j^{\alpha_j-r_j-1}
=
C\ell|z|^{\vartheta_j}
\Bigl(\frac{|z|}{h_j}\Bigr)^{1-\vartheta_j}
\le
C\ell|z|^{\vartheta_j},
\]
\endgroup
while for \(|z|>h_j\),
\[
\bigl\|\partial_j^{r_j}b_\theta(\cdot+ze_j)
-\partial_j^{r_j}b_\theta\bigr\|_2
\le
2\|\partial_j^{r_j}b_\theta\|_2
\le
C\ell h_j^{\vartheta_j}
\le
C\ell|z|^{\vartheta_j}.
\]
Here \(C\) depends only on \(\psi\), \(d\), \(Q\), and
\(\boldsymbol\alpha\). Decreasing \(\ell_0\) if necessary so that
\(C\ell_0\le L_\alpha\), we conclude
\(b_\theta\in\mathcal N_d^d(\boldsymbol\alpha,L_\alpha)\), and
therefore
\(b_\theta\in\mathcal B_{\boldsymbol\alpha}(L_\alpha,L_0,B)\).

We now select the binary vectors. With the Hamming distance
\(d_H(\theta,\theta'):=\sum_{k=1}^{m_\rho}
\mathbf1_{\{\theta_k\ne\theta'_k\}}\), the Varshamov--Gilbert bound
\cite[Lemma~2.9, p.~104]{tsybakov2009introduction} provides
\(\theta^0,\ldots,\theta^M\in\{0,1\}^{m_\rho}\) with
\(\theta^0=(0,\ldots,0)\),
\(d_H(\theta^q,\theta^{q'})\ge m_\rho/8\) for \(q\ne q'\), and
\(M\ge2^{m_\rho/8}\); hence
\begingroup 
\(\log M\ge(\log2)\,m_\rho/8\ge c_1\rho^{-d/\bar\alpha}\) with
\(c_1:=(\log2)c_Q/8\).
\endgroup
Finally, set \(b_q:=b_{\theta^q}\). By disjointness of the supports,
\begin{align*}
\|b_q-b_{q'}\|_{2,I}^2
=
\ell^2\rho^2
\Bigl(\prod_{j=1}^dh_j\Bigr)\|\psi\|_2^2\,
d_H(\theta^q,\theta^{q'})
\ge
\frac{\ell^2\|\psi\|_2^2}{8}\,
\rho^2\rho^{d/\bar\alpha}\,m_\rho
\ge
c_2\ell^2\rho^2,
\qquad
c_2:=\frac{c_Q\|\psi\|_2^2}{8},
\end{align*}
where the last inequality uses
\(m_\rho\ge c_Q\rho^{-d/\bar\alpha}\). This proves the lemma.
\qed
\subsection{Proof of Lemma~\ref{lem:kl-girsanov}}
\label{Sec:pf-kl-girsanov}
\begingroup Fix \(q\in\{1,\ldots,M\}\).\endgroup
Work on a weak-solution space carrying \((X,W_q)\) under \(\mathbb P_q\),
and let \(\mathbb P_q^X\) denote the \(X\)-marginal. Under \(\mathbb P_q\),
\[
dX(t)=b_q(X(t))\,dt+\Sigma(X(t))\,dW_q(t),
\qquad X(0)=x_0,
\]
where \(W_q\) is a Brownian motion; recall that \(b_0=0\) in the
construction of Lemma~\ref{lem:lipschitz-packing}. Set
\(a:=\Sigma\Sigma^\top\) and \(\eta_q:=\Sigma^\top a^{-1}b_q\). Then
\(\eta_q(X(t))\) is progressively measurable, and by uniform ellipticity
of \(a\) together with Lemma~\ref{lem:lipschitz-packing},
\begingroup 
\[
|\eta_q(x)|^2
=
b_q(x)^\top a(x)^{-1}b_q(x)
\le
C_\Sigma\,\ell^2\rho^2,
\qquad x\in\mathbb R^d,
\]
\endgroup
so that \(\int_0^T|\eta_q(X(t))|^2\,dt\le C_\Sigma T\ell^2\rho^2\)
deterministically; in particular, Novikov's condition holds. By
Girsanov's theorem, the probability measure \(\mathbb Q_q\) with density
\[
\frac{d\mathbb Q_q}{d\mathbb P_q}
=
\exp\left(
-\int_0^T \eta_q(X(t))^\top dW_q(t)
-\frac12\int_0^T |\eta_q(X(t))|^2\,dt
\right)
=:Z_q(T)
\]
makes \(\widehat W_q(t):=W_q(t)+\int_0^t\eta_q(X(s))\,ds\) a Brownian
motion under \(\mathbb Q_q\). Since \(\Sigma\eta_q=b_q\), under
\(\mathbb Q_q\) we have \(dX(t)=\Sigma(X(t))\,d\widehat W_q(t)\).
\begingroup 
By uniqueness in law for this zero-drift SDE, the \(X\)-marginal of
\(\mathbb Q_q\) is \(\mathbb P_0^X\); and since \(Z_q(T)>0\)
\(\mathbb P_q\)-a.s., the measures \(\mathbb P_q\) and \(\mathbb Q_q\)
are equivalent, whence \(\mathbb P_q^X\ll\mathbb P_0^X\).
\endgroup
Moreover, 
\begin{align*}
K(\mathbb P_q^X,\mathbb P_0^X)
\le
K(\mathbb P_q,\mathbb Q_q)
=
\mathbb E_q\!\left[-\log Z_q(T)\right]
=
\frac12\,
\mathbb E_q\!\left[\int_0^T|\eta_q(X(t))|^2\,dt\right]
\le
C\ell^2\rho^2,
\end{align*}
\begingroup 
where the stochastic integral in \(-\log Z_q(T)\) has zero mean because
its integrand is bounded, hence square-integrable.
\endgroup
For \(N\) independent paths, absolute continuity is preserved under
products and Kullback-Leibler divergence is additive, so
\begingroup 
\(
K(\mathbb P_q^{X,N},\mathbb P_0^{X,N})
=N\,K(\mathbb P_q^X,\mathbb P_0^X)
\le CN\ell^2\rho^2,
\)
\endgroup
which proves the lemma.
\qed
\section{Implementation details and additional simulation results}
\label{app:simulation-implementation-tables}

This appendix collects all simulation choices needed to reproduce
Section~\ref{sec:simulation}. Additional drift
reconstructions, empirical MISE curves, and the corresponding numerical
values are reported at the end.

\smallskip
\noindent\textbf{Effective region.}
The retained region is based on
\[
p_{\mathrm{unc}}(x)=
\frac{\int_{t_0}^T\mathbb P(\delta_t=1\mid Y_t=x)p_{Y_t}(x)\,dt}
{\int_{t_0}^T p_{Y_t}(x)\,dt},
\quad
\widehat p_{\mathrm{unc}}(x)=
\frac{\sum_{i,k}\delta^{(i)}_{t_k}K_{(0.5,0.5)}(Y^{(i)}_{t_k}-x)}
{\sum_{i,k}K_{(0.5,0.5)}(Y^{(i)}_{t_k}-x)}.
\]
Here the sums run over \(i=1,\ldots,N\) and \(k=n_0,\ldots,n\), whenever
the denominator is positive. We set
\(\mathcal R=\{x\in[-2,2]^2:\widehat p_{\mathrm{unc}}(x)>0.5\}\). The condition \(\widehat p_{\mathrm{unc}}(x)>0.5\) retains points that
are uncensored at least half of the time, a quantitative empirical
analog of the visibility condition in
Assumption~\ref{ass:censoring-nondegeneracy}, which guarantees
\(\inf_I\overline f>0\) via
Proposition~\ref{prop:positive-denominator}. Since \(\mathcal R\) only
restricts the evaluation grid and does not enter the estimators, its
data-dependence does not interfere with the theory; outside
\(\mathcal R\), too few uncensored observations are available for any
method to be informative.

\smallskip
\noindent\textbf{MISE.}
The evaluation grid is
\[
\mathcal G=\left\{\left(-2+\frac{4\ell_1}{34},-2+\frac{4\ell_2}{34}\right):
(\ell_1,\ell_2)\in\{0,\ldots,34\}^2\right\},
\qquad
\mathcal G_{\mathcal R}=\mathcal G\cap\mathcal R.
\]
The empirical MISE is
\[
\widehat{\operatorname{MISE}}_{\mathcal R}^{\mathrm{MC}}=
\frac{1}{M}\sum_{m=1}^M\frac{1}{|\mathcal G_{\mathcal R}|}
\sum_{x\in\mathcal G_{\mathcal R}}
\|\widehat b^{(m)}(x)-b(x)\|^2,
\qquad M=50.
\]
The normalization by \(|\mathcal G_{\mathcal R}|\) accounts for the fact
that the retained region may have different sizes across censoring
mechanisms.

\smallskip
\noindent\textbf{Discrete kernel estimators.}
Let \(n_0=t_0/\Delta\). For \(x\in I\),
\begin{align}
\widehat f_{N,\mathbf h'}(x)
&=\frac{1}{N(n-n_0)}\sum_{i=1}^N\sum_{k=n_0}^{n-1}
\delta^{(i)}_{t_k}K_{\mathbf h'}(Y^{(i)}_{t_k}-x),
\label{eq:app-discrete-f}\\[-0.3em]
\widehat{bf}_{N,\mathbf h}(x)
&=\frac{1}{N(n-n_0)\Delta}\sum_{i=1}^N\sum_{k=n_0}^{n-1}
\delta^{(i)}_{t_k}\delta^{(i)}_{t_{k+1}}K_{\mathbf h}(Y^{(i)}_{t_k}-x)
\bigl(Y^{(i)}_{t_{k+1}}-Y^{(i)}_{t_k}\bigr).
\label{eq:app-discrete-bf}
\end{align}
The factor \(\delta^{(i)}_{t_{k+1}}\) keeps only increments whose two
endpoints are uncensored. The ratio is
\[
\widehat b_{N,\mathbf h,\mathbf h'}(x)=
\frac{\widehat{bf}_{N,\mathbf h}(x)}{\widehat f_{N,\mathbf h'}(x)}
\mathbf 1_{\{\widehat f_{N,\mathbf h'}(x)\ge \tau\}},
\quad
\tau=0.05\,\mathrm{median}\{\widehat f_{N,\mathbf h'}(x):x\in\mathcal G,
\widehat f_{N,\mathbf h'}(x)>0\}.
\]
The same thresholding convention is used for GL and PCO.

\smallskip
\noindent\textbf{Kernel and bandwidth grids.}
We use the product biweight kernel
\[
k(u)=\frac{15}{16}(1-u^2)^2\mathbf 1_{\{|u|\le1\}},
\qquad
K_{\mathbf h}(x)=\frac{1}{h_1h_2}k(x_1/h_1)k(x_2/h_2).
\]
For the two OU models,
\(\mathcal H_{\mathrm{iso}}=\{0.18(1.10/0.18)^{(j-1)/29}:1\le j\le30\}\).
For the Fixman-like model,
\[
\mathcal H_{\mathrm{aniso}}=\mathcal H_1\times\mathcal H_2,\quad
\mathcal H_1=\{0.25,0.28,0.32,0.35,0.40,0.45,0.52,0.58\},
\]
\[
\mathcal H_2=\{0.35,0.40,0.45,0.52,0.58,0.67,0.75,0.86,1.00,1.15\}.
\]
\smallskip
\noindent\textbf{GL calibration constants and plug-in constant.}
For GL, numerator and denominator bandwidths are selected separately by
\eqref{eq:GL-hhat-bf} and \eqref{eq:GL-hhat-f}, with
\((\kappa_1,\kappa_2,\kappa_3,\kappa_4)=(0.001,0.002,0.001,0.002)\).
These values were calibrated once on the \emph{uncensored} isotropic OU
model, by retaining values for which the selection rule is nondegenerate,
that is, for which the selected bandwidths do not systematically coincide
with the smallest or the largest element of the bandwidth grid. They are
then kept fixed across all models and censoring scenarios.

The constant \(\mathrm C_f=\|K\|_2^2\) is known. The unknown
\(\mathrm C_{bf}\) in \eqref{eq:variance-constants} is replaced by the
plug-in estimate
\[
\widehat{\mathrm C}_{bf}
:=
2\|K\|_2^2
\left(
\widehat D_N
+
\frac{1}{T-t_0}\,
\widehat S_N
\right),
\]
where \(\widehat D_N\) and \(\widehat S_N\) estimate
\(\|b\|_{2,\overline f}^2\) and
\(\bigl\|\sqrt{\operatorname{Tr}(\Sigma\Sigma^\top)}\bigr\|_{2,\overline f}^2\),
respectively, using the empirical counterparts of identity~\eqref{def:uncensoring-prob}.
\begin{align*}
\widehat D_N
&:=
\frac{1}{N(n-n_0)}
\sum_{i=1}^N\sum_{k=n_0}^{n-1}
\delta^{(i)}_{t_k}\,
\bigl|\widehat b_{N,\mathbf h_0,\mathbf h_0}(Y^{(i)}_{t_k})\bigr|^2,
\qquad
\mathbf h_0=(0.5,0.5),
\\
\widehat S_N
&:=
\frac{1}{N(n-n_0)\Delta}
\sum_{i=1}^N\sum_{k=n_0}^{n-1}
\delta^{(i)}_{t_k}\delta^{(i)}_{t_{k+1}}\,
\bigl|Y^{(i)}_{t_{k+1}}-Y^{(i)}_{t_k}\bigr|^2 .
\end{align*}
where in \(\widehat D_N\), the estimator 
\(\widehat b_{N,\mathbf h_0,\mathbf h_0}\) of $b$ with a fixed bandwidth is used as a pilot estimator. In \(\widehat S_N\), the squared increment
\(|Y^{(i)}_{t_{k+1}}-Y^{(i)}_{t_k}|^2/\Delta\) over pairs of uncensored
endpoints is the discrete quadratic-variation analog of
\(\operatorname{Tr}(\Sigma\Sigma^\top)(Y^{(i)}_{t_k})\).

\smallskip
\noindent\textbf{Least-squares estimator.}
For \(d=1\), LS is the censored projection estimator of
Huang~\cite{huang2026censoredsde}. For \(d=2\), we use the tensor-product
implementation of Dussap~\cite{dussap2023nonparametric}. On \(A=[-2,2]\),
let
\[
\varphi_0(u)=4^{-1/2},\quad
\varphi_{2\ell-1}(u)=\sqrt{2/4}\cos\{2\pi\ell(u+2)/4\},\quad
\varphi_{2\ell}(u)=\sqrt{2/4}\sin\{2\pi\ell(u+2)/4\}.
\]
For \(\mathbf m=(m_1,m_2)\), define
\(\Phi_{j_1,j_2}(x)=\varphi_{j_1}(x_1)\varphi_{j_2}(x_2)\),
\(0\le j_1<m_1\), \(0\le j_2<m_2\). With indices \(j,k\),
\begin{align*}
\widehat\Psi_{\mathbf m}(j,k)
&=\frac{1}{N(n-n_0)}\sum_{i=1}^N\sum_{r=n_0}^{n-1}
\delta^{(i)}_{t_r}\delta^{(i)}_{t_{r+1}}\Phi_j(Y^{(i)}_{t_r})\Phi_k(Y^{(i)}_{t_r}),\\[-0.3em]
\widehat Z_{\mathbf m,\ell}(j)
&=\frac{1}{N(n-n_0)\Delta}\sum_{i=1}^N\sum_{r=n_0}^{n-1}
\delta^{(i)}_{t_r}\delta^{(i)}_{t_{r+1}}\Phi_j(Y^{(i)}_{t_r})
\bigl(Y^{(i)}_{\ell,t_{r+1}}-Y^{(i)}_{\ell,t_r}\bigr),\quad \ell=1,2.
\end{align*}
If \(\widehat\Psi_{\mathbf m}\) is nonsingular,
\(\widehat\alpha_{\mathbf m,\ell}=\widehat\Psi_{\mathbf m}^{-1}
\widehat Z_{\mathbf m,\ell}\). The reported LS curve is the oracle
estimator
\[
\widehat b_N^{\mathrm{LS}}
=\widehat b^{\mathrm{LS}}_{\widehat{\mathbf m}_{\star}},
\qquad
\widehat{\mathbf m}_{\mathrm{\star}}\in
\operatorname*{arg\,min}_{\mathbf m\in\mathcal M_{\mathrm{LS}}}
\frac{1}{|\mathcal G_{\mathcal R}|}\sum_{x\in\mathcal G_{\mathcal R}}
\|\widehat b^{\mathrm{LS}}_{\mathbf m}(x)-b(x)\|^2.
\]
Thus the true drift is used only to select the LS dimension; LS is not a
data-driven adaptive estimator.

\smallskip
\noindent\textbf{PCO estimator.}
We implement the PCO method of Lacour, Massart and Rivoirard~\cite{lacour2017estimator}, in the form extended to i.i.d.\ diffusion paths by Marie and Rosier~\cite{marie2023nadaraya} on
\eqref{eq:app-discrete-f}--\eqref{eq:app-discrete-bf}. \begingroup   We consider an auxiliary kernel \(g\), here we take $g=k$ (the biweight kernel). \endgroup
Let \(\mathbf h_0,\mathbf h'_0\) be the smallest bandwidths in the grid set for the numerator and denominator estimators, and set
\[
\langle u,v\rangle_{\mathcal G,g}:=|\mathcal G|^{-1}\sum_{x\in\mathcal G}
g(x)\langle u(x),v(x)\rangle,\quad
\|u\|_{\mathcal G,g}^2:=\langle u,u\rangle_{\mathcal G,g}.
\]
For \(i=1,\ldots,N\), define
\[
B_{\mathbf h}^{(i)}(x):=\sum_{k=n_0}^{n-1}\delta_{t_k}^{(i)}\delta_{t_{k+1}}^{(i)}
K_{\mathbf h}(Y_{t_k}^{(i)}-x)(Y_{t_{k+1}}^{(i)}-Y_{t_k}^{(i)}),
\quad
F_{\mathbf h'}^{(i)}(x):=\Delta\sum_{k=n_0}^{n-1}\delta_{t_k}^{(i)}
K_{\mathbf h'}(Y_{t_k}^{(i)}-x).
\]
The selected bandwidths are
\[
\begin{aligned}
\widehat{\mathbf h}^{\mathrm{PCO}}
&\in\arg\min_{\mathbf h}
\left\{
\|\widehat{bf}_{N,\mathbf h}-\widehat{bf}_{N,\mathbf h_0}\|_{\mathcal G,g}^2
+\frac{2}{N^2(T-t_0)^2}\sum_{i=1}^N
\langle B_{\mathbf h}^{(i)},B_{\mathbf h_0}^{(i)}\rangle_{\mathcal G,g}
\right\},\\
\widehat{\mathbf h}^{\prime\mathrm{PCO}}
&\in\arg\min_{\mathbf h'}
\left\{
\|\widehat f_{N,\mathbf h'}-\widehat f_{N,\mathbf h'_0}\|_{\mathcal G}^2
+\frac{2}{N^2(T-t_0)^2}\sum_{i=1}^N
\langle F_{\mathbf h'}^{(i)},F_{\mathbf h'_0}^{(i)}\rangle_{\mathcal G}
\right\}.
\end{aligned}
\]
The resulting estimator is
\(\widehat b_{N,\widehat{\mathbf h}^{\mathrm{PCO}},
\widehat{\mathbf h}^{\prime\mathrm{PCO}}}\), with the same threshold and region
\(\mathcal R\) as GL.

\smallskip
\noindent\textbf{Additional simulation tables.}

\begin{table}[!htbp]
\centering
\scriptsize
\setlength{\tabcolsep}{3pt}
\begin{tabular}{llccc}
\toprule
Censoring mechanism & \(N\) & isoGL & LS & PCO \\
\midrule
Random ball & 50
& \(0.0406 \pm 0.0039\)
& \(0.0453 \pm 0.0033\)
& \(0.0656 \pm 0.0050\) \\
& 100
& \(0.0398 \pm 0.0032\)
& \(0.0375 \pm 0.0029\)
& \(0.0440 \pm 0.0034\) \\
& 200
& \(0.0107 \pm 0.0007\)
& \(0.0352 \pm 0.0020\)
& \(0.0332 \pm 0.0012\) \\
& 300
& \(0.0098 \pm 0.0005\)
& \(0.0342 \pm 0.0018\)
& \(0.0282 \pm 0.0010\) \\
\addlinespace
Coordinatewise right & 50
& \(0.1806 \pm 0.0070\)
& \(0.1708 \pm 0.0118\)
& \(0.2067 \pm 0.0218\) \\
& 100
& \(0.0780 \pm 0.0118\)
& \(0.1572 \pm 0.0160\)
& \(0.1665 \pm 0.0202\) \\
& 200
& \(0.0670 \pm 0.0059\)
& \(0.1468 \pm 0.0068\)
& \(0.1221 \pm 0.0112\) \\
& 300
& \(0.0616 \pm 0.0057\)
& \(0.1247 \pm 0.0098\)
& \(0.1002 \pm 0.0122\) \\
\bottomrule
\end{tabular}
\caption{Empirical MISE for the linear OU model. The entries report
mean \(\pm\) Monte Carlo standard error over \(M=50\) replications.}
\label{tab:mise-linear-ou}
\end{table}

\begin{table}[!htbp]
\centering
\scriptsize
\setlength{\tabcolsep}{3pt}
\begin{tabular}{llccc}
\toprule
Censoring mechanism & \(N\) & isoGL & LS & PCO \\
\midrule
Random ball & 50
& \(0.2285 \pm 0.0028\)
& \(0.0745 \pm 0.0064\)
& \(0.0881 \pm 0.0062\) \\
& 100
& \(0.2245 \pm 0.0027\)
& \(0.0594 \pm 0.0019\)
& \(0.0567 \pm 0.0046\) \\
& 200
& \(0.0812 \pm 0.0017\)
& \(0.0519 \pm 0.0023\)
& \(0.0469 \pm 0.0026\) \\
& 300
& \(0.0795 \pm 0.0013\)
& \(0.0502 \pm 0.0026\)
& \(0.0405 \pm 0.0020\) \\
\addlinespace
Coordinatewise right & 50
& \(0.3639 \pm 0.0161\)
& \(0.2120 \pm 0.0150\)
& \(0.2847 \pm 0.0359\) \\
& 100
& \(0.1181 \pm 0.0057\)
& \(0.1677 \pm 0.0110\)
& \(0.1687 \pm 0.0113\) \\
& 200
& \(0.1031 \pm 0.0030\)
& \(0.1512 \pm 0.0232\)
& \(0.1203 \pm 0.0101\) \\
& 300
& \(0.1027 \pm 0.0034\)
& \(0.1203 \pm 0.0147\)
& \(0.1079 \pm 0.0090\) \\
\bottomrule
\end{tabular}
\caption{Empirical MISE for the rotational OU model. The entries report
mean \(\pm\) Monte Carlo standard error over \(M=50\) replications.}
\label{tab:mise-rotational-ou}
\end{table}

\begin{table}[!htbp]
\centering
\scriptsize
\setlength{\tabcolsep}{3pt}
\begin{tabular}{llcccc}
\toprule
Censoring mechanism & \(N\) & isoGL & anisoGL & LS & PCO \\
\midrule
Random ball & 50
& \(0.2258 \pm 0.0098\)
& \(0.1768 \pm 0.0089\)
& \(0.1966 \pm 0.0100\)
& \(0.1069 \pm 0.0046\) \\
& 100
& \(0.1662 \pm 0.0203\)
& \(0.1844 \pm 0.0086\)
& \(0.1940 \pm 0.0103\)
& \(0.0703 \pm 0.0062\) \\
& 200
& \(0.1101 \pm 0.0031\)
& \(0.1188 \pm 0.0270\)
& \(0.2058 \pm 0.0093\)
& \(0.0623 \pm 0.0044\) \\
& 300
& \(0.1116 \pm 0.0026\)
& \(0.0392 \pm 0.0009\)
& \(0.2110 \pm 0.0083\)
& \(0.0572 \pm 0.0037\) \\
\addlinespace
Coordinatewise right & 50
& \(0.3411 \pm 0.0182\)
& \(0.2007 \pm 0.0097\)
& \(0.1140 \pm 0.0058\)
& \(0.1550 \pm 0.0099\) \\
& 100
& \(0.2602 \pm 0.0309\)
& \(0.2090 \pm 0.0089\)
& \(0.1093 \pm 0.0064\)
& \(0.1228 \pm 0.0106\) \\
& 200
& \(0.2138 \pm 0.0083\)
& \(0.2213 \pm 0.0074\)
& \(0.0977 \pm 0.0063\)
& \(0.0974 \pm 0.0125\) \\
& 300
& \(0.2150 \pm 0.0063\)
& \(0.2247 \pm 0.0056\)
& \(0.0929 \pm 0.0050\)
& \(0.0860 \pm 0.0090\) \\
\bottomrule
\end{tabular}
\caption{Empirical MISE for the Fixman-like model. The entries report
mean \(\pm\) Monte Carlo standard error over \(M=50\) replications.}
\label{tab:mise-fixman-like}
\end{table}

\smallskip
\noindent\textbf{Additional simulation figures.}
Figures~\ref{fig:linear-ou-comparison} and~\ref{fig:rotational-ou-comparison}
provide the corresponding comparisons for the OU and rotational OU drifts.
Figure~\ref{fig:mise-comparison} reports the empirical MISE curves as a
function of the number of independent trajectories.

\begin{figure}[!htbp]
\centering
\includegraphics[width=0.82\textwidth]{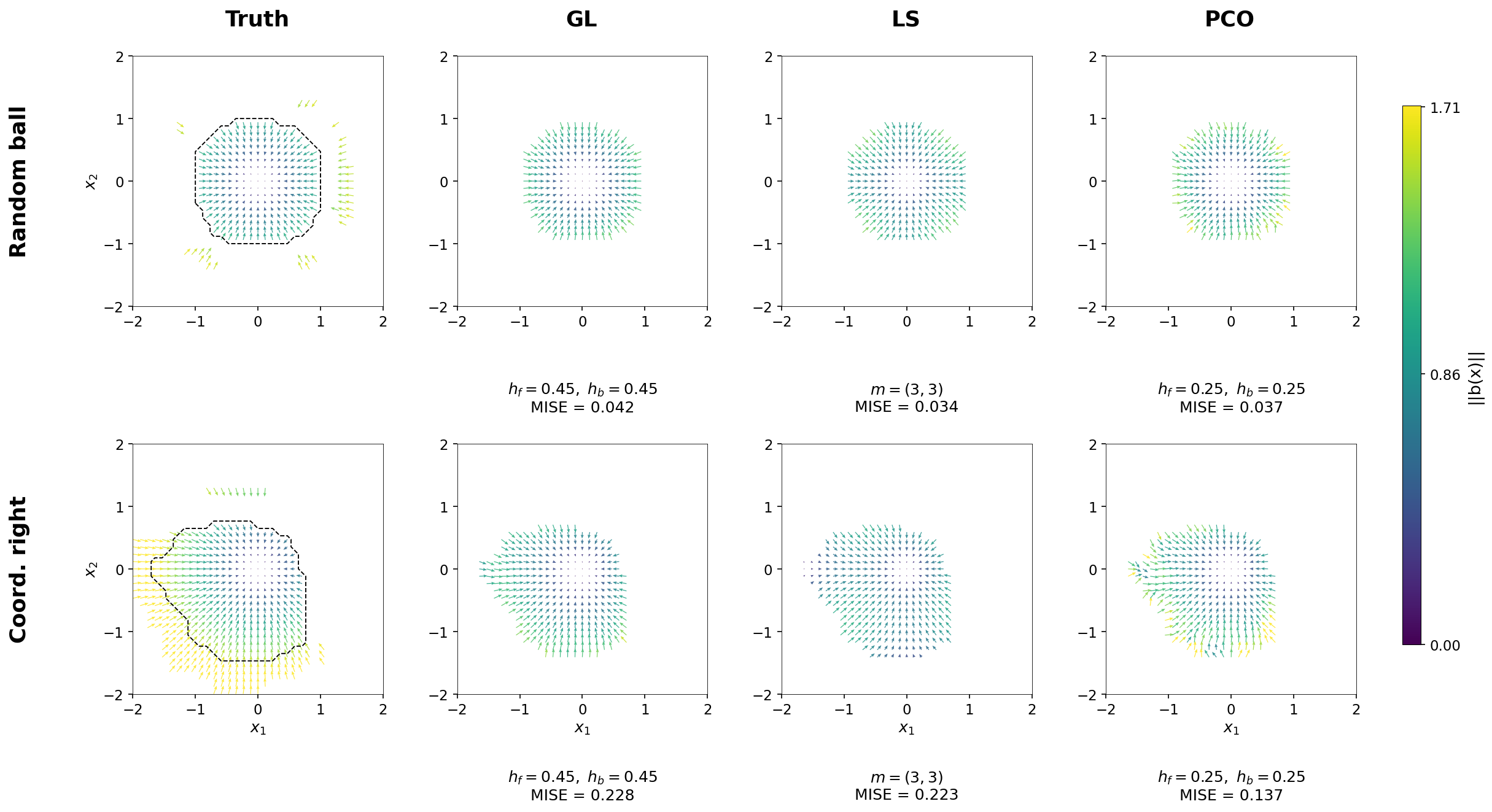}
\caption{Drift reconstruction for the OU model. The first column shows
the true drift field, and the next three columns show the estimates
obtained with \(\widehat b_N^{\mathrm{GL}}\),
\(\widehat b_N^{\mathrm{LS}}\), and
\(\widehat b_N^{\mathrm{PCO}}\). The top row corresponds to random ball
censoring and the bottom row to coordinatewise right censoring.}
\label{fig:linear-ou-comparison}
\end{figure}

\begin{figure}[!htbp]
\centering
\includegraphics[width=0.82\textwidth]{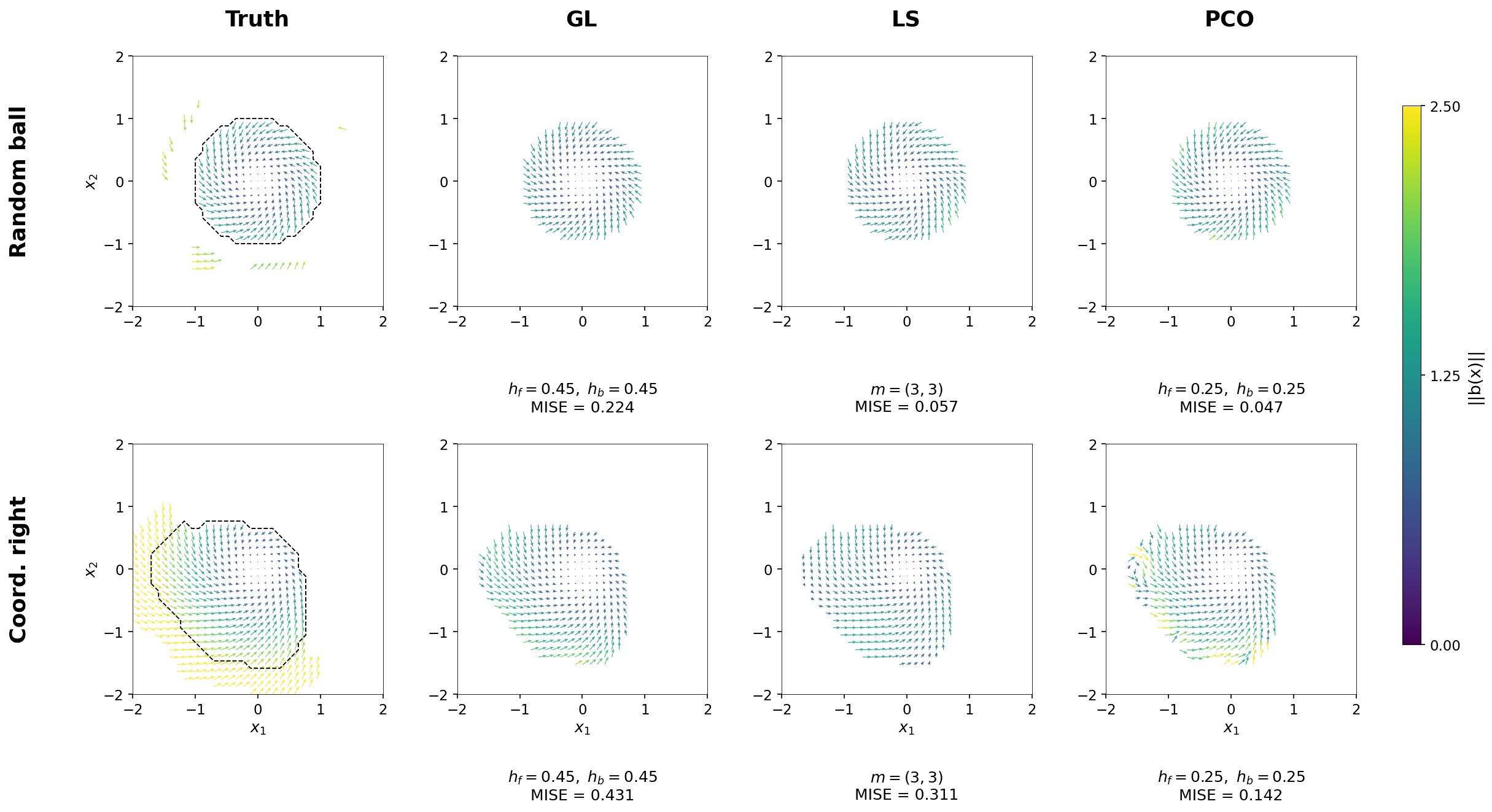}
\caption{Drift reconstruction for the rotational OU model. The first
column shows the true drift field, and the next three columns show the
estimates obtained with \(\widehat b_N^{\mathrm{GL}}\),
\(\widehat b_N^{\mathrm{LS}}\), and
\(\widehat b_N^{\mathrm{PCO}}\). The top row corresponds to random ball
censoring and the bottom row to coordinatewise right censoring.}
\label{fig:rotational-ou-comparison}
\end{figure}

\begin{figure}[!t]
\centering
\includegraphics[width=0.74\textwidth]{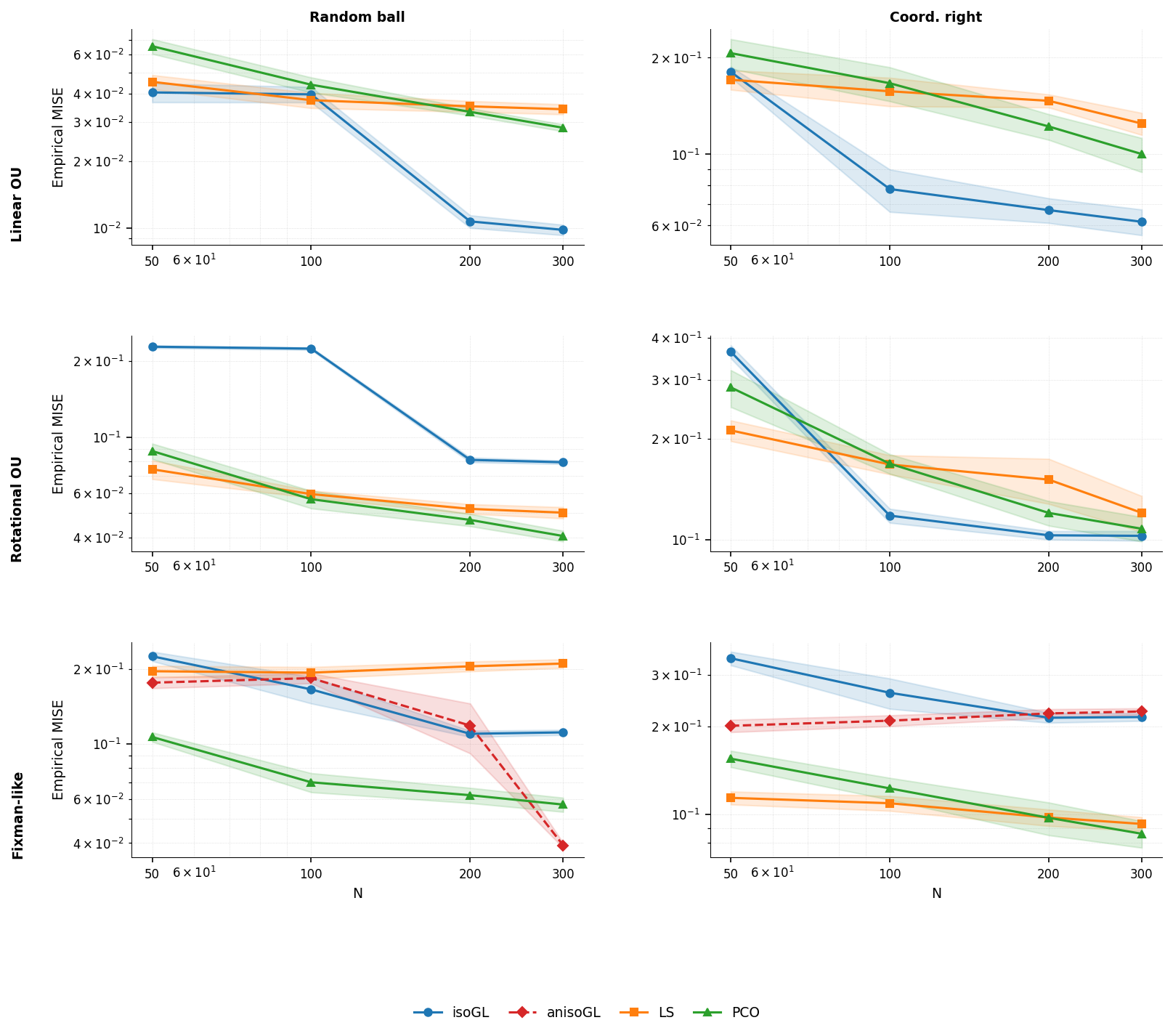}
\caption{Empirical MISE as a function of the number \(N\) of independent
trajectories. The error is evaluated on the effective region
\(\mathcal R\).}
\label{fig:mise-comparison}
\end{figure}

\clearpage

\clearpage
\bibliographystyle{plainnat}
\bibliography{biblio}

\end{document}